\documentclass[11pt]{article}
\usepackage{xinyu,pdfpages}

\usepackage{graphicx}

\renewcommand{\a}{a}
\newcommand{\arcsz}{\scrI_0}
\newcommand{\arcs}{\scrI}
\newcommand{\terms}{\scrT}

\newcommand{\exc}{\operatorname{exc}}
\newcommand{\OurTrace}{\textbf{Trace}}
\newcommand{\ourpoly}{p}
\newcommand{\Xram}{\mathfrak{X}}
\newcommand{\q}{e}
\newcommand{\G}{\scrG}
\newcommand{\lift}{\mathfrak{L}}
\newcommand{\blift}{\boldsymbol{\mathfrak{L}}}

\newcommand{\btlift}{\boldsymbol{\mathfrak{T}}}

\newcommand{\an}{an\xspace}
\newcommand{\MPL}{MPL\xspace} \newcommand{\mpl}{\MPL}
\newcommand{\XOR}{XOR\xspace}
\newcommand{\SDP}{\operatorname{SDP}}
\newcommand{\Eig}{\operatorname{Eig}}
\newcommand{\obj}{\operatorname{obj}}

\title{Explicit near-fully X-Ramanujan graphs\footnote{Or, \textbf{ABCDEFG}  --- \emph{\textbf{\emph{A}}dventures with \textbf{\emph{B}}ordenave--\textbf{\emph{C}}ollins: \textbf{\emph{D}}erandomization and \textbf{\emph{E}}xamples of \textbf{\emph{F}}un \textbf{\emph{G}}raphs}}}

\author{Ryan O'Donnell\thanks{\texttt{odonnell@cs.cmu.edu}, \texttt{xinyuwu@cmu.edu}. Computer Science Department, Carnegie Mellon University.  Author order randomized.  Supported by NSF grant CCF-1717606. This material is based upon work supported by the National Science Foundation under grant numbers listed above. Any opinions, findings and conclusions or recommendations expressed in this material are those of the author and do not necessarily reflect the views of the National Science Foundation (NSF).}\and Xinyu Wu\footnotemark[2]}

\date{\footnotesize{September 5, 2020}}

\begin{document}

\thispagestyle{empty}

\includepdf{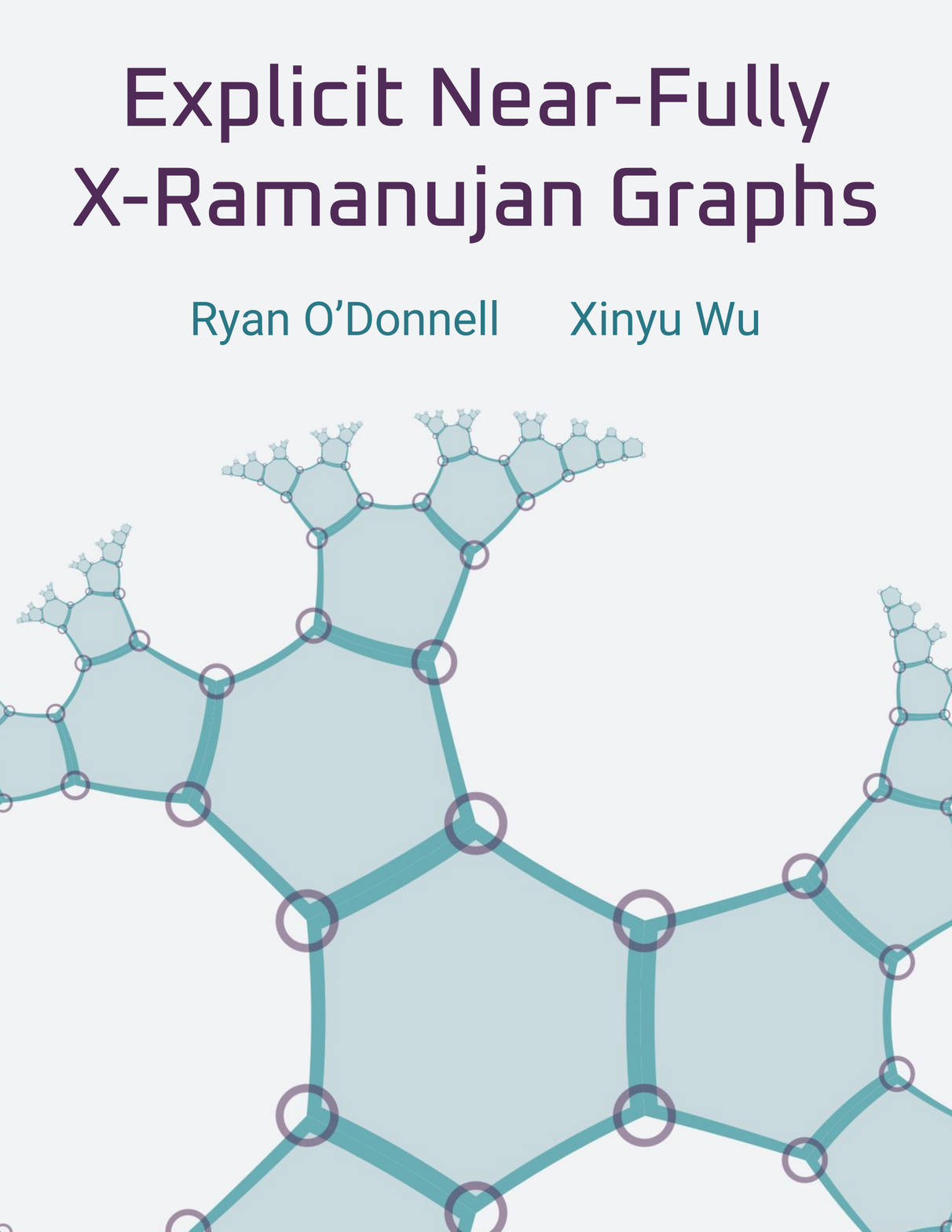}

\maketitle

\begin{abstract}
    Let $\ourpoly(Y_1, \dots, Y_d, Z_1, \dots, Z_\q)$ be a self-adjoint noncommutative polynomial, with coefficients from $\C^{r \times r}$, in the indeterminates $Y_1, \dots, Y_d$ (considered to be self-adjoint), the indeterminates $Z_1, \dots, Z_\q$, and their adjoints $Z_1^\conj, \dots, Z_\q^\conj$.
    Suppose $Y_1, \dots, Y_d$ are replaced by independent random $n \times n$ matching matrices, and $Z_1, \dots, Z_\q$ are replaced by independent random $n \times n$ permutation matrices.
    Assuming for simplicity that $\ourpoly$'s coefficients are $0$-$1$ matrices, the result can be thought of as a kind of random $rn$-vertex graph~$\bG$.
    As $n \to \infty$, there will be a natural limiting infinite graph~$\Xram$ that covers any finite outcome for~$\bG$.
    A recent landmark result of Bordenave and Collins shows that for any $\eps > 0$, with high probability the spectrum of a random~$\bG$ will be $\eps$-close in Hausdorff distance to the spectrum of~$\Xram$ (once the suitably defined ``trivial'' eigenvalues are excluded).
    We say that $\bG$ is ``$\eps$-near fully $\Xram$-Ramanujan''.

    Our work has two contributions: First we study and clarify the class of infinite graphs~$\Xram$ that can arise in this way.
    Second, we derandomize the Bordenave--Collins result: for any~$\Xram$, we provide explicit, arbitrarily large graphs~$G$ that are covered by~$\Xram$ and that have (nontrivial) spectrum at Hausdorff distance at most~$\eps$ from that of~$\Xram$.
    This significantly generalizes the recent work of Mohanty~et~al.,\ which provided explicit near-Ramanujan graphs for every degree~$d$ (meaning $d$-regular graphs with all nontrivial eigenvalues bounded in magnitude by $2\sqrt{d-1} + \eps$).
    To give two simple examples:
    \begin{itemize}
        \item For any $d \geq c \geq 2$ we obtain explicit arbitrarily large $(c,d)$-biregular graphs whose spectrum (excluding $0, \pm \sqrt{cd}$) is $\eps$-close in Hausdorff distance to

            $[-(\sqrt{d-1}+\sqrt{c-1}), -(\sqrt{d-1}-\sqrt{c-1})] \cup [(\sqrt{d-1}-\sqrt{c-1}), (\sqrt{d-1}+\sqrt{c-1})]$.
        \item We obtain explicit arbitrarily large graphs covered by the modular group --- i.e., $3$-regular graphs in which every vertex participates in a triangle --- whose spectrum (excluding~$3$) is $\eps$-close in Hausdorff distance to

            $\{-2\} \cup [\frac{1-\sqrt{13 + 8\sqrt{2}}}{2}, \frac{1-\sqrt{13 - 8\sqrt{2}}}{2}] \cup \{0\} \cup [\frac{1+\sqrt{13 - 8\sqrt{2}}}{2}, \frac{1+\sqrt{13 + 8\sqrt{2}}}{2}]$.
    \end{itemize}

    As an application of our main technical theorem, we are also able to determine the ``eigenvalue relaxation value'' for a wide class of average-case degree-$2$ constraint satisfaction problems.
\end{abstract}

\thispagestyle{empty}

\newpage

\tableofcontents

\setcounter{page}{3}

\section{Introduction}  \label{sec:intro}

Let $G$ be an $n$-vertex, $d$-regular graph.
Its adjacency matrix~$A$ will always have a ``trivial'' eigenvalue of~$d$ corresponding to the eigenvector $\frac1{n}(1, 1, \dots, 1)$, the stationary probability distribution for the standard random walk on~$G$.
Excluding this eigenvalue, a bound on the magnitude~$\lambda$ of the remaining nontrivial eigenvalues can be very useful; for example, $\lambda$ can be used to control the mixing time of the random walk on~$G$~\cite{Mar08}, the maximum cut in~$G$~\cite{Fie73}, and the error in the Expander Mixing Lemma for~$G$~\cite{AC88}.

The Alon--Boppana theorem~\cite{Alo86} gives a lower bound on how small~$\lambda$ can be, namely $2\sqrt{d-1} - o_{n\to\infty}(1)$.
This number $2\sqrt{d-1}$ arises from the spectral radius $\rho(\mathbb{T}_d)$ of the infinite \mbox{$d$-regular} tree~$\mathbb{T}_d$, which is the universal cover for all $d$-regular graphs ($d \geq 3$).
Celebrated work of Lubotzky--Phillips--Sarnak~\cite{LPS88} and Margulis~\cite{Mar88} (see also~\cite{Iha66}) shows that for infinitely many~$d$, there exists an explicit infinite family of $d$-regular graphs satisfying $\lambda \leq 2\sqrt{d-1}$.
Graphs meeting this bound were dubbed \emph{$d$-regular Ramanujan graphs}, and subsequent constructions~\cite{Chi92,Mor94} gave explicit families of $d$-regular Ramanujan graphs whenever $d-1$ is a prime power.
The fact that these graphs are optimal (spectral) expanders, together with the fact that they are \emph{explicit} (constructible deterministically and efficiently), has made them useful in a variety of application areas in computer science, including coding theory~\cite{SS96}, cryptography~\cite{CFLMP18}, and derandomization~\cite{NN93}.

The analysis of LPS/Margulis Ramanujan graphs famously relies on deep results in number theory, and it is still unknown whether infinitely many $d$-regular Ramanujan exist when $d-1$ is not a prime power.
On the other hand, if one is willing to settle for \emph{nearly}-Ramanujan graphs, there is a simple though inexplicit way to construct them for any~$d$ and~$n$:  Friedman's landmark resolution~\cite{Fri08} of Alon's conjecture shows that for any $\eps > 0$, a \emph{random} $n$-vertex $d$-regular graph has $\lambda \leq 2\sqrt{d-1} + \eps$ with high probability (meaning probability $1-o_{n \to \infty}(1)$).
The proof of Friedman's theorem is also very difficult, although it was notably simplified by Bordenave~\cite{Bor19}.
The distinction between Ramanujan and nearly-Ramanujan does not seem to pose any problem for applications, but the lack of explicitness does, particularly (of course) for applications to derandomization.

There are several directions in which Friedman's theorem could conjecturally be generalized.
One major such direction was conjectured by Friedman himself~\cite{Fri03}: that for any fixed base graph~$K$ with universal cover tree~$\Xram$, a random $n$-lift~$\bG$ of~$K$ is nearly ``$\Xram$-Ramanujan'' with high probability.
Here the term ``$\Xram$-Ramanujan'' refers to two properties: first, $\Xram$ \emph{covers}~$\bG$ in the graph theory sense; second, the ``nontrivial'' eigenvalues of~$\bG$, namely those not in $\spec(K)$, are bounded in magnitude by the spectral radius~$\rho(\Xram)$ of~$\Xram$.
The modifier ``nearly'' again refers to relaxing $\rho(\Xram)$ to $\rho(\Xram)+\eps$, here.
(We remark that for bipartite~$K$, Marcus, Spielman, and Srivastava~\cite{MSS15a} showed the existence of an exactly $\Xram$-Ramanujan $n$-lift for every~$n$.)
An even stronger version of this conjecture would hold that $\bG$ is near-\emph{fully} $\Xram$-Ramanujan with high probability; by this we mean that for every $\eps > 0$, the nontrivial spectrum of~$\bG$ is $\eps$-close in Hausdorff distance to the spectrum of~$\Xram$ (i.e., every nontrivial eigenvalue of~$\bG$ is within~$\eps$ of a point in~$\Xram$'s spectrum, and vice versa).

This stronger conjecture --- and in fact much more --- was recently proven by Bordenave and Collins~\cite{BC19}.
Indeed their work implies that for a wide variety of \emph{non-tree} infinite graphs~$\Xram$, there is a random-lift method for generating arbitrarily large finite graphs, covered by~$\Xram$, whose non-trivial spectrum is near-fully $\Xram$-Ramanujan.
However besides universal cover trees, it is not made clear in~\cite{BC19} precisely to which $\Xram$'s their results apply.

Our work has two contributions.
First, we significantly clarify and partially characterize the class of infinite graphs~$\Xram$ for which the Bordenave--Collins result can be used; we term these \emph{\MPL graphs}.
We establish that all free products of finite vertex-transitive graphs~\cite{Zno75} (including Cayley graphs of free products of finite groups), free products of finite rooted graphs~\cite{Que94}, additive products~\cite{MO20}, and amalgamated free products~\cite{VK19}, inter alia, are \MPL graphs --- but also, that \MPL graphs must be unimodular, hyperbolic, and of finite treewidth.
The second contribution of our work is to derandomize the Bordenave--Collins result: for every \MPL graph~$\Xram$ and every $\eps > 0$, we give a $\poly(n)$-time deterministic algorithm that outputs a graph on $n' \sim n$ vertices that is covered by~$\Xram$ and whose nontrivial spectrum is $\eps$-close in Hausdorff distance to that of~$\Xram$.

\subsection{Bordenave and Collins's work}

Rather than diving straight into the statement of Bordenave and Collins's main theorem, we will find it helpful to build up to it in stages.

\paragraph{$\textit{d}$-regular graphs.}
Let us return to the most basic case of random $n$-vertex, $d$-regular graphs.
A natural way to obtain such a graph~$\bcalG_n$ (provided~$n$ is even) is to independently choose~$d$ uniformly random matchings $\bcalM_1, \dots \bcalM_d$ on the same vertex set $V_n = [n] = \{1, 2, \dots, n\}$ and to superimpose them.
It will be important for us to remember which edge in $\bcalG_n$ came from which matching, so let us think $\bcalM_1, \dots, \bcalM_d$ as being \emph{colored} with colors $1, \dots, d$.
Then $\bcalG_n$ may be thought of as a ``color-regular graph''; each vertex is adjacent to a single edge of each color.

Moving to linear algebra, the adjacency matrix~$\bA_n$ for $\bcalG_n$ may be thought of as follows: First, we take the formal polynomial $\ourpoly(Y_1, \dots, Y_d) = Y_1 + \cdots + Y_d$.  Next, we obtain $\bA_n$ by substituting $Y_j = P_{\bsigma_j}$ for each $j \in [d]$, where the $\bsigma_j$'s are independent uniformly random matchings on~$[n]$ (i.e., permutations in $\symm{n}$ with all cycles of length~$2$) and where $P_{\sigma}$ denotes the permutation matrix associated to~$\sigma$.

If we fix a vertex $o \in V_n$ and a number~$\ell \in \N$, with high probability the radius-$\ell$ neighborhood of~$o$ in~$\bcalG_n$ will look like the radius-$\ell$ neighborhood of the root of an infinite $d$-color-regular tree (i.e., the infinite $d$-regular tree in which each vertex is adjacent to one edge of each color).
This tree may be identified with the Cayley graph of the free group $V_\infty = \Z_2 \star Z_2 \star \cdots \Z_2$ with generators $g_1, \dots, g_d$.
These generators act as permutations on $V_\infty$ by left-multiplication.
Indeed, if one writes $P_{g_j}$ for the associated permutation operator on $\ell_2(V_\infty)$, then the adjacency operator for the Cayley graph is $A_\infty = \ourpoly(P_{g_1}, \dots, P_{g_d})$.

Bordenave and Collins's generalization of Friedman's theorem may thus be viewed as follows: for $\ourpoly(Y_1, \dots, Y_d) = Y_1 + \cdots + Y_d$ we have that for any $\eps > 0$, if $\bA_n = \ourpoly(P_{\bsigma_1}, \dots, P_{\bsigma_d})$  and $A_\infty = \ourpoly(P_{g_1}, \dots, P_{g_d})$, then with high probability the ``nontrivial'' spectrum of~$\bA_n$ is $\eps$-close in Hausdorff distance to the spectrum of $A_\infty$.
Here ``nontrivial'' refers to excluding $\ourpoly(1, \dots, 1) = d$.

\paragraph{Weighted color-regular graphs.}
The Bordenave--Collins theorem is more general than this, however.  It also applies to \emph{(edge-)weighted} color-regular graphs.
Let $a_1, \dots, a_d \in \R$ be real weights associated with the~$d$ colors, and consider the more general linear polynomial $\ourpoly(Y_1, \dots, Y_d) = a_1Y_1 + \cdots + a_d Y_d$.
Then $\bA_n = \ourpoly(P_{\bsigma_1}, \dots, P_{\bsigma_d})$ is the (weighted) adjacency matrix of a random ``color-regular'' graph in which each vertex is adjacent to one edge each of colors $1, \dots, d$, with edge-weights $a_1, \dots, a_d$ respectively.
Similarly, $A_\infty = \ourpoly(P_{g_1}, \dots, P_{g_d})$ is the adjacency operator on $\ell_2(V_\infty)$ for the version of the $d$-color-regular infinite tree in which the edges of color~$j$ are weighted by~$a_j$.
Again,  the Bordenave--Collins result implies that for all $\eps > 0$, with high probability the nontrivial spectrum of $\bA_n$ (meaning, when $\ourpoly(1, \dots, 1) = \sum_j a_j$ is excluded) is $\eps$-close in Hausdorff distance to the spectrum of $A_\infty$.

There are several examples where this may be of interest.
The first is non-standard random walks on color-regular graphs; for example, taking $a_1 = 1/2$, $a_2 = 1/3$, $a_3 = 1/6$ models random walks where one always ``takes the red edge with probability~$1/2$, the blue edge with probability~$1/3$, and the green edges with probability~$1/6$''.
Another example is the case of $a_1 = \cdots = a_{d/2} = +1$, $a_{d/2+1} = \cdots = a_{d} = -1$.
Here $\bA_n$ is a $d$-regular random graphs in which each vertex is adjacent to $d/2$ edges of weight~$+1$ and $d/2$ edges of weight~$-1$.
This is a natural model for random $d$-regular instances of the \emph{2XOR} constraint satisfaction problem.
Studying the maximum-magnitude eigenvalue of $\bA_n$ is interesting because it commonly used to efficiently compute an upper bound on the optimal CSP solution (which is $\mathsf{NP}$-hard to find in the worst case); see \Cref{sec:csp} for further discussion.
Conveniently, the ``trivial eigenvalue'' of $\bA_n$ is~$0$, and the spectrum of $A_\infty$ is easily seen to be identical to that of the $d$-regular infinite tree, $[-2\sqrt{d-1}, 2\sqrt{d-1}]$.
Thus this setting is very similar to that of unweighted random $d$-regular graphs, but without the annoyance of the eigenvalue of~$d$.

\paragraph{Self-loops and general permutations.}
The Bordenave--Collins theorem is more general than this, however.
Here are two more modest generalizations it allows for.
First, one can allow ``self-loops'' in our template polynomials.
In other words, one can generalize to polynomials $\ourpoly(Y_1, \dots, Y_d) = a_0 \Id + a_1 Y_1 + \cdots + a_d Y_d$, where $a_0 \in \R$ and $\Id$ can be thought of as a new ``indeterminate'' which is always substituted with the identity operator (both in the finite case of producing~$\bA_n$ and in the infinite case of~$A_\infty$).
Second, in addition to having $Y_i$ indeterminates that are substituted with random $n \times n$ matching matrices, one may also allow new indeterminates that are substituted with uniformly random $n \times n$ general permutation matrices.
One should be careful to create self-adjoint matrices, i.e.\ undirected (weighted) graphs, though.
To this end, Bordenave and Collins consider polynomials of the form
\begin{equation}    \label{eqn:general-linear}
    \ourpoly(Y_1, \dots, Y_d, Z_1, \dots, Z_\q) = a_0 \Id + a_1 Y_1 + \cdots + a_d Y_d + a_{d+1} Z_1 + \cdots + a_{d+\q} Z_\q + a_{d+1}^* Z_1^* + \cdots + a_{d+\q}^* Z_\q^*.
\end{equation}
Here $a_0, \dots, a_d \in \R$, $a_{d+1}, \dots, a_{d+\q} \in \C$, and $Z_1, \dots, Z_\q$ are new indeterminates that in the finite case are always substituted with random $n \times n$ general permutation matrices.
We say that the above \emph{polynomial} is ``self-adjoint'', with the indeterminates $\Id, Y_1, \dots, Y_d$ being treated as self-adjoint.
Note that the finite adjacency matrix~$\bA_n = \ourpoly(P_{\bsigma_1}, \dots, P_{\bsigma_d}, P_{\bsigma_{d+1}}, \dots, P_{\bsigma_{d+\q}})$ that is self-adjoint and hence that represents a (weighted) undirected $n$-vertex graph.
(As a reminder, here $\bsigma_1, \dots, \bsigma_d$ are random matching permutations and $\bsigma_{d+1}, \dots, \bsigma_{d+\q}$ are random general permutations.)
As for the infinite case, we extend the notation $V_\infty$ to denote $\Z_2^{\star d} \star \Z^{\star \q}$, the free product of $d$ copies of~$\Z_2$ and $\q$ copies of~$\Z$.
Then $A_\infty = \ourpoly(P_{g_1}, \dots, P_{g_d}, P_{g_{d+1}}, \dots, P_{g_{d+\q}})$, where $g_{d+1}, \dots, g_{d+\q}$ denote the generators of the~$\Z$ factors (and note that $P_{g_{d+i}}^* = P_{g_{d+i}}^{-1} = P_{g_{d+i}^{-1}}$).

\paragraph{Matrix coefficients.}
Now comes one of the more dramatic generalizations: the Bordenave--Collins result also allows for \emph{matrix edge-weights/coefficients}.
One motivation for this generalization is that it is needed for the ``linearization'' trick, discussed below.
But another motivation is that it allows the theory to apply to non-regular graphs.
The setup now is that for a fixed dimension $r \in \N^+$, we will consider color-regular graphs where each color~$j$ is now associated with an edge-weight that may be a \emph{matrix} $a_j \in \C^{r \times r}$.
The adjacency matrix of an $n$-vertex graph with $r \times r$ matrix edge-weights is, naturally, the $n \times n$ block matrix whose $(u,v)$ block is the $r \times r$ weight matrix for edge $(u,v)$.
(To be careful here, an undirected edge should be thought of as two opposing directed edges; we insist these directed edges get matrix weights that are adjoints of one another, so as to overall preserve self-adjointness.)
In case all edges have the same matrix weight, the resulting adjacency matrix is just the Kronecker product of the original adjacency matrix and the weight.
For example, if $P_\sigma$ is the adjacency matrix of a matching on~$[n]$, and each edge in the matching is assigned the weight
\[
    a = \begin{pmatrix}
               0 & 1 & 1 & 1 \\
               1 & 0 & 1 & 0 \\
               1 & 1 & 0 & 0 \\
               1 & 0 & 0 & 0
          \end{pmatrix}
    \qquad
    \text{(the adjacency matrix of } \triangleright\hspace{-.4em}-\text{\,)},
\]
then the resulting matrix-weighted graph has adjacency matrix $P_\sigma \otimes a$, an operator on $\C^n \otimes \C^r$.

Note that a matrix weighted graph's adjacency operator on $\C^n \otimes \C^r$ can simultaneously be viewed as an operator on~$\C^{nr}$.  In this viewpoint, it is the adjacency matrix of an (uncolored) scalar-weighted $nr$-vertex graph, which we call the \emph{extension} of the underlying matrix-weighted graph.
The situation is particularly simple when the matrix edge-weights are $0$-$1$ matrices; in this case, the extension is an ordinary unweighted graph.
In our above example, $P_\sigma \otimes a$ is the adjacency matrix of $n/2$ disjoint copies of the graph formed from $C_6$ by taking two opposing vertices and hanging a pendant edge on each.
Notice that this is a non-regular graph, even though the original matching is regular.

The Bordenave--Collins theorem shows that for any self-adjoint polynomial as in \Cref{eqn:general-linear}, where the coefficients $a_j$ are from $\C^{r \times r}$, we again have that for all $\eps > 0$, the resulting random adjacency operator $\bA_n$ (on $\C^{nr}$) has its nontrivial spectrum $\eps$-close in Hausdorff distance to that of the operator $A_\infty$ (on $\ell_2(V_\infty \times [r])$).
Here the ``nontrivial spectrum'' refers to the $nr - r$ eigenvalues obtained by removing the eigenvalues of $\ourpoly(1, \dots, 1) = \sum_{j =0}^d a_j + \sum_{j = 1}^\q (a_j + a_j^*)$ from the spectrum of~$\bA_n$.

The most notable application of this result is the generalized Friedman conjecture about the spectrum of random lifts of a base graph~$K = (R,E)$.
That result is obtained by taking $r = |R|$, $\q = |E|$, $d = 0$, $a_0 = 0$, and $\ourpoly = \sum_{(u,v) \in E} a_{uv} Z_{uv}$, where $(u,v)$ denotes a directed edge, $a_{uv}$ is the $r \times r$ matrix that has a single~$1$ in the $(u,v)$ entry ($0$'s elsewhere), and where $Z_{vu}$ denotes $Z_{uv}^*$ when $v > u$.
In this case, the random matrices $\bA_n$ are adjacency matrices of (extension) graphs that are random $n$-lifts of~$K$, and the operator $\bA_\infty$ is the adjacency operator for the universal cover tree~$\Xram$ of~$K$.

\paragraph{Nonlinear polynomials.}
We now come to Bordenave and Collins's other dramatic generalization: the polynomials~$\ourpoly$ that serve as ``recipes'' for producing random finite graphs and their infinite covers need not be linear.
Restricting to $0$-$1$ matrix weights but allowing for nonlinear polynomials leads to a wealth of possible infinite graphs~$\Xram$ (not necessarily trees), which we term \emph{\MPL (matrix polynomial lift) graphs}.
(Note that even if one ultimately only cares about matrix weights in $\{0,1\}^{r \times r}$, the ``linearization'' reduction produces linear polynomials with general matrix weights.)
An example \MPL graph is depicted on our title page; it arises from the polynomial
\begin{equation*}
    \begin{pmatrix}
    0 & 1 & 0 & 0 & 0 & 1\\
    1 & 0 & 1 & 0 & 0 & 0\\
    0 & 1 & 0 & 1 & 0 & 0\\
    0 & 0 & 1 & 0 & 1 & 0\\
    0 & 0 & 0 & 1 & 0 & 1\\
    1 & 0 & 0 & 0 & 1 & 0\\
    \end{pmatrix} \cdot \Id
    +
    a_{14} Z_1 + a_{41} Z_1^* + a_{36} Z_2 + a_{63} Z_2^* + a_{25} Z_2 Z_1 + a_{52} Z_1^* Z_2^*,
\end{equation*}
where again $a_{uv}$ denotes the $6 \times 6$ matrix with a~$1$ in the $(u,v)$ entry.

We may now finally state Bordenave and Collins's main theorem:
\begin{theorem} \label{thm:bc-main-intro}
        Let $\ourpoly$ be a self-adjoint noncommutative polynomial with coefficients from $\C^{r \times r}$ in the self-adjoint indeterminates $\Id, Y_1, \dots, Y_d$ and the indeterminates $Z_1, \dots, Z_\q, Z_1^\conj, \dots, Z_\q^\conj$.
        Then for all $\eps, \beta > 0$ and sufficiently large~$n$, the following holds:

        Let $\bA_n$ be the operator on $\C^{n} \otimes \C^{r}$ obtained by substituting the $n \times n$ identity matrix for~$\Id$, independent random $n \times n$ matching matrices for $Y_1, \dots, Y_d$, and independent random $n \times n$ permutation matrices for $Z_1, \dots, Z_\q$.
        Write $\bA_{n,\bot}$ for the restriction of $\bA_n$ to the codimension-$r$ subspace orthogonal to $\mathrm{span}\{(1, \dots, 1)\} \otimes \C^r$.
        Then except with probability at most~$\beta$, the spectra $\sigma(\bA_{n,\bot})$ and $\sigma(A_\infty)$ are at Hausdorff distance at most~$\eps$.

        Here $A_\infty$ is the operator acting $\ell_2(V_\infty) \otimes \C^r$, where $V_\infty = \Z_2^{\star d} \star \Z^{\star \q}$ is the free product of $d$~copies of the group~$\Z_2$ and $\q$~copies of the group~$\Z$, obtained by substituting for $Y_1, \dots, Y_d$ and $Z_1, \dots, Z_\q$ the left-regular representations of the generators of~$V_\infty$.
\end{theorem}

As discussed further in \Cref{sec:our-work}, our work derandomizes this theorem by providing explicit (deterministically $\poly(n)$-time computable) $n \times n$ permutation matrices $P_{\sigma_1}, \dots, P_{\sigma_d}$ (matchings), $P_{\sigma_{d+1}}, \dots, P_{\sigma_{d+\q}}$ (general), for which the conclusion holds.
In fact, our result has the stronger property that for fixed constants $d, \q, r, k, R$, and~$\eps$, we construct in deterministic $\poly(n)$ time  $P_{\sigma_1}, \dots, P_{\sigma_{d+\q}}$ that have the desired $\eps$-Hausdorff closeness \emph{simultaneously for all polynomials~$\ourpoly$} (with degree bounded by~$k$ and coefficient matrices bounded in norm by~$R$).
A very simple but amusing consequence of this is that for every constant~$D \in \N^+$ and $\eps > 0$ we get explicit $n$-vertex matchings $\calM_1, \dots, \calM_{D}$ such that $\calM_{1} + \calM_{2} + \cdots + \calM_{d}$ is $\eps$-nearly $d$-regular Ramanujan for each $d \leq D$.

\subsection{X-Ramanujan graphs}

We would like to now rephrase the Bordenave--Collins result in terms of a new definition of ``$X$-Ramanujan'' graphs.
Over the years, a number of works have raised the question of how to generalize the classic notion of a $d$-regular Ramanujan graph to the case of non-regular graphs~$G$; see, e.g., \cite[Sec.~2.2]{MO20} for an extended discussion.
A natural first possibility is simply to compare the ``nontrivial'' spectral of~$G$ to that of its universal cover tree~$\Xram$.
However this idea is rather limited in scope, particularly because it only pertains to locally tree-like graphs.
To illustrate the deficiency, consider the sorts of graph $G$ arising from average-case analysis of constraint satisfaction problems.
For example, random regular instances of the 3Sat or NAE-3Sat problem lead one to study graphs~$G$ composed of triangles, arranged such that every vertex participates in a fixed number of triangles --- say,~$4$, for a concrete example.
Such graphs are $8$-regular, so in analyzing $G$'s second-largest eigenvalue one might be tempted to compare it to the Alon--Boppana bound~$2\sqrt{7}$, inherited from the $8$-regular infinite tree.
However as the below theorem of Grigorchuk and \.{Z}uk shows, the graph~$G$ will in fact always have second-largest eigenvalue at least~$1+2\sqrt{6} -o(1) > 2\sqrt{7}$.
The reason is that the finite triangle graph~$G$ is \emph{covered} (in the graph-theoretic sense\footnote{For the complete definition of graph covering, see \Cref{def:cover}.}) by the (non-tree) infinite free product graph $\Xram = C_3 \star C_3 \star C_3 \star C_3$, which is known to have spectral radius $\rho(\Xram) = 1+2\sqrt{6}$.
\begin{theorem}                                     \label{thm:gz}
    (\cite{GZ99}'s generalization of the Alon--Boppana bound.)
    Let $\Xram$ be an infinite graph and let $\eps > 0$.  Then there exists $c > 0$ such that any $n$-vertex graph~$G$ covered by $\Xram$ has at least $cn$ eigenvalues at least $\rho(\Xram)-\eps$.  (In particular, for large enough~$n$ the second-largest eigenvalue of~$G$ is at least $\rho(\Xram)-\eps$.)
\end{theorem}

In light of this, and following~\cite{GZ99,Cla07,MO20}, we instead take the perspective that the property of ``Ramanujan-ness'' should derive from the nature of the \emph{infinite graph}~$\Xram$, rather than that of the finite graph~$G$:
\begin{definition}[X-Ramanujan, slightly informal\footnote{Both of the two bullet points in this definition require a caveat: (i)~does ``covering'' allow for disconnected~$G$? (ii)~what exactly counts as a ``nontrivial eigenvalue''?  These points are addressed at the end of this section.}] \label{def:X-ram}
    Given an infinite graph~$\Xram$, we say that finite graph~$G$ is \emph{$\Xram$-Ramanujan} if:
    \begin{itemize}
        \item $\Xram$ covers~$G$;
        \item the ``nontrivial eigenvalues'' of~$G$ are bounded in magnitude by $\rho(\Xram)$.
    \end{itemize}
    If the bound is relaxed to $\rho(\Xram)+\eps$, we say that $G$ is \emph{$\epsilon$-nearly} $\Xram$-Ramanujan.
\end{definition}
Thus the classic definition of $G$ being a ``$d$-regular Ramanujan graph'' is equivalent to being  $\mathbb{T}_d$-Ramanujan for $\mathbb{T}_d$ the infinite $d$-regular tree.
It was shown in~\cite{MO20} (via non-explicit methods) that for a fairly wide variety of~$\Xram$, infinitely many $\Xram$-Ramanujan graphs exist.
This wide variety includes all free products of Cayley graphs, and all ``additive products'' (see \Cref{def:additive}).
Friedman's generalized conjecture (proven by Bordenave--Collins) holds that whenever $\Xram$ is the universal cover tree of a base graph~$K$, random lifts of $K$ are $\eps$-nearly $\Xram$ with high probability (for any fixed~$\eps > 0$).

With this perspective in hand, one can be much more ambitious.
Take the earlier example of graphs~$G$ where every vertex participates in~$4$ triangles; i.e., graphs covered by $\Xram = C_3 \star C_3 \star C_3 \star C_3$, which is known to have spectrum $\sigma(\Xram) = [1-2\sqrt{6}, 1+2\sqrt{6}]$.
The above definition of $\Xram$-Ramanujan asks for $G$'s nontrivial eigenvalues to be upper-bounded in magnitude by $1+2\sqrt{6}$.
But it seems natural to ask if $G$ can also have these eigenvalues bounded \emph{below} by~$1-2\sqrt{6}$.
As another example, it well known that the spectrum of the $(c,d)$-biregular infinite tree~$\mathbb{T}_{c,d}$ with $d > c$ is $[-(\sqrt{d-1}+\sqrt{c-1}), -(\sqrt{d-1}-\sqrt{c-1})] \cup \{0\} \cup  [(\sqrt{d-1}-\sqrt{c-1}), (\sqrt{d-1}+\sqrt{c-1})]$, which (when $0$ is excluded) contains a notable \emph{gap} between $\pm (\sqrt{d-1}-\sqrt{c-1})$.
Are there infinitely many $(c,d)$-biregular finite graphs~$G$ with nontrivial spectrum inside these two intervals?
(As far as we are aware, the answer to this question is unknown, but if an $\eps$-tolerance is allowed, the Bordenave--Collins theorem gives a positive answer.)
Even further we might ask for $(c,d)$-biregular $G$'s whose nontrivial spectrum does not have any \emph{other} gaps besides the one between $\pm (\sqrt{d-1}-\sqrt{c-1})$.
Of course, since $G$'s spectrum is a finite set this is not strictly possible, but we might ask for it to hold up to an~$\epsilon$.
Taking these questions to their limit leads to following definition:

\begin{definition}[Near-fully X-Ramanujan, slightly informal]
    Given an infinite graph~$\Xram$ and $\eps > 0$, we say that finite graph~$G$ is \emph{$\eps$-near fully $\Xram$-Ramanujan} if:
    \begin{itemize}
        \item $\Xram$ covers~$G$;
        \item every nontrivial eigenvalue of~$G$ is within~$\epsilon$ of a point $\Xram$'s spectrum and vice versa --- i.e., the Hausdorff distance of $G$'s nontrivial spectrum to that of~$\Xram$ is at most~$\eps$.
    \end{itemize}
\end{definition}

\begin{example} \label{eg:sarnak}
    It is known that for every $\eps > 0$, a sufficiently large $d$-regular LPS/Ramanujan graph~$G$ is $\eps$-near fully $\mathbb{T}_d$-Ramanujan; i.e., every point in $[-2\sqrt{d-1}, 2\sqrt{d-1}]$ is within distance~$\eps$ of~$\spec(G)$.\footnote{This fact has been attributed to Serre.~\cite{Sar20}}
    Rather remarkably, is has also been shown this is true of \emph{every} sufficiently large $d$-regular Ramanujan graph~\cite{AGV16}.
\end{example}

Bordenave and Collins's \Cref{thm:bc-main-intro} precisely implies that for any ``\mpl graph'' --- i.e., any infinite graph~$\Xram$ arising from the ``infinite lift'' of a $\{0,1\}^r$-edge-weighted polynomial --- we can use finite random lifts to (inexplicitly) produce arbitrarily large $\eps$-near fully $\Xram$-Ramanujan graphs.
Our work, describes in the next section, makes this construction explicit, and significantly characterizes \emph{which} graphs are \mpl.

\paragraph{Technical definitional matters: ``nontrivial'' spectrum and connectedness.}
As soon as we move away from $d$-regular graphs, it's no longer particularly clear what the correct definition of ``nontrivial spectrum'' should be.
For example, in Friedman's generalized conjecture about random lifts~$G$ of a base graph~$K$ with universal covering tree~$\Xram$, then ``nontrivial spectrum'' of~$G$ is taken to be $\spec(G) \setminus \spec(K)$.
But note that this definition is not a function just of~$\Xram$, since many different base graphs~$K$ can have~$\Xram$ as their universal cover tree.
Rather, it depends on the ``recipe'' by which~$\Xram$ is realized, namely as the ``infinite lift'' of~$K$.
Taking this as our guide, we will pragmatically define ``nontrivial spectrum'' only in the context of a specific matrix polynomial~$\ourpoly$ whose infinite lift generates~$\Xram$; as in Bordenave--Collins's \Cref{thm:bc-main-intro}, the trivial spectrum is precisely~$\spec(\ourpoly(1, \dots, 1))$, the spectrum of the ``$1$-lift of~$\ourpoly$''.

We also need to add a word about connectedness in the context of graph covering.
Traditionally, to say that ``$\Xram$ covers $G$'' one requires that both $\Xram$ and~$G$ be connected.
In the context of the classic ``$d$-regular Ramanujan graph'' definition, there are no difficulties because $d \leq 2$ is typically excluded; note that for $d = 1$ or~$2$ we have the random $d$-regular graphs are surely or almost surely disconnected.
However when we move away from trees it does not seem to be a good idea to insist on connectedness.
For one, there are many \mpl graphs consist of multiple disjoint copies of some infinite graph~$\Xram$; it seems best to admit $\Xram$ as an \mpl graph in this case.
For two, it's a remarkably delicate question as to when (the extension of) a random $n$-lift of a matrix polynomial is connected.
Fortunately, in most cases $\Xram$ is ``non-amenable'' and this implies that the infinite explicit families of near fully $\Xram$-Ramanujan graphs we produce are connected; see \Cref{sec:connected} for discussions.
Nevertheless, for convenience in this work we will make say (see \Cref{def:cover}) that ``$\Xram$ covers $G$'' provided each connected component of~$G$ is covered by some connected component of~$\Xram$.

\subsection{Our results, and comparison with prior work}    \label{sec:our-work}

The first part of our paper is devoted understanding the class of ``\mpl graphs''.
Recall these are defined as follows (cf.~\Cref{def:mpl}):
Suppose the Bordenave--Collins \Cref{thm:bc-main-intro} is applied with matrix coefficients $a_j \in \{0,1\}^{r \times r}$.
Then the resulting operator $A_\infty$ on $\ell_2(V_\infty) \otimes \C^r$ can be viewed as the adjacency operator of an infinite graph on vertex set $V_\infty \times [r]$.
We say that $\Xram$ is an \mpl graph if (one or more disjoint isomorphic copies of) $\Xram$ can be realized in this way.

Our main results concerning \mpl graphs are as follows:
\begin{itemize}
    \item (\Cref{sec:graph-examples}.) All free products of Cayley graphs of finite groups are \mpl graphs. More generally, all ``additive products'' (as defined in~\cite{MO20}) are \mpl graphs.
    Additionally, all ``amalgamated free products'' (as defined in~\cite{VK19}) are \mpl graphs.
    (Furthermore, there are still more \mpl graphs that do not appear to fit either category; e.g., the graph depicted on the title page.)
    \item (\Cref{eg:zig-zag}) Some zig-zag products and replacement products (as defined in~\cite{RVW02}) of finite graphs may be viewed as lifts of matrix-coefficient noncommutative polynomials.
    \item (\Cref{prop:finite-treewidth}.) All \mpl graphs have finite treewidth. (So, e.g., an infinite grid is not an \mpl graph.)
    \item (\Cref{prop:hyperbolic}.)  Each connected component of \an \mpl graph is hyperbolic. (So, e.g., an \mpl graph's simple cycles are of bounded length.)
    \item (\Cref{prop:unimodular}.) All \mpl graphs are unimodular. (So, e.g., Trofimov~\cite{Tro85}'s ``grandparent graph'' is not an \mpl graph.)
    \item (\Cref{prop:connectivity}.) Given an \mpl graph (by its generating polynomial), as well as two vertices, it is efficiently decidable whether or not these vertices are connected.
\end{itemize}

The remainder of our paper is devoted to derandomizing the Bordenave--Collins \Cref{thm:bc-main-intro}; i.e., obtaining \emph{explicit} (deterministically polynomial-time computable) arbitrarily large $\eps$-near fully $\Xram$-Ramanujan graphs.
Thanks to the linearization trick utilized in~\cite{BC19}, it eventually suffice to derandomize \Cref{thm:bc-main-intro} in the case of \emph{linear} polynomials with matrix coefficients.
This means that one is effectively seeking $\eps$-near fully $\Xram$-Ramanujan graphs for~$\Xram$ being a (matrix-weighted) color-regular infinite tree.

Our technique is directly inspired by the recent work of Mohanty et al.~\cite{MOP20b}, which obtained an analogous derandomization of Friedman's theorem, based on Bordenave's proof~\cite{Bor19}.
Although the underlying idea (dating further back to~\cite{BL06}) is the same, the technical details are significantly more complex, in the same way that~\cite{BC19} is significantly more complex than~\cite{Bor19}.  (See the discussion toward the end of \cite[Sec.~4.1]{BC19} for more on this comparison.)
Some distinctions include the fact that the edge-weights no longer commute, one needs the spectral radius of the nonbacktracking operator to directly arise in the trace method calculations (as opposed to its square-root arising as a proxy for the graph growth rate), and one needs to simultaneously handle a net of all possible matrix edge-weights.

Similar to~\cite{MOP20b}, our key technical theorem \Cref{thm:main} concerns random edge-signings (essentially equivalent to random $2$-lifts) of sufficiently ``bicycle-free'' color-regular graphs.
Here bicycle-freeness (also referred to as ``tangle-freeness'') refers to the following:
\begin{definition}[Bicycle-free]
    An undirected multigraph is said to be \emph{$\lambda$-bicycle free} provided that the distance-$\lambda$ neighborhood of every vertex has at most one cycle.

    We also use this terminology for an ``$n$-lift' --- i.e., a sequence of permutations $\sigma_1, \dots, \sigma_{d}$ (matchings), $\sigma_{d+1}, \dots, \sigma_{d+\q}$ (general permutations) on $V_n = [n]$ --- when the multigraph with adjacency matrix $P_{\sigma_1} + \cdots + P_{\sigma_d} + P_{\sigma_{d+1}} + P_{\sigma_{d+1}}^\conj + \cdots + P_{\sigma_{d+\q}} + P_{\sigma_{d+\q}}^\conj$ is $\lambda$-bicycle free.
\end{definition}

Let us state our key technical theorem in an informal way (for the full statement, see \Cref{thm:main}):
\begin{theorem}                                     \label{thm:main-informal}
    (Informal statement.)
    Let $G$ be an $n$-vertex color-regular graph with matrix weights $a_1, \dots, a_d, a_{d+1}, \dots, a_{d+\q}, a_{d+1}^\conj, \dots, a_{d+\q}^\conj \in \C^{r \times r}$, and assume $G$ is $\lambda$-bicycle free for $\lambda \gg (\log \log n)^2$.
    Consider a uniformly random edge-signing of~$G$, and let $\bB_n$ denote the nonbacktracking operator of the result.
    Then for any $\eps > 0$, with high probability we have $\rho(\bB_n) \leq \rho(B_\infty) + \eps$, where $B_\infty$ denotes the nonbacktracking operator of the color-regular infinite tree with matrix weights $a_1, \dots, a_{d+\q}^\conj$.
\end{theorem}

As in~\cite{MOP20b}, although our proof of \Cref{thm:main-informal} is similar to, and inspired by, the proof of the key technical theorem of Bordenave--Collins~\cite[Thm.~17]{BC19}, it does not follow from it in a black-box way; we needed to fashion our own variant of it.
Incidentally, this (non-derandomized) theorem on random edge-signings is also needed for our applications to random CSPs; see \Cref{sec:csp}.

After proving \Cref{thm:main-informal}, and carefully upgrading it so that the conclusion holds simultaneously for \emph{all} weight sets~$(a_j)$ (of bounded norm), the overall derandomization task \emph{is} a straightforward recapitulation of the method from~\cite{BL06,MOP20b}.
Namely, we first run through the proof of the Bordenave--Collins theorem to establish that $O(\log n)$-wise uniform random permutations are sufficient to derandomize it.
Constructing these requires $n^{O(\log n)}$ deterministic time, but we apply them with $n = N_0 = 2^{\Theta(\sqrt{\log N})}$, where~$N$ is (roughly) the size of the final graph we wish to construct.
Thus in $N_0^{O(\log N_0)} = \poly(N)$ deterministic time we obtain a ``good'' $N_0$-lift.
We also show that this derandomized $N_0$-lift will preserve the property of a truly random lift, that its associated graph is (with high probability) $\lambda$-bicycle free for $\lambda = \Theta(\log N_0) = \Theta(\sqrt{\log N}) \gg (\log \log N)^2$.
Next we show that $1/\poly(N)$-almost $O(\log N)$-wise uniform random \emph{bit-strings} are sufficient to derandomize \Cref{thm:main-informal}, and recall that these can be constructed in $\poly(N)$ deterministic time.  It then remains to repeatedly apply this $2$-lifts arising from the derandomized \Cref{thm:main-informal} to obtain an explicit ``good'' $N'$-lift, for $N' \sim N_0$.
We remark that, as in~\cite{BL06,MOP20b}, this final $N'$-lift is not ``strongly explicit', although it does have the intermediate property of being ``probabilistically strongly explicit'' (see \Cref{sec:good-lifts} for details).
In the end we obtain the following theorem (informally stated; see \Cref{thm:weakly-explicit-main} for the full statement):
\begin{theorem}                                     \label{thm:informal-weakly-explicit-main}
    (Informal statement.)
    For fixed constants $d, \q, r, R$, and $\eps > 0$, there is a deterministic algorithm that, on input~$n$, runs in $\poly(n)$ time and outputs an $n'$-vertex unweighted color-regular graph ($n' \sim n$) such that the following holds:  For all ways of choosing edge-weights $a_1, \dots, a_{d+\q}^* \in \C^{r \times r}$ for the colors with Frobenius-norm bounds $\|a_j\|_\frob, \|a_j^{-1}\|_\frob \leq R$, the resulting color-regular graph's nonbacktracking operator $B_{N'}$ has its nontrivial eigenvalues bounded in magnitude by $\rho(B_\infty) + \eps$, where $B_\infty$ denotes the nonbacktracking operator for the analogously weighted color-regular infinite tree.
\end{theorem}

At this point, it would seem that we are essentially done, and we need only apply the (non-random) results from~\cite{BC19} that let them go from nonbacktracking operator spectral radius bounds for linear polynomials (as in \Cref{thm:informal-weakly-explicit-main}) to adjacency operator Hausdorff-closeness for general polynomials (as in \Cref{thm:bc-main-intro}), taking a little care to make sure the parameters in these reductions only depend on $d, \q, r, R, \eps$, and~$k$ (the degree of the polynomial), and not on the polynomial coefficients~$a_j$ themselves.
The tools needed for these reductions include: (i)~a version of the Ihara--Bass formula for matrix-weighted (possibly infinite) color-regular graphs, to pass from nonbacktracking operators to adjacency operators (\cite[Prop.~9, Prop.~10]{BC19}); (ii)~a reduction from bounding the spectral radius of nonbacktracking operators to obtaining Hausdorff closeness for linear polynomials (\cite[Thm.~12, relying on Prop.~10]{BC19}); (iii)~a way to ensure that this reduction does not blow up the norm of the coefficients~$a_j$ involved; (iv)~the linearization trick to reduce Hausdorff closeness for general polynomials to that for linear polynomials.
Unfortunately, and not to put too fine a point on it, there are bugs in the proof of each of (i),~(ii),~(iii) in~\cite{BC19}.
Correcting these is why the remainder of our paper (\Cref{sec:ihara-bass,sec:final}) still requires new material.\footnote{After consultation with the authors, we are hopeful they will soon be able to published amended proofs. The bug in~(i) is not too serious with multiple ways to fix it.  The bug in~(ii) is fixed satisfactorily in our work, and the authors of~\cite{BC19} outlined to alternate fix involving generalizing \cite[Thm.~17]{BC19} to non-self-adjoint polynomials.  The bug in~(iii) is perhaps the most serious, but the authors may have an alternative patch in mind.}
In brief, the bug in~(i) involves a missing case for the spectrum of non-self-adjoint operators~$B$; the bug in~(ii) arises because their reduction converts self-adjoint linear polynomials to non-self-adjoint ones, to which their Theorem~17 does not apply.
We fill in the former gap, and derive an alternative reduction for~(ii) preserving self-adjointness.
Finally, the bug in~(iii) seems to require more serious changes.
We patched it by first establishing only norm bounds for linear polynomials, and then appealing to an alternative version of Anderson's linearization~\cite{And13}, namely Pisier's linearization~\cite{Pis18}, the quantitative ineffectiveness of which required some additional work on our part.

\subsection{Implications for degree-2 constraint satisfaction problems}
\label{sec:csp}
In this section we discuss applications of our main technical theorem on random edge-signings, \Cref{thm:main-informal}, to the study of constraint satisfaction problems (CSPs).
In fact, putting this theorem together with the finished Bordenave--Collins theorem implies the following variant (cf.~\cite[Thm.~1.15]{MOP20b}), whose full statement appears as \Cref{thm:signed-lift-probabilistic}:

\begin{theorem} \label{thm:bc-with-edge-signing}
    (Informal statement.)
    In the setting of \Cref{thm:bc-main-intro}, if $\bA_n$ is the operator produced by substituting random \emph{$\pm 1$-signed} permutations into the matrix polynomial~$\ourpoly$, then the Hausdorff distance conclusion holds for the full spectrum~$\sigma(\bA_n)$ vis-a-vis~$\sigma(A_\infty)$; i.e., we do not have to remove any ``trivial eigenvalues'' from~$\bA_n$.
\end{theorem}

This theorem will allow us to determine the ``eigenvalue relaxation value'' for a wide class of average-case CSPs.
Roughly speaking, we consider random regular instances of Boolean valued CSPs where the constraints are expressible as degree-$2$ polynomials (with no linear term).
Our work determines the typical eigenvalue relaxation bound for these CSPs; recall that this is a natural, efficiently-computable upper bound on the optimum value of Boolean quadratic programs (and on the SDP/quantum relaxation value).
This generalizes previous work~\cite{MS16,DMO+19,MOP20} on random Max-Cut, NAE3-Sat, and ``$2$-eigenvalue $2$\XOR-like CSPs'', respectively.
We remark again that our results here do not require the derandomization aspect of our work, but they \emph{do} rely on \Cref{thm:bc-with-edge-signing} concerning random \emph{signed} lifts, which is not derivable in a  black-box fashion from the work of Bordenave--Collins.

We will be concerned throughout this section with Boolean CSPs: optimization problems over a Boolean domain, which we take to be $\{\pm 1\}$ (equivalent to $\{0,1\}$ or $\{\textsf{True},\textsf{False}\}$).
The hallmark of a CSP is that it is defined by a collection of \emph{local} constraints of similar type.
Our work is also general enough to handle certain \emph{valued CSPs}, meaning ones where the constraints are not simply predicates (which are satisfied/unsatisfied) but are real-valued functions.
These may be thought of giving as ``score'' for each assignment to the variables in the constraint's scope.

\begin{definition}[Degree-$2$ valued CSPs]\label{def:CSP}
    A \emph{Boolean valued CSP} is defined by a set $\Psi$ of \emph{constraint types}~$\psi$.
    Each $\psi$ is a function $\psi : \{\pm 1\}^r \to \R$, where $r$ is the \emph{arity}.
    We will say such a CSP is \emph{degree-$2$} if each~$\psi$ can be represented as a degree-$2$ polynomial, with no linear terms, in its inputs.

    An \emph{instance} $\calI$ of such a CSP is defined by a set of~$n$ variables~$V$ and a list of~$m$ constraints $C = (\psi,S)$, where each $\psi \in \Psi$ and each $S$ is an $r$-tuple of distinct variables from~$V$.
    More generally, in an instance with \emph{literals allowed}, each constraint is of the form $C = (\psi, S, \ell)$, where $\ell \in \{\pm 1\}^r$.

    The computational task associated with $\calI$ is to determine the value of the optimal \emph{assignment} $x : V \to \{\pm 1\}$; i.e.,
    \begin{equation}    \label{eqn:CSP-opt}
        \Opt(\calI) = \max_{x : V \to \{\pm 1\}} \obj(x), \qquad \text{where } \obj(x) = \sum_{C = (\psi,S,\ell) \in \calI} \psi(\ell_1 x(S_1), \dots, \ell_r x(S_r)).
    \end{equation}
    Note that for a degree-$2$ CSP, the objective $\obj(x)$ may be considered as a degree-$2$ homogeneous polynomial (plus a constant term) over the Boolean cube $\{\pm 1\}^n$.
\end{definition}

\begin{example}
    Let us give several examples where $\Psi$ contains just a single constraint~$\psi$.

    When $r = 2$ and $\psi(x_1,x_2) = \frac12 - \frac12 x_1 x_2$ we obtain the Max-Cut CSP.
    If we furthermore allow literals here, we obtain the 2XOR CSP.
    If literals are disallowed, and $\psi$ is changed to $\psi(x_1,x_2,x_3,x_4) = \frac12 - \frac12 x_1 x_2 + \frac12 - \frac12 x_3 x_4$, we get a version of the 2XOR CSP in which instances must have an equal number of ``equality'' and ``unequality'' constraints.

    When $r = 3$ and $\psi(x_1,x_2,x_3) = \frac34 - \frac14(x_1x_2 + x_1x_3 + x_2x_3)$, the CSP becomes $2$-coloring a $3$-uniform hypergraph; if literals are allowed here, the CSP is known as NAE-3Sat.

    The case of the predicate $\psi(x_1,x_2,x_3,x_4) = \frac12 + \frac14(x_1x_2 + x_2 x_3 + x_3 x_4 - x_1 x_4)$ with literals allowed yields the Sort${}_4$ CSP; the predicate here is equivalent to ``CHSH game'' from quantum mechanics~\cite{CHSH69}.
\end{example}

\begin{remark}
    As additive constants are irrelevant for the task of optimization, we will henceforth assume without loss of generality that degree-$2$ CSPs involve \emph{homogeneous} degree-$2$ polynomial constraints.
\end{remark}

\begin{definition}[Instance graph]\label{def:CSP-instance-graph}
    Given an instance $\calI$ of a degree-$2$ CSP as above, we may associate an \emph{instance graph}~$G$, with adjacency matrix~$A$.
    This $G$ is the undirected, edge-weighted graph with vertex set~$V = [n]$ where, for each nonzero monomial $w_{ij} x_i x_j$ appearing in $\obj(x)$ from \Cref{eqn:CSP-opt}, $G$ contains edge $\{i,j\}$ with weight $\frac12 w_{ij}$.
    The factor of $\frac12$ is included so that given an assignment $x \in \{\pm 1\}^n$, we have $\obj(x) = x^\top A x = \langle A, x x^\top \rangle$, where $\langle A, B \rangle$ denotes the matrix inner product.
\end{definition}

For almost all Boolean CSPs, exactly solving the quadratic program to determine $\Opt$ is $\NP$-hard; hence computationally tractable relaxations of the problem are interesting.
Perhaps the simplest and most natural such relaxation is the \emph{eigenvalue bound} --- the generalization of the Fiedler bound~\cite{Fie73} for Max-Cut:
\begin{definition}[Eigenvalue bound]\label{def:eigenvalue-bound}
    Let $\calI$ be a degree-$2$ CSP instance and let $A$ be the adjacency matrix of its instance graph.  The \emph{eigenvalue bound} is
    \[
        \Eig(\calI) =  \max_{\substack{x \in \ell_2(V) \\ \norm{x}_2^2 = n}} \langle A, xx^\top \rangle
        =  \max_{\substack{X \in \R^{n \times n} \text{ PSD} \\ \tr(X) = n}}  \langle A, X \rangle =  n \cdot \lambda_{\max}(A).
    \]
\end{definition}
It is clear that
\[
    \Opt(\calI) \leq \Eig(\calI)
\]
always, and $\Eig(\calI)$ can be computed (to arbitrary precision) in polynomial time.
For many average-case optimization instances, this kind of spectral certificate provides the best known efficiently-computable bound on the instance's optimal value; it is therefore of great interest to characterize $\Eig(\calI)$ for random instances.  See, e.g.,~\cite{MOP20} for further discussion.
As we will describe below, for a natural model of random instances $\bcalI$ of a degree-$2$ CSP, our \Cref{thm:bc-with-edge-signing} allows us to determine the typical value of $\Eig(\bcalI)$ for large instances.
Often one can show this exceeds the typical value of $\Opt(\bcalI)$ for these random instances, thus leading to a potential \emph{information-computation gap} for the certification task.

We should also mention another efficiently-computable upper bound on $\Opt(\calI)$:

\begin{definition}[SDP bound]\label{def:SDP-relaxation}
    Let $\calI$ be a degree-$2$ CSP instance and let $A$ be the adjacency matrix of its instance graph.  The \emph{SDP bound} is
    \[
        \SDP(\calI) = \max_{\substack{X \in \R^{n \times n} \text{ PSD} \\ X_{ii} = 1\ \forall i}} \langle A, X \rangle.
    \]
\end{definition}
This quantity can also be computed (to arbitrary precision) in polynomial time, and it is easy to see that $\Opt(\calI) \leq \SDP(\calI) \leq \Eig(\calI)$ always.
Thus the SDP value can only be a better efficiently-computable upper bound on~$\Opt(\calI)$, and it would be of interest also to characterize its typical value for random degree-$2$ CSPs, as was done in~\cite{MS16,DMO+19,MOP20}.
Proving that $\SDP(\bcalI) \sim \Opt(\bcalI)$ with high probability (as happened in those previous works) seems to require that the CSP has certain symmetry properties that do not hold in our present very general setting.
We leave investigation of this to future work.\\

We now describe a model of random degree-$2$ CSPs for which Bordenave--Collins's \Cref{thm:bc-main-intro} and our \Cref{thm:bc-with-edge-signing} lets one determine the (high probability) value of the eigenvalue relaxation bound.
It is based on the ``additive lifts'' construction from~\cite{MOP20}.

Suppose we have a degree-$2$ Boolean CSP with constraints $\Psi = \{\psi_1, \dots, \psi_t\}$ of arity~$r$.
We wish to create random ``constraint-regular'' instances defined by numbers $c_{jk}$ ($1 \leq j \leq r$, $1 \leq k \leq t$) in which each variable appears in the~$j$th position of $c_{jk}$ constraints of type $\psi_k$.
Let $c = \sum_{jk} c_{jk}$, and for notational simplicity let $\calI_0 = (\phi_1, \dots, \phi_c)$ denote a minimal such CSP instance on a set $[r]$ of variables, where each $\phi_i$ stands for some $\psi_k$ with permuted variables, and every scope is considered to be $(1, \dots, r)$.
For each constraint $\phi_i$, we associate an \emph{atom} graph~$A_i$, which is the instance graph on vertex set~$[r]$ defined by the single constraint $\phi_i$ applied to the~$r$ variables.

Now given $n \in \N^+$, we will construct a random CSP on $nr$ variables with $nc$ constraints as a ``random lift''.
We begin with a \emph{base graph}: the complete bipartite graph $K_{r,c}$, with the $r$ vertices in one part representing the variables, and the $c$ vertices in the other part representing the constraints.
We call any $n$-lift of this base graph a \emph{constraint graph}; we can view it as an instance of the CSP where the edges encode which variables participate in which constraints for the random CSP.
When we do a random $n$-lift of this $K_{r,c}$, we create $r$ groups of $n$ variables each, and $c$ groups of $n$ constraints, with a random matching being placed between every group of variables and every group of constraints.
In this random bipartite graph, each variable participates in exactly~$c$ constraints, one in each group, while each constraint gets exactly one variable from each group.

To obtain the instance graph from the constraint graph, we start with an empty graph with $nr$ vertices corresponding to the variables, and then, for each constraint vertex in the constraint graph from some group $j$, we place a copy of the atom $A_j$ on the $r$ vertices to which the constraint vertex is adjacent.

With this lift-based model for generating random constraint-regular CSP instances, it is important to allow for \emph{random literals} in the lifted instance.
(Otherwise one may easily generate trivially satisfiable CSPs. Consider, for example, random NAE-3Sat instances., as in~\cite{DMO+19}.
Without literals, all lifted instances will be trivially satisfiable due to the partitioned structure of the $3n$ variables: one can just assigning two of the three $n$-variable groups the label~$+1$ and the other $n$-variable group the label~$-1$.
Hence the need for random literals.)
Thus instead of doing a random $n$-lift of the base $K_{r,c}$ graph, we do a random \emph{signed} $n$-lift, placing a random $\pm 1$ sign on each edge in the lifted random graph.
Then, if $u$ and $v$ are variables and they are connected to a constraint vertex $y$ in the constraint graph, then, in the instance graph the weight of the edge $\{u,v\}$ (if it exists in the atom) will pick up an additional sign of $\sign(\{u,y\}) \cdot \sign(\{v,y\})$.
This has the effect of uniformly randomly negating a variable when a predicate is applied to it.

Next, we will see how this model of constructing a random regular CSP is equivalent to a random lift of a particular matrix polynomial.
The coefficients of the polynomial will be in $\C^{r \x r}$, indexed by the $r$ variables which the atoms act on.
The $rc$ indeterminates in the polynomial are indexed by the edges of the base constraint graph; i.e., we have one (non-self-adjoint) indeterminate $X_{u,j}$ for each variable $u$ and each constraint $j$.
We construct the polynomial $\ourpoly$ iteratively: for each constraint $j$, and every pair of variables $u$ and $v$, if $\{u,v\}$ is an edge in in $A_j$ then we add the terms
\[
    \frac12 w_{vu}\ketbra{v}{u}X_{v,j}X_{u,j}^* + \frac12 w_{uv}\ketbra{u}{v}X_{u,j}X_{v,j}^*
\]
to $\ourpoly$\footnote{Note that this is the same construction as with the additive lifts defined in \Cref{eg:additive}}.
Then, a random signed lift fits precisely into the model of our \Cref{thm:bc-with-edge-signing}, and we conclude that for this model of random regular general degree-$2$ CSPs, the eigenvalue relaxation bound is, with high probability, $n \cdot (\lambda_{\text{max}(A_\infty)} \pm \eps)$ for any $\eps > 0$.

\section{Setup and definitions}
\label{sec:setup}

The general framework of color-regular matrix-weighted graphs and polynomials lifts was introduced in~\cite{BC19}; however we will add some additional terminology.
We will also make use of Dirac bra-ket notation.

\subsection{Matrix-weighted graphs}

\begin{definition}[Matrix-weighted graph]
    Let $G = (V,E)$ be a directed multigraph, with $V$ countable and~$E$ locally finite.
    We say that $G$ is \emph{matrix-weighted} if, for some $r \in \N^+$, each directed edge $e \in E$ is given an associated nonzero ``weight'' $a_{e} \in \C^{r \times r}$.
    We say that $G$ is an \emph{undirected} matrix-weighted graph if (with a minor exception) its directed edges are partitioned into pairs $e$ and $e^*$, where $e^*$ is the reverse of~$e$ and where $a_{e^*} = a_e^{\conj}$.
    We call each such pair an undirected edge.
    The minor exception is that we allow any subset of the self-loops in~$E$ to be unpaired, provided each unpaired self-loop~$e$ has a self-adjoint weight, $a_{e}^\conj = a_{e}$.
    (In the terminology of Friedman~\cite{Fri93}, such self-loops are called ``half-loops'', in contrast with paired self loops which are called ``whole-loops''.)
    The adjacency operator~$A$ for~$G$, acting on~$\ell_2(V) \otimes \C^r$, is given by
    \[
        \sum_{e = (u,v) \in E} \ketbra{v}{u} \otimes a_e.
    \]
    It can be helpful to think to think of $A$ in matrix form, as an $|V| \times |V|$ matrix whose entries are themselves $r \times r$ edge-weight matrices.
    Note that if $G$ is undirected then $A$ will be self-adjoint, $A = A^\conj$.
\end{definition}

\begin{definition}[Extension of a matrix-weighted graph]
    Let $G = (V,E)$ be a matrix-weighted graph.
    We will write $\wt{G}$ for the \emph{extension} of~$G$, the scalar-weighted multigraph on vertex set $V \times [r]$ formed as follows:  for each $e = (u,v) \in E$ with weight~$a_e = \sum_{i,j=1}^r c_{ij} \ketbra{i}{j}$, we include into~$\wt{G}$ the edge $((u,i),(v,j))$ with scalar weight~$c_{ij}$ for all $i,j$ such that $c_{ij} \neq 0$.
    The scalar-weighted graph $\wt{G}$ will be undirected when~$G$ is, and the two graphs have the same adjacency operator when $\ell_2(V) \otimes \C^{r}$ is identified with $\ell_2(V \times [r])$.
    We will be most interested in the case when~$G$ is undirected with weight matrices $a_e \in \{0,1\}^{r \times r}$; in this case, the extension~$\wt{G}$ will be an ordinary unweighted, undirected (multi)graph.
\end{definition}

\begin{definition}[Index/color set]
    Given parameters $d, \q \in \N$, the associated \emph{index set}, or \emph{color set}, is defined to be $\arcs = \{0, 1, 2, \dots, d, d+1, \dots, d+2\q\}$, together with the involution $* : \arcs \to \arcs$ that has $j^* = j$ for $j \leq d$ and $j^* = j+\q$ for $d < j \leq d+\q$.
    Index~$0$ is called the \emph{identity-index}, indices $1, \dots, d$ are called the \emph{matching-indices}, and $d+1, \dots, d+2\q$ are called the \emph{permutation-indices}.
    We also refer to indices as ``colors''.
    Finally, we sometimes allow an index set to not include an identity-index.
\end{definition}

\begin{definition}[Color-regular graph]
    Fix an index set $\arcs$ and a sequence $(a_j)_{j \in \arcs}$ of nonzero matrices in $\C^{r \times r}$ with $a_{j^*} = a_j^*$.
    We say that an unweighted matrix-weighted graph~$\calG$ is \emph{color-regular with weights $(a_j)_{j \in \arcs}$} if each directed edge has an associated color from the set~$\arcs$, and the following properties hold:
    \begin{itemize}
        \item Each vertex~$v$ has exactly one outgoing edge of each color.
        \item If $0 \in \arcs$, then each edge colored~$0$ is a half-loop.
        \item If $(e,e^*)$ is an undirected edge, and $e$ is colored~$j$, then $e^*$ is colored~$j^*$.
        \item The weight of every $j$-colored edge is~$a_j$.
    \end{itemize}
\end{definition}

We will find it convenient to make the following definition.
\begin{definition}[Matrix bouquet]
    A \emph{(matrix) bouquet} $\calK$ on index set~$\arcs$ is a color-regular graph on one vertex.  Note that $\calK$ is completely specified by the matrices $(a_j)_{j \in \arcs}$.
\end{definition}
\Cref{fig:matrix-bouquet-eg} gives some examples of matrix bouquets.

\begin{figure}[ht]
    \centering
    \raisebox{.28in}{\includegraphics[width=.245\textwidth]{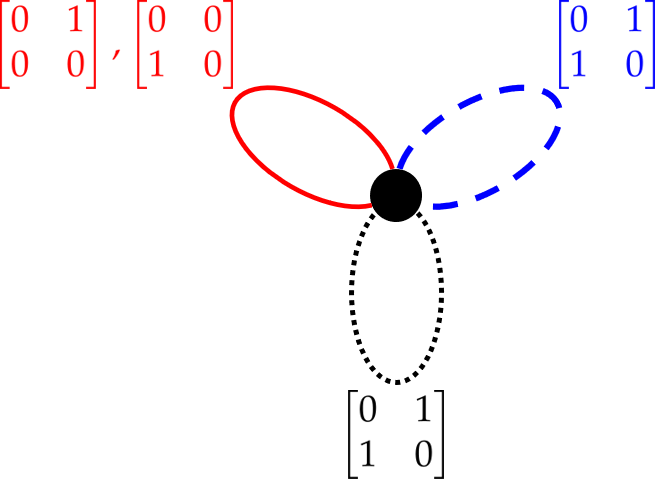}}
        \qquad \qquad
    \raisebox{.38in}{\includegraphics[width=.143\textwidth]{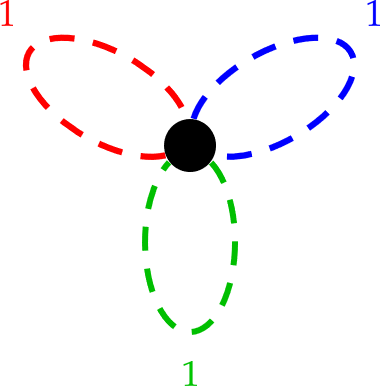}}
        \qquad \qquad
    \includegraphics[width=.362\textwidth]{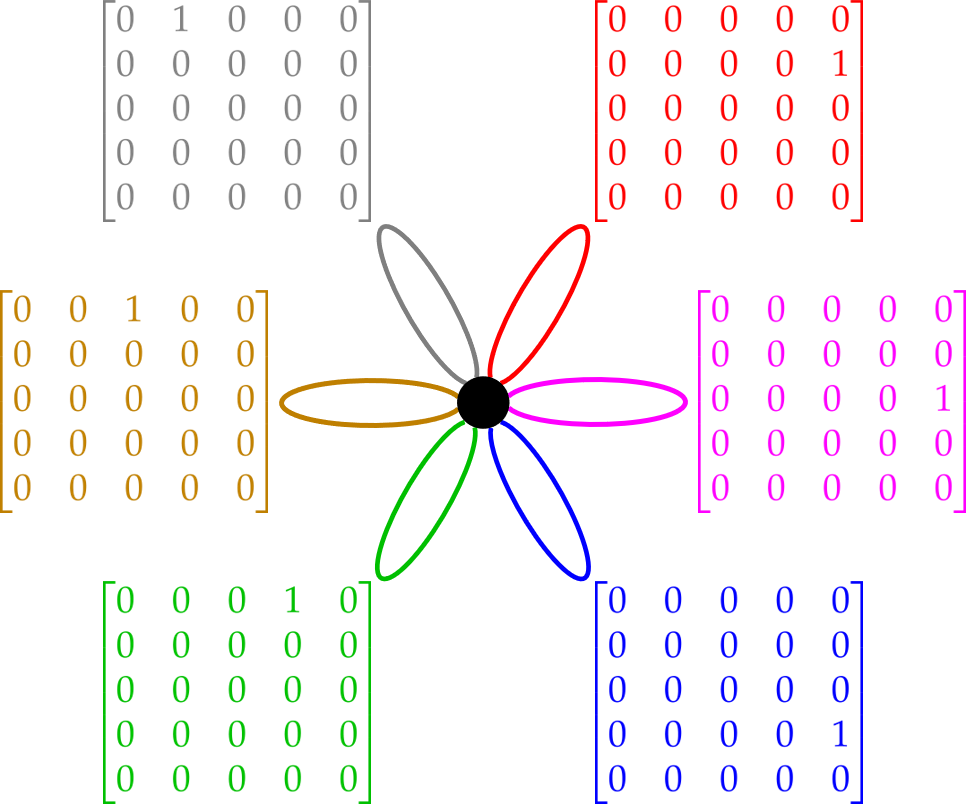}
    \caption{Example matrix bouquets. In each, identity-indices are pictured dotted, matching-indices are depicted dashed, and permutation-indices are depicted  solid.  All matrix edge-weights $a_i$ happen to have $0$-$1$ entries, though this is not in general necessary. In the left bouquet, $d = 1$, $\q = 1$, $r = 2$, and $a_0, a_1, a_2, a_3$ are $\protect\begin{bmatrix} 0 & 1 \\ 1 & 0 \protect\end{bmatrix}, \protect\begin{bmatrix} 0 & 1 \\ 0 & 0 \protect\end{bmatrix}, \protect\begin{bmatrix} 0 & 0 \\ 1 & 0 \protect\end{bmatrix}, \protect\begin{bmatrix} 0 & 1 \\ 1 & 0 \protect\end{bmatrix}$. In the middle bouquet, $d = 3$, $\q = 0$, $r = 1$, and $a_1, a_2, a_3$ are $1, 1, 1$ (there is no identity-index in this example).  In the right bouquet, $d = 0$, $\q = 6$, $r = 5$, there is no identity-index, and we have only depicted one of the $a,a^\conj$ pairs on each whole-loop.}
    \label{fig:matrix-bouquet-eg}
\end{figure}

\begin{definition}[The bouquet of an ordinary graph]
    Let $G = (V,E)$ be an ordinary unweighted, undirected multigraph (for simplicity, without half-loops), having $|V| = r$ vertices and $|E| = \q$ edges.
    We define the associated \emph{bouquet of~$G$} to be the matrix bouquet $\calK_G$ with no identity-index, $d = 0$, and $\a_1, \dots, \a_{2\q}$ being ``$0$-$1$ indicator matrices'' for the directed edges of~$G$.
    In other words, if the $i$th edge in~$E$ is $\{u,v\}$, then $\a_i = \ketbra{v}{u}$ and $\a_{i+\q} = \ketbra{u}{v}$.
    Another description is that $\calK_G$ is the one-vertex color-regular graph with $0$-$1$ matrix-weights whose extension, $\wt{\calK}_G$, is~$G$.
    We remark that $\sum_{i \in \arcs} \a_i$ is the adjacency matrix of~$G$.
\end{definition}

The third bouquet in \Cref{fig:matrix-bouquet-eg} is an example of this construction; it is the bouquet $\calK_G$ associated to the graph~$G$ in \Cref{fig:p2p2p2} (whose vertices have been numbered and whose edges have been colored, for illustrative purposes).
This example also illustrates the general point that if $\wt{\calG}$ is the extension of a color-regular graph~$\calG$ with $0$-$1$ matrix-weights, the ordinary graph~$\wt{\calG}$ need not be regular.

\begin{figure}[ht]
    \centering
    \includegraphics[width=.25\textwidth]{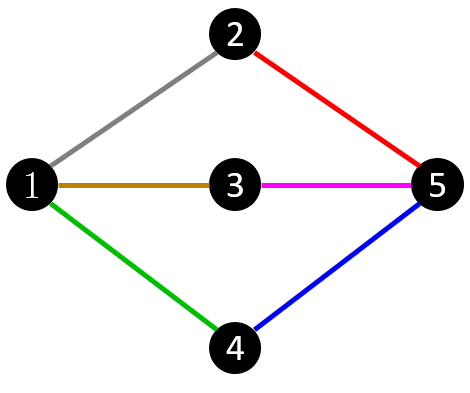}
    \caption{An undirected graph $H$ with $5$ vertices and $6$ edges; its associated matrix bouquet is the third one depicted in \Cref{fig:matrix-bouquet-eg}.}
    \label{fig:p2p2p2}
\end{figure}

We now recall the standard definitions of multigraph coverings (see, e.g.,~\cite[Sec.~2.3]{Mak15}).
We include certain extensions to allow for weighted graphs and possibly-disconnected graphs.

\begin{definition}[Matrix-weighted graph covering] \label{def:cover}
    Let $X$ and $G$ be connected undirected matrix-weighted graphs.
    A \emph{homomorphism} from $X$ to~$G$ is a pair of maps $f_V : V(X) \to V(G)$ and $f_E : E(X) \to E(G)$ such that if $e = (u,v)$ is an edge in~$X$ with weight~$a$, then $f_E(e)$ is $(f_V(u),f_V(v))$ and has weight~$a$ in~$G$.
    We say such a homomorphism is a \emph{covering} (and we say that $X$ \emph{covers} $G$) if:
    \begin{itemize}
        \item the preimage under~$f_E$ of every undirected edge (i.e., pair of opposing edges $e,e^*$) in~$E(G)$ is a collection of undirected edges in~$X$;
        \item for every vertex $v \in V(X)$, out-edges are mapped bijectively by $f_E$ to the out-edges of $f_V(v)$, and similarly for $v$'s in-edges.
    \end{itemize}

    We extend this definition to allow $X$ and/or $G$ to be disconnected; in this case we stipulate that $X$ covers~$G$ provided each connected component of~$G$ is covered by some connected component of~$X$ and, vice versa, each connected component of~$X$ covers some connected component of~$G$.
\end{definition}

We make the following observations:
\begin{fact} \label{fact:universal-covered}
    Let $\calK = (a_j)_{j \in \arcs}$ be a matrix bouquet, and let $\calG$ be a color-regular graph with weights $(a_j)_{j \in \arcs}$. Then $\calG$ covers~$\calK$.
\end{fact}

\begin{fact} \label{fact:extension}
    Suppose $X$ and $G$ are undirected matrix-weighted graphs with $X$ covering~$G$.  Then $\wt{X}$ covers~$\wt{G}$.
\end{fact}

\subsection{Matrix polynomials}

We now define \emph{matrix polynomials}, which may be thought of as recipes for producing for color-regular graphs.

\begin{definition}[Unreduced matrix polynomial]
    Given an index set $\arcs$ (necessarily including the identity-index~$0$) and a dimension $r \in \N^+$, an \emph{(unreduced) matrix polynomial} is a noncommutative polynomial over indeterminates $(X_j)_{j \in \arcs}$ with coefficients from~$\C^{r \x r}$; we explicitly disallow the empty monomial.
    More formally, the matrix polynomials are the free left $\C^{r \x r}$-module with basis given by all words~$W$ of positive length over the alphabet  $(X_j)_{j \in \arcs}$.
    We write $X^W$ for the monomial associated to word~$W$.

    We make the matrix polynomials into a noncommutative ring by specifying that $(a_{W_1} X^{W_1}) \times (a_{W_2} X^{W_2}) = a_{W_1} a_{W_2} X^{W_1W_2}$, where $W_1W_2$ denotes concatenation.
    This ring does not have a multiplicative identity, but see \Cref{def:simpler-polys} below.

    We further make the matrix polynomials into a $*$-ring by specifying that $(a X_{j_1} \cdots X_{j_k})^* = a^\conj X_{j_k}^* \cdots X_{j_1}^*$, where $X_j^*$ is a synonym for~$X_{j^*}$ (recalling the involution associated to~$\arcs$).

    We will be mainly interested in polynomials $\ourpoly$ that are \emph{self-adjoint}, meaning $\ourpoly^* = \ourpoly$.
    We will also be particularly interested in \emph{linear} matrix polynomials~$\ourpoly$, meaning $\ourpoly = \sum_{j \in \arcs} a_j X_j$.
\end{definition}

\begin{definition}[Evaluating a matrix polynomial]
    \label{def:evaluation}
    We will be ``evaluating'' matrix polynomials at bounded operators on a complex Hilbert space~$\calH$.
    Given bounded operators $(U_j)_{j \in \arcs}$, the evaluation of monomial $a X_{j_1} \cdots X_{j_k}$ is defined to be $U_{j_1} \cdots U_{j_k} \otimes a$, an operator $\calH \otimes \C^{r}$.\footnote{The reason for this strange-looking convention is that: (i)~we wish to have coefficients to the left of monomials in our polynomials, as is standard; (ii)~we wish to think of the extension graph~$\wt{G}$ as being formed from~$G$ by replacing each vertex in~$V$ with a small ``cloud'' of~$r$ vertices, and connecting edges between clouds --- but this forces us to tensor/Kronecker-product the $\C^{r \times r}$ coefficients on the right when forming adjacency matrices.}
    Furthermore, in this paper we will \emph{only} ever evaluate matrix polynomials at bounded operators satisfying the following conditions:
    \begin{itemize}
        \item $U_0 = \Id$, the identity operator on~$\calH$.
        \item $U_{j^*} = U_j^*$.
        \item All $U_j$'s are unitary, $U_j^\conj = U_j^{-1}$.
    \end{itemize}
    In light of the first two conditions above, we need not explicitly specify $U_0, U_{d + \q + 1}, \dots, U_{d+2\q}$, and hence may just write $\ourpoly(U_1, \dots, U_d, U_{d+1}, \dots, U_{d+\q})$ for a polynomial evaluation.
    Also note that $\ourpoly^*(U_1, \dots, U_{d+\q}) = \ourpoly(U_1, \dots, U_{d+\q})^*$, and hence the evaluation of any self-adjoint polynomial will be a self-adjoint operator.
\end{definition}

\begin{definition}[Reduced matrix polynomial]
    \label{def:simpler-polys}
    Given the restrictions on evaluations we imposed in \Cref{def:evaluation}, we may somewhat simplify the ring of matrix polynomials with which we work.
    First, we will let $\Id$ be a synonym for the indeterminate~$X_0$.
    Second, we may take $X_j X_j^* = X_j^* X_j = \Id$ and $\Id X_j = X_j \Id = X_j$ as ``relations'', resulting in a quotient ring which we term the \emph{(reduced) matrix polynomials}.
    Here monomials correspond to ``reduced words'' in $\Z_2^{\star d} \star \Z^{\star \q}$, the free product of $d$~copies of the group~$\Z_2$ and $\q$~copies of the group~$\Z$.
    The empty reduced word corresponds to the monomial~$\Id$, and we will sometimes abbreviate the monomial $a_0 \Id$ just as~$a_0$.
    Notice that if this $a_0$ is itself the identity operator $\Id_{r \x r}$, then the monomial becomes a multiplicative identity for the quotient ring.
\end{definition}

\begin{remark}
    In the remainder of this work, a ``matrix polynomial'' will mean a reduced matrix polynomial unless otherwise specified.
    We will write $\ourpoly = \sum_w a_w X^w$ for a generic such polynomial, where $a_w$ is nonzero for only finitely many $w \in \Z_2^{\star d} \star \Z^{\star \q}$.
\end{remark}

\subsection{Lifts of matrix polynomials} \label{sec:lifts}

\begin{definition}[$n$-lift]
    Fix an index set $\arcs = \{0, 1, \dots, d+2\q\}$.
    Let $n \in \N^+$ and write $V_n = \{1, 2, \dots, n\}$.
    We define an \emph{$n$-lift} to be a sequence $\lift_n$ of permutations $\sigma_0, \sigma_1, \dots, \sigma_{d+2\q} \in \symm{V_n} \cong \symm{n}$ satisfying
    \[
        \sigma_{j^*} = \sigma_j^{-1} \text{ for } 0 \leq j \leq d+2\q, \quad \text{with } \sigma_0 = \Id \text{ and } \sigma_j \text{ being a matching for $1 \leq j \leq d$}.
    \]
    Here a permutation is said to be a \emph{matching} when all its cycles have length~$2$.
    (We tacitly disallow simultaneously having $n > 1$ odd and $d > 0$.)
    Given a permutation $\sigma \in \symm{V_n}$, we will write $P_\sigma$ for the associated permutation operator acting on the Hilbert space~$\ell_2(V_n) \cong \C^n$, namely
    \begin{equation}    \label{eqn:Psigma}
        P_\sigma = \sum_{u \in V_n} \ketbra{\sigma(u)}{u}.
    \end{equation}
    A special case occurs when $n=1$; the unique $1$-lift $\lift_1$ is has $\sigma_0, \sigma_1, \dots, \sigma_{d+2\q}$ all equal to the identity permutation $\Id$.
\end{definition}

\begin{definition}[$\infty$-lift]
    We extend the definition of an $n$-lift to the case of $n = \infty$, as follows.
    Let $V_\infty$ denote the group $\Z_2^{\star d} \star \Z^{\star \q}$, with its components generated by $g_1, \dots, g_{d+\q}$.
    Each of these generators acts as a permutation on $V_\infty$ by left-multiplication; we write $\sigma_1, \dots, \sigma_{d+\q}$ for these permutations.
    Writing also $\sigma_0$ for the identity permutation on~$V_\infty$, and $\sigma_{j^*} = \sigma_{j}^{-1}$ for $d < j \leq d+\q$, we define $\lift_\infty = (\sigma_0, \dots, \sigma_{d+2\q})$ to be ``the'' $\infty$-lift associated to index set~$\arcs$.
    We continue to use the $P_\sigma$ notation from \Cref{eqn:Psigma} for the permutation operator acting on $\ell_2(V_\infty)$ associated to~$\sigma$.
\end{definition}

\begin{definition}[Polynomial lift]
    Let  $\ourpoly$ be a self-adjoint matrix polynomial over index set $\arcs = \{0, 1, \dots, d+2\q\}$ with coefficients in $\C^{r \times r}$.
    Write $\ourpoly = \sum_{w \in \terms} a_w X^w$, where $\terms$ is the finite set of reduced words on which~$\ourpoly$ is supported; for $w \in \terms$ we call $a_w X^w$ the associated \emph{term}.
    The $*$-operation is an involution on these terms, since $\ourpoly$ is self-adjoint.
    Thus we may consider $\terms$ to be a color set, with each $w \in \terms$ being designated a matching-index or a permutation-index depending on whether $w = w^*$ or not.
    (If $\terms$ contains the empty word, we treat that as the identity-index.)
    Now given an $n$-lift $\lift_n = (\bbi, \sigma_1, \dots, \sigma_{d+2\q})$, with $n \in \N^+ \cup \{\infty\}$, we define the associated \emph{polynomial lift} to be the $\terms$-color-regular graph $\G_n(\lift_n, \ourpoly)$ on vertex set $V_n$ defined as follows:
    for each vertex $u \in V_n$ and each term $a_w X^w$, we include a directed edge from $u$ to $\sigma^w(u)$, with matrix-weight~$a_w$.
    Here $\sigma^w$ denotes the permutation formed from the monomial $X^w$ by substituting $X_j$ with $\sigma_j$ for each $j \in \arcs$ (and it denotes the identity permutation if $w$ is the empty word).
\end{definition}

\begin{notation}
        Given a polynomial lift $\calG_n = \G_n(\lift_n,\ourpoly)$ as in the preceding definition, we write
        \[
            A_n(\lift_n, \ourpoly) = \sum_{w \in \terms} P_{\sigma^w} \otimes a_w
        \]
        for its adjacency operator on $\ell_2(V_n) \otimes \C^{r}$.
        As noted earlier, this is also the adjacency operator of its extension $\wt{\calG}_n$.
\end{notation}

We will be specifically interested in two kinds of polynomial lifts.  The first is the case when the polynomial $\ourpoly$ is \emph{linear}. As Bordenave and Collins~\cite{BC19} show, thanks to the ``linearization trick'', in order to understand the spectrum of general polynomial lifts, it suffices to understand the spectrum of linear matrix polynomial lifts (and indeed linear lifts with no ``constant term'' $a_0 \bbi$).   Because of the importance of this case, we extend the ``bouquet'' terminology:

\begin{definition}[Lifts of linear polynomials/bouquets]
    Given a linear matrix polynomial $\sum_{j = 0}^{d+2\q} a_j X_j$, we may associate it with a matrix bouquet in the natural way, deleting from the index set any~$j$ with~$a_j = 0$.
    Conversely, given any matrix bouquet $\calK = (a_j)_{j \in \arcs}$, we will identify it with the linear polynomial $\sum_{j \in \arcs} a_j X_j$ (extending the index set $\arcs$ to include~$0$ if necessary, and putting $a_0 = 0$ in this case).\footnote{It is \emph{almost} the case that the one-vertex graph bouquet is the $1$-lift of this linear polynomial; the only catch is that, following~\cite{BC19}, we have insisted that a ``matching'' permutation has no self-loops.  An alternative inelegancy would be to allow matchings to have self-loops, as in the $1$-regular configuration model.}
    Given this identification, we may write $\G_n(\lift_n, \calK)$ for the $n$-lift of this polynomial.
    When $n = 1$, we will simply denote the $1$-lift $\G_1(\lift_1, \calK)$ as $\G_1(\calK)$.
    In particular, the $\infty$-lift $\G_\infty(\lift_\infty, \calK)$ is a color-regular infinite tree of degree~$|\arcs \setminus \{0\}|$  (with self-loops, if $0 \in \arcs$).
    This graph is the \emph{universal cover} of the bouquet~$\calK$.
    For notational simplicity, we will henceforth denote it simply by $\G_\infty(\calK)$.
\end{definition}

The terminology ``universal cover'' stems from the following observation (cf.~\Cref{fact:universal-covered}):
\begin{fact}
    Let $\calG$ be a color-regular graph with weights~$(a_j)_{j \in \arcs}$ and let $\calK$ be the associated bouquet.
    Then $\G_\infty(\calK)$ covers~$\calG$.
\end{fact}

More generally, we have the following key observation:
\begin{fact}    \label{fact:key}
    Let  $\ourpoly$ be a self-adjoint matrix polynomial over index set $\arcs = \{0, 1, \dots, d+2\q\}$ with coefficients in $\C^{r \times r}$.
    Let $\lift_n = (\bbi, \sigma_1, \dots, \sigma_{d+2\q})$ be an $n$-lift, $n \in \N^+$.
    Then $\G_\infty(\ourpoly)$ covers $\G_n(\lift_n, \ourpoly)$ and hence (\Cref{fact:extension}) also $\Xram = \wt{\G}_\infty(\ourpoly)$ covers $G = \wt{\G}_n(\lift_n, \ourpoly)$.
\end{fact}

The second kind of polynomial lift that will concern us is the case when the polynomial $\ourpoly$'s coefficient matrices have $0$-$1$ entries.
In this case, the extended $\infty$-lift~$\Xram$ described in \Cref{fact:key} will be an ordinary unweighted infinite graph, and the extended $n$-lift~$G$ will be an ordinary unweighted finite graph that is covered by~$\Xram$.
The main theorem of Bordenave and Collins~\cite{BC19} implies that when the $n$-lift $\lift_n$ is chosen uniformly at random, the resulting~$G$ will be $\Xram$-Ramanujan (cf.~\Cref{def:X-ram}) with high probability.
Our work has two aspects.
First, we derandomize the Bordenave--Collins result, provided deterministic $\poly(n)$-time algorithms for producing $n$-lifts $\lift_n$ such that $G = \wt{\G}_n(\lift_n, \ourpoly)$ is $\Xram$-Ramanujan (starting in \Cref{sec:setup2}).
Second, we explore and partly characterize the kinds of infinite graphs~$\Xram$ that may arise as $\Xram = \wt{\G}_\infty(\ourpoly)$ (in \Cref{sec:fun-graphs}).
As we will be significantly investigating these graphs, we will give them a name:
\begin{definition}[\MPL graph] \label{def:mpl}
    We say an (undirected, unweighted, multi-)graph $\Xram$ is \emph{\an \MPL graph} if there is a matrix polynomial $\ourpoly$ with coefficient matrices in $\{0,1\}^{r \times r}$ such that $\wt{\G}_\infty(\ourpoly)$ consists of disjoint copies of~$\Xram$.
\end{definition}
We allow disjoint copies for two reasons: (i)~in some cases, we only know how to generate multiple copies of~$\Xram$ via polynomial lifts; (ii)~if $\Xram'$ consists of disjoint copies of~$\Xram$, then the notions of ``$\Xram$-Ramanujan'' and ``$\Xram'$-Ramanujan'' coincide (since $\Xram'$ has the same spectrum --- indeed, spectral measure --- as $\Xram$, and since $\Xram'$ covers~$G$ if and only if $\Xram$ covers~$G$).

\subsection{Projections}\label{sec:projections}
\begin{notation}[Projection to the nontrivial subspace]
    For $n < \infty$, we define the following unit vector:
    \[
        \ket{+}_n = \tfrac{1}{\sqrt{n}}\sum_{u \in V_n} \ket{u} \in \ell_2(V_n).
    \]
    We sometimes identify this vector with its $1$-dimensional span, and we write $\ket{+}_n^\bot \leq \ell_2(V_n)$ for its $(n-1)$-dimensional orthogonal complement.
    Every permutation matrix $P_\sigma$ for $\sigma \in \symm{V_n}$ preserves both $\ket{+}_n$ and $\ket{+}_n^\bot$; thus we may write
    \begin{equation}\label{P-perp}
        P_\sigma = \ket{+}_n\!\bra{+}_n + (0 \oplus P_{\sigma,\bot}),
    \end{equation}
    where $P_{\sigma,\bot}$ denotes the action of $P_\sigma$ on $\ket{+}_n^\bot$ (i.e., the standard group representation of~$\sigma$) and the $0$ is operating on $\ket{+}_n$.
    We may analogously define $A_{n, \bot}(\lift_n, \ourpoly)$, and for linear polynomials, $B_{n,\bot}(\lift_n, \ourpoly)$ for the action of adjacency/nonbacktracking operators on $\ket{+}_n^\bot$.
\end{notation}

We refer to the eigenvalues in $\spec(A_n) \setminus \spec(A_{n,\bot})$ as the ``trivial'' eigenvalues. These trivial eigenvalues are precisely the $r$ eigenvalues of the $1$-lift $A_1(\ourpoly{})$.
\begin{proposition}\label{prop:trivial-eigs}
    The following multiset identity holds:
    \[
        \spec(A_n(\lift_n, \ourpoly)) = \spec(A_1(\ourpoly)) \cup \spec(A_{n,\bot}(\lift_n, \ourpoly)).
    \]
\end{proposition}
\begin{proof}
From \Cref{P-perp},
\[
    A_n(\lift_n,\ourpoly) = \sum_{w \in \terms} \ketbra{+}{+}\otimes \a^w + \sum_{w\in\terms} (0 \oplus P_{\sigma^w,\bot}) \otimes \a^w. \qedhere
\]
\end{proof}

\section{On \mpl graphs}\label{sec:fun-graphs}

The goal of this section is to illustrate a wide variety of infinite graphs that can be realized as \mpl graphs, and to prove some partial characterizations of \mpl graphs.
In this section we will freely switch between writing $X_j^*$ and $X_j^{-1}$ for the adjoint of $X_j$.
We may also sometimes return to the convention (from the introduction) of writing $Y_j$ for self-adjoint indeterminates and $Z_j, Z_j^*$ for the remaining adjoint pairs.
We remind the reader of the convention (arising because we multiply matrices on the left) that a term like $X_i X_j$ means ``first do~$X_j$, then do~$X_i$''.

\subsection{Examples of \mpl graphs}
\label{sec:graph-examples}
In this section we will give several examples of \mpl graphs, and demonstrate that generalize a number of graph products found in the literature, including free products of finite vertex transitive graphs~\cite{Zno75}, free products of finite rooted graphs~\cite{Que94}, additive products~\cite{MO20}, and amalgamated free products~\cite{VK19}.

For finite lifts, we additionally show how some replacement products and zig-zag products~\cite{RVW02} may be expressed as matrix polynomial lifts.

\begin{example}[$r = 1$, linear polynomials]
    The simplest example of \mpl graphs occurs when $r = 1$ and the polynomial is linear.
    In this case, $\ourpoly = Y_1 + \cdots + Y_d + Z_1 + Z_1^{-1} + \cdots + Z_\q^{-1}$ for some $d, \q$, and the resulting \mpl graph, $\Xram = \wt{\G}_\infty(\ourpoly)$, is the $(d+2\q)$-regular infinite tree.
\end{example}

\begin{example}[$r = 1$, a general polynomial]
    In fact, more interesting \mpl graphs can already be created with $r = 1$.
    For instance, with $\ourpoly = Y + Z + Z^{-1} + Z^{-1} Y Z$ we obtain that  $\wt{\G}_\infty(\ourpoly)$ is $C_4 \star C_4$, the free product of two $4$-cycles.
    See \Cref{fig:gallery-0} for an illustration, and \Cref{eg:free-products} for a generalization to arbitrary free products of Cayley graphs.
\end{example}

\begin{example}[Lifts and universal covers]
As we have already seen in \Cref{sec:lifts}, when we take $\calK$ to be a matrix bouquet of a finite graph $G = (V,E)$ (a linear polynomial where each coefficient has only $1$ nonzero entry),
the extended $n$-lift $\wt\G_n(\lift_n, \calK)$ is a lift of $G$ in the sense of Amit and Linial~\cite{AL02}.
Moreover, the infinity-lift $\wt\G_\infty(\calK)$ contains $|V|$ copies of the universal covering tree of $G$.

Using matrix bouquets we can easily obtain non-regular graphs.
For example, if we let $G$ be a $(3,4)$-bipartite complete graph, the infinity-lift of the matrix bouquet of $G$ contains $7$ copies of the infinite $(3,4)$-biregular tree.
\end{example}

\begin{example}[Adding cycles by polynomial terms]\label{eg:cycles}
Next, we give a somewhat esoteric example, where the extension of the infinite lift $\wt\G_\infty$ contains infinitely many copies of a finite graph.
This construction will not be as useful for applying \Cref{thm:bc-main-intro}, but will be illustrative for further examples of various graph products.

When the infinite lift does have infinite copies of a finite graph, the spectrum is equal to the finite graph's spectrum (but with infinite multiplicity).

To create infinitely many copies of a finite undirected graph $G = (V,E)$, we construct the polynomial iteratively.
We first start off with a spanning tree of $G$, which we call $H = (V, E')$, and let $\ourpoly$ be the linear polynomial corresponding to the matrix bouquet of $H$.
When we create the matrix bouquet, $\ourpoly$ has $|E|$ pairs of adjoint indeterminates $Z_i$ and $Z_{i}^*$, one associated with each undirected edge in $H$.
We further have an involution on the indices such that $Z_{i}^* = Z_{i^*}$.
Since the universal cover of a tree is itself, it's clear that $\wt\G_\infty(\ourpoly)$ now contains countably infinite copies of $H$.
Next, we ``add'' $\ourpoly{}$ the terms which will generate the edges of $G$ which are not in $H$.
If $\{u,v\}$ is an edge in $G$ but not $H$, there is a sequence of directed edges bringing $u$ to $v$ which corresponds to a monomial $Z_{i_1}Z_{i_2}\dots Z_{i_k}$.
We then add the terms $\ketbra{v}{u} Z_{i_k}\dots Z_{i_1} + \ketbra{u}{v} Z_{i_1^*} \dots Z_{i_k^*}$ to the polynomial.
An example is in~\Cref{fig:cycles}.
\begin{figure}[htbp]
    \centering
    \begin{subfigure}{0.3\textwidth}
        \includegraphics[width=\linewidth]{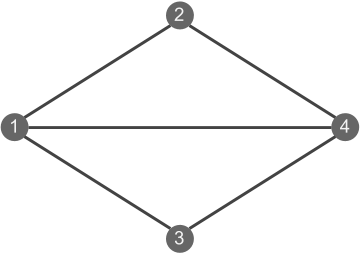}
        \caption{The graph we want to create}
        \label{fig:cycles1}
    \end{subfigure}
    \qquad\qquad
    \begin{subfigure}{0.3\textwidth}
        \includegraphics[width=\linewidth]{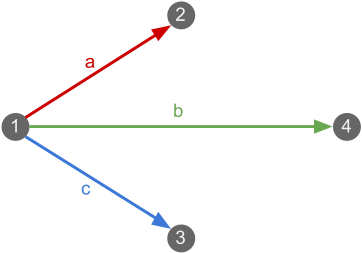}
        \caption{A spanning tree, where only one of the directed edges is shown}
        \label{fig:cycles2}
    \end{subfigure}
    \caption{An example of creating arbitrary graphs with cycles in the infinite lift. First we find a spanning tree (shown in~\Cref{fig:cycles2}. The edges have been colored and oriented for illustration purposes; the direction of the red edge indicates the directed edge corresponding to $Z_a$, the opposite edge corresponds to $Z_a^*$). The spanning tree corresponds to the linear polynomial $\ketbra{2}{1}Z_a + \ketbra{4}{1}Z_b + \ketbra{3}{1}Z_c + \ketbra{1}{2}Z_{a^*} + \ketbra{1}{4}Z_{b^*} + \ketbra{1}{3}Z_{c^*}$. The two non-tree edges are given by $\ketbra{4}{2}Z_bZ_{a^*} + \ketbra{2}{4}Z_aZ_{b^*}$ and $\ketbra{3}{4}Z_cZ_{b^*} + \ketbra{4}{3}Z_bZ_{c^*}$. Adding all these together gives the nonlinear polynomial that creates copies of the graph in~\Cref{fig:cycles1} in the infinite polynomial lift.}
    \label{fig:cycles}
\end{figure}
\end{example}

\begin{example}[Free products of finite vertex transitive graphs] \label{eg:free-products}
The construction in~\Cref{eg:cycles} is a helpful building block in creating graph products, for instance the free product of vertex transitive graphs (as defined by Zno\u{\i}ko~\cite{Zno75}).
For instance, this construction includes Cayley graphs of free products finite groups.

Let $G$ and $H$ be finite vertex transitive graphs (e.g.~those of Cayley graphs of finite groups; the particular generating set used does not matter here).
We will construct the free product $G \star H$ as \an \mpl graph.
Let $\ov G$ and $\ov H$ be spanning trees of $G$ and $H$ respectively.
Then, let $\ourpoly$ be the linear polynomial corresponding to the matrix bouquet of $\ov G \x \ov H$ (the Cartesian product), and let $q$ be the sum of the polynomial terms corresponding to edges in $G \x H$ but not in $\ov G \x \ov H$ (constructed in the same way as in~\Cref{eg:cycles}).
Then, $p + q$ is a polynomial whose infinite lift contains (multiple copies of) the free product of $G$ and $H$. An example for creating $C_3 \star C_4$ is shown in~\Cref{fig:free_prod_ex}.

\begin{figure}[htbp]
    \centering
    \includegraphics[width=0.3\textwidth]{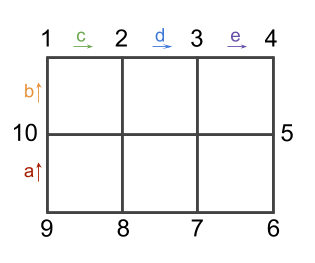}
    \caption{Here is an example of how to create $C_3 \star C_4$ as \an \mpl graph.
    We first get the matrix bouquet of the Cartesian product of a spanning tree of $C_3$ and $C_4$ (which are simple paths of length 2 and 3 respectively), creating the grid shown above.
    Then, when creating the matrix bouquet, we assign generators to each edge (some omitted in the drawing), which orients the edges as shown.
    Finally, we add terms corresponding to the edges in $C_3 \x C_4$ but not in the grid, for example, $\ketbra{1}{9} Z_b Z_a + \ketbra{9}{1} Z_{a^*} Z_{b^*}$ for the edge $\{1,9\}$ and $\ketbra{1}{4}Z_e Z_d Z_c + \ketbra{4}{1}Z_{c^*}Z_{d^*}Z_{e^*}$ for the edge $\{1,4\}$.
    We similarly add terms corresponding to $\{10,5\}$, $\{2,8\}$, $\{3,7\}$ and $\{4,6\}$.
    }
    \label{fig:free_prod_ex}
\end{figure}
\end{example}

\begin{example}[Additive lifts and products]\label{eg:additive}
Additive lifts and products are a graph product for non-vertex-transitive finite graphs defined by Mohanty and O'Donnell in in~\cite{MO20}.

The components of the additive product are finite, unweighted, undirected graphs called \emph{atoms}, which are defined on a common vertex set.

\begin{definition}[Additive products {\cite[Definition 3.4]{MO19}}] \label{def:additive}
    Let $A_1, \dots, A_c$ be atoms on a common vertex set $[r]$. Assume that the sum graph $G = A_1 + \dots + A_c$ is connected; letting $\underline{A}_j$ denote $A_j$ with isolated vertices removed, we also assume that each $\underline{A}_j$ is nonempty and connected. We now define the (typically infinite) \emph{additive product} graph $A_1 \pluscirc \dots \pluscirc A_c \coloneqq (V,E)$ where $V$ and $E$ are constructed as follows.

    Let $v_1$ be a fixed vertex in $[r]$; let $V$ be the set of strings of the form $v_{k+1}C_kv_kC_{k-1}\dots v_2 C_1 v_{1}$ for $k \geq 0$ such that:
    \begin{enumerate}[label=\roman{*})]
        \item each $v_i$ is in $[r]$ and each $C_i$ is in $[c]$,
        \item $C_i \ne C_{i+1}$ for all $i < k$,
        \item $v_i$ and $v_{i+1}$ are both in $\underline{A}_{C_i}$ for all $i \leq k$;
    \end{enumerate}
    and, let $E$ be the set of edges on vertex set $V$ such that for each string $s \in V$,
    \begin{enumerate}[label=\roman{*})]
        \item we let $\{uCs, vCs\}$ be in $E$ if $\{u, v\}$ is an edge in $\underline{A}_C$,
        \item we let $\{uCs, vC'uCs\}$ be in $E$ if $\{u,v\}$ is an edge in $\underline{A}_{C'}$, and
        \item we let $\{v_1, vCv_1\}$ be in $E$ if $\{v_1, v\}$ is an edge in $\underline{A}_C$.
    \end{enumerate}
\end{definition}

We show two ways to construct additive products using polynomial lifts, which illustrates that constructions using polynomial lifts are in general not unique. Let the atoms be $A_1, \dots, A_m$ on $r$ vertices each.

For the first construction, for each $A_i$, let $\ov A_i$ be a spanning tree.
Let $H$ be the sum of $\ov A_1, \dots, \ov A_m$, including parallel edges (in other words, sum the adjacency matrices of $\ov A_1, \dots, \ov A_m$).
Then, we start with the matrix bouquet of $H$.
Then, for each $A_i$, for each edge not in the spanning tree $\ov A_i$, we add the term corresponding to that edge to the polynomial similar to~\Cref{eg:cycles}.

The second construction has a pair of adjoint indeterminates $Z_{A,v}$ and $Z_{A,v}^*$ for each atom $A$ and each vertex $v$ which is not an isolated vertex in $A$.
We construct the polynomial $\ourpoly$ iteratively. For each edge $\{u,v\}$ in the atom $A$, we add the term $\ketbra{v}{u}Z_{A,v}Z_{A,u}^* + \ketbra{u}{v}Z_{A,u}Z_{A,v}^*$ to the polynomial $\ourpoly$.
We repeat this for for every edge in each every atom $A_i$, $1 \leq i \leq m$. Then $\wt\G_\infty(\ourpoly)$ consists of ($r$ copies of) the additive lift $A_1 \pluscirc \dots \pluscirc A_m$.

The finite graphs arising from the $n$-lifts of the second construction also has a nice interpretation. Mohanty and O'Donnell make the following definition of \emph{additive lifts},
\begin{definition}[additive lifts]
    Let $A_1, \dots, A_c$ be atoms on a common vertex set $[r]$. The \emph{additive $n$-lift} of $A_1, \dots, A_c$ is the following:
    \begin{enumerate}[label=\roman{*})]
        \item It has vertex set $[n] \x [r]$.
        \item For each vertex $v$ in each atom $A_i$, let $\sigma_{A_i, v}$ be a permutation on $[n]$.
        \item For each atom $A_i$, for each original edge $(u,v)$, add the matching $\sigma_{A_i, v}\inv \sigma_{A_i, u}$ to the vertices $[n] \x \{u\}$ and $[n] \x \{v\}$.
    \end{enumerate}
\end{definition}
One can check that an additive $n$-lift is the same as an $n$-lift of the polynomial described in the second construction above.
Applying \Cref{thm:bc-main-intro}, we can moreover deduce that as $n \to \infty$ the spectrum of a uniformly random additive $n$-lift (with the trivial eigenvalues of $A_1 + \dots + A_m$ removed) is close in Hausdorff distance to the corresponding additive product with high probability.
\end{example}

\begin{example}[Amalgamated free products]
Vargas and Kulkarni~\cite{VK19} define another graph product based on amalgamated free products from free probability.
These generalize the notion of free products of graphs defined in~\cite{Que94}, and can be used to express Cayley graphs of amalgamated free group products.
These graph products are defined on \emph{rooted} graphs, i.e.\ graphs $(V,E)$ with a distinguished vertex $r$. We denote these rooted graphs as a triple $(V,E,r)$.

\begin{definition}[\cite{VK19}]
Let $G_1 = (V_1,E_1,o_1), \dots, G_n = (V_n, E_n, o_n)$ be finite rooted undirected graphs. Assume that each $G_i$ comes equipped with an edge coloring $c_i: E_i \to C_i$ such that $C_i \cap C_j = \vn$ for every $i \ne j$. Let $C \coloneqq \bigcup_{i=1}^n C_i$ and let $G = (V,E,r)$ be a rooted graph $G = (V,E,r)$ together with an edge coloring $c : E \to C$, and where $V = [k]$. We call $(G,c)$ the \emph{relator graph}.

We construct the \emph{free product of the $(G_i,c_i)$ with amalgamation} over $(G,c)$, which we denote by $*_{G,c}\{(G_i,c_i)\}_{i=1}^n$. Let $V_0$ be the set of strings of the form $v_kv_{k-1}\dots v_2v_1e$ for $k > 0$ where
\begin{enumerate}[label=\roman{*})]
\item $o$ is one of $o_1,\dots ,o_n$,
\item $v_1,\dots,v_k$ are elements of $V_1\setminus o_1, \dots, V_n\setminus o_n$,
\item $v_i, v_{i+1}$ do not belong to the same $V_j \setminus o_j$ for $1 \leq i \leq k-1$.
\end{enumerate}
The vertex set is $\{(r,o)\} \cup [k] \x V_0$.
The edges are defined such that $\{(i,vu), (i', v'u)\}$ is an edge if
\begin{enumerate}[label=\roman{*})]
\item $\{i,i'\}$ is in $E$,
\item $u$ is the empty string or some element of $V_0$,
\item $\{v,v'\}$ is an edge in some $V_j$ for $1 \leq j \leq n$, and we treat any of $o_1,\dots, o_n$ appearing as one of $v$ or $v'$ as the empty string, except when both $u$ and $v$ or $v'$ are the empty string, in which case we let $uv$ or $uv'$ be $o$,
\item finally, $c(\{i,i'\}) = c_j(\{v,v'\})$.
\end{enumerate}
\end{definition}

This can be constructed as \an \mpl graph in the following way: Let $r$ be the number of colors $k$.
For each graph $G_i$ in the product and for each non-root vertex $v$ in $G_i$, let $Y_{G_i, v}$ be a self-adjoint indeterminate.
Let $\{u,v\}$ be an edge in $G_i$ where neither $u$ nor $v$ is $o_i$. We then add the term $(\ketbra{k_1}{k_0} + \ketbra{k_0}{k_1})Y_{G_i,u}Y_{G_i,v}$ to the polynomial for every edge $\{k_0, k_1\}$ in the relator graph $G$ such that $c(\{k_0,k_1\}) = c_i(\{u,v\})$.
When $\{o_i,v\}$ is an edge in $G_i$, we add the term $(\ketbra{k_1}{k_0} + \ketbra{k_0}{k_1}Y_{G_i,v}$ to the polynomial for every edge $\{k_0, k_1\}$ in the relator graph $G$ such that $c(\{k_0,k_1\}) = c_i(\{o_i,v\})$.

As an example (derived from Example 6.1 of~\cite{VK19}), let
\begin{align*}
    p &= \begin{pmatrix}1 & 1 \\ 1 & 1\end{pmatrix}Y_{1,1}
       + \begin{pmatrix}1 & 0 \\ 0 & 1\end{pmatrix}Y_{2,1}
       + \begin{pmatrix}0 & 1 \\ 1 & 0\end{pmatrix}Y_{2,2}
       + \begin{pmatrix}1 & 0 \\ 0 & 1\end{pmatrix}Y_{2,1}Y_{2,2}.
\end{align*}
Then, $\wt\G_\infty(p)$ is the Cayley graph of $SL(2,\Z)$ with the group presentation $\angle{a,b\,|\,a^4=b^6=1}$.
\end{example}

\begin{example}[Constant terms; beyond additive products, and amalgamated free products]
The class of \mpl graphs also include graphs beyond additive products and amalgamated free products.
To construct such graphs, one key observation is that we have yet to use the constant term $a_0$ in the polynomials.
One use case of adding the constant term is that we can create \mpl graphs where the extension of the $\infty$-lift contains only one component.
For instance, one can add edges which connect the different copies of the universal covering tree created by lifting a matrix bouquet.
For example, the following polynomial
\[
    \begin{pmatrix}
    0 & 1 \\ 1 & 0
    \end{pmatrix}
    +
    \begin{pmatrix}
    1 & 0 \\ 0 & 1
    \end{pmatrix}Z
    +
    \begin{pmatrix}
    1 & 0 \\ 0 & 1
    \end{pmatrix}Z^*
\]
gives rise to a graph which looks like a infinite ladder (\Cref{fig:ladder}).
In general, there may be multiple, possibly very different, ways to construct the same graph. We can also create ladders without using the constant term $a_0$. For example, we can construct (6 copies of) the graph in \Cref{fig:ladder-hexagon} by starting off with the matrix bouquet of the graph in~\Cref{fig:ladder-hexagon-base} and adding further edges to create cycles. The vertices are labeled $1,\dots,6$ and the edges are labeled $a,\dots,f$.
Let
\begin{align*}
    p &= \ketbra{1}{5}Z_a + \ketbra{2}{1}Z_b + \ketbra{3}{2}Z_c + \ketbra{4}{3}Z_d + \ketbra{5}{4}Z_e + \ketbra{2}{6}Z_f \\&\quad+ \ketbra{3}{5}Z_cZ_bZ_a + \ketbra{6}{1}Z_aZ_eZ_dZ_cZ_f.
\end{align*}
Then $\wt\G_\infty(p + p^*)$ is the graph depicted in \Cref{fig:ladder-hexagon}.

\begin{figure}[ht]
    \centering
    \begin{subfigure}{0.4\textwidth}
        \includegraphics[width=\linewidth]{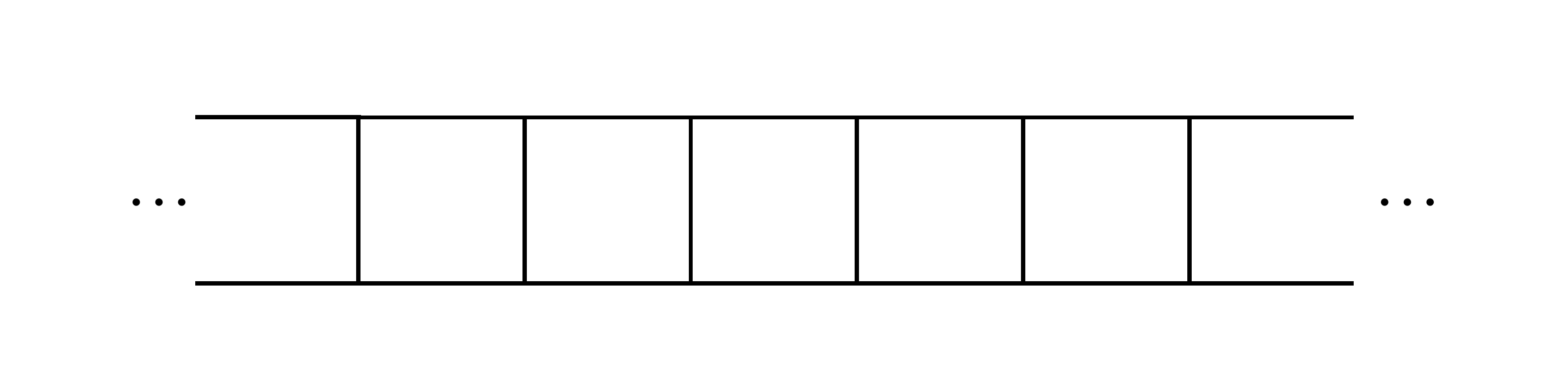}
        \caption{The infinite ladder graph}
        \label{fig:ladder}
    \end{subfigure}
    \qquad\qquad
    \begin{subfigure}{0.4\textwidth}
        \includegraphics[width=\linewidth]{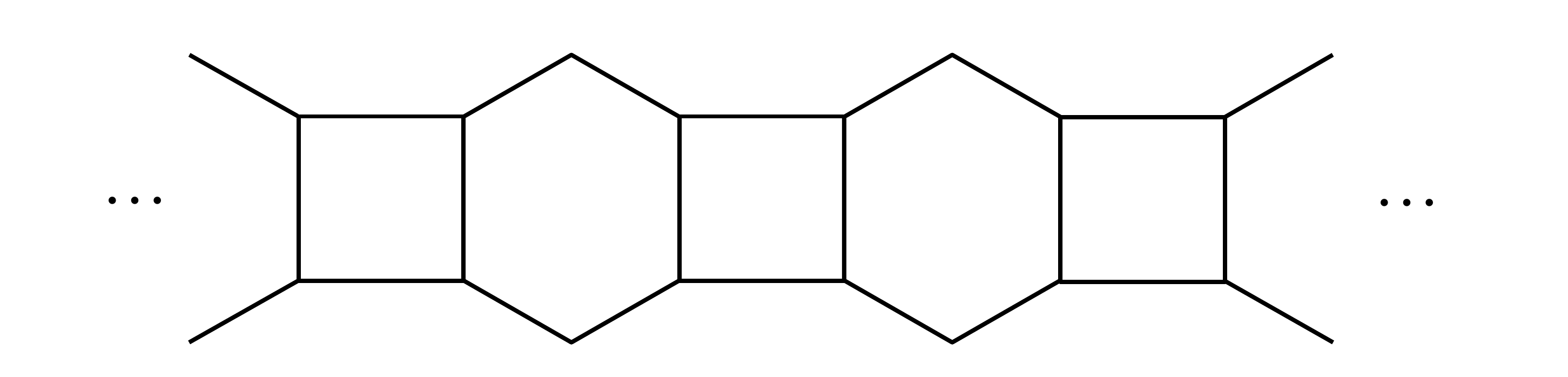}
        \caption{A ladder with alternating hexagons and squares}
        \label{fig:ladder-hexagon}
    \end{subfigure}
    \caption[Examples of \mpl graphs]{\centering Examples of \mpl graphs}
\end{figure}
\begin{figure}[ht]
    \centering
    \includegraphics[width=0.25\textwidth]{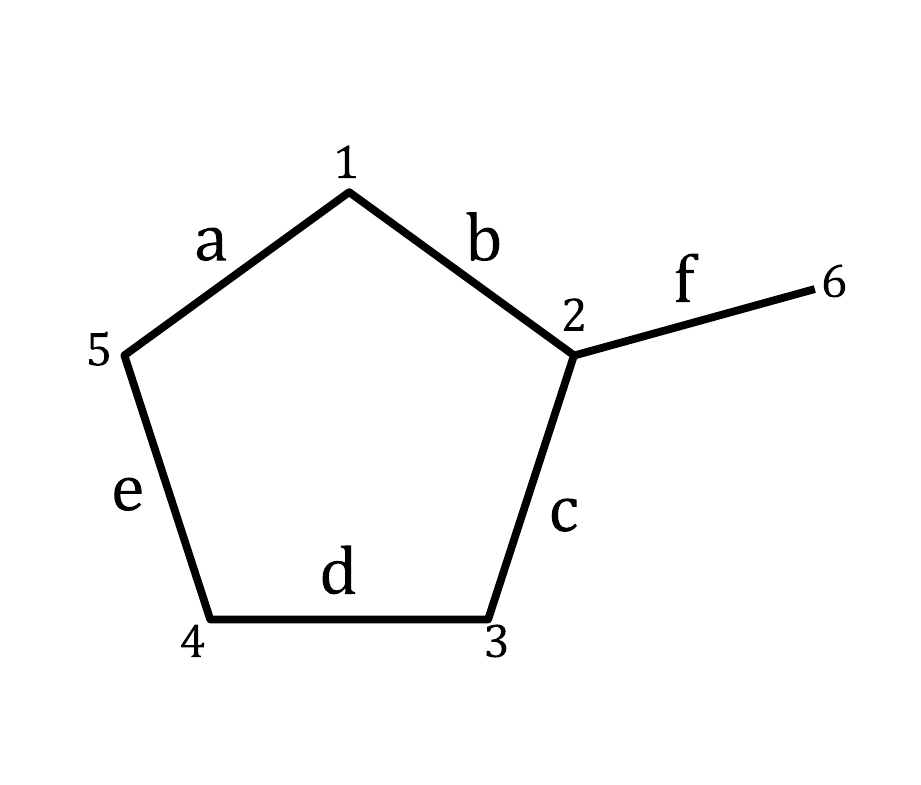}
    \caption[Base graph of the hexagon-square ladder]{\centering An illustration of the graph whose matrix bouquet serves as a base for \Cref{fig:ladder-hexagon}}
    \label{fig:ladder-hexagon-base}
\end{figure}

Another example of a graph which, as far as we know, requires a constant term to express as \an \mpl graph is depicted in \Cref{fig:gallery-1}.

\end{example}

\newcommand{\rprod}{\mathbin{\raisebox{.2ex}{
      \hspace{-.4em}$\bigcirc$\hspace{-.73em}{\rm r}\hspace{.15em}}}}
\newcommand{\zprod}{\mathbin{\raisebox{.2ex}{
      \hspace{-.4em}$\bigcirc$\hspace{-.76em}{\rm z}\hspace{.15em}}}}
\newcommand{\Rot}{\operatorname{Rot}}
\begin{example}[Replacement products and zig-zag products] \label{eg:zig-zag}
It is also possible to view some graphs arising from polynomial lifts as the result of replacement products and zig-zag products~\cite{RVW02}.

\begin{definition}[Replacement product]
    Let $G$ be a $D$-regular undirected, unweighted graph on $n$ vertices, and let $H$ be a $d$-regular undirected, unweighted graph on $D$ vertices.
    Further, $G$ is equipped with a \emph{rotation map} $\Rot_G: [n] \x [D] \to [n] \x [D]$ is a permutation where $\Rot_G(u,i) = (v,j)$ if the $i\th$ outgoing edge from $u$ is the $j\th$ outgoing edge from $v$. This can be thought of as a coloring on the edges.

    The \emph{replacement product} of $G$ and $H$, denoted $G \rprod H$, is a graph on $nD$ vertices which we identify with $[n] \x [D]$. The graph is constructed by first making $n$ copies of $H$, labeling the vertices with $[n] \x [D]$, and then joining every $(u,i) \in [n] \x [D]$ to $\Rot_G(u,i)$.
    This results in a $(d+1)$-regular graph.
\end{definition}

With some further restrictions on $G$, we can view the replacement product as a matrix polynomial lift. Namely, we require that $G$ is the sum of $\q$ permutations and $d'$ matchings on $[n]$. In particular, we fix an involution $*$ on $\{0,1,\dots,d+2\q\}$ as in a matrix bouquet, and require that $G$ is the graph of an $n$-lift $\lift_n = (\sigma_0,\sigma_1,\dots,\sigma_{d+2\q'})$ (satisfying $\sigma_{i^*} = \sigma_i \inv$). In terms of the rotation map, this results in $\Rot_G(u,i) = (\sigma_i(u), i^*)$.
Now let $A_H$ be the adjacency matrix of $H$, and define
\[
    p \coloneqq A_H + \sum_{i=1}^{d+2\q'}\ketbra{i}{i^*}X_i.
\]
The $X_i$ terms are self adjoint whenever the corresponding permutation $\sigma_i$ is a matching, and are part of an adjoint pair $X_i, X_{i^*}$ otherwise.
Then, $A_n(\lift_n, p)$ is the adjacency matrix of the replacement product $G \rprod H$.

\begin{definition}
    Similar to replacement products, the \emph{zig-zag product} of a $D$-regular graph $G$ on $n$ vertices, and a $d$-regular graph $H$ on $D$ vertices is a graph on $[n] \x [D]$, denoted $G \zprod H$. First, we create $n$ \emph{clouds} of $D$ vertices, each corresponding to a copy of $H$. An edge exists in $G \zprod H$ from $(u,i)$ to $(v,j)$ if there exist $i'$ and $j' \in [D]$ such that $(i,i')$ is an edge in $H$, $\Rot_G(u,i') = (v,j')$, and $(j',j)$ is an edge in $H$. This can be thought of as taking a step within a cloud (with edges defined by $H$), then a step between clouds (with edges defined by $G$), and then a step within a cloud. This construction results in a $d^2$-regular graph.
\end{definition}

Again, we can express zig-zag products as the lift of a matrix polynomial when $G$ is the graph of an $n$-lift $\lift_n$. Let $A_H$ be the adjacency matrix of $H$, and define
\[
    p \coloneqq \sum_{i=1}^{d+2\q'} A_H^2 X_i.
\]
Then, $A_n(\lift_n, p)$ is the adjacency matrix of the zig-zag product $G \zprod H$.
\end{example}

\subsection{Structure, connectivity and geometry of polynomial lifts}
In this section, we will examine some properties that infinite polynomial lifts must satisfy.
We will take $\ourpoly$ to be a self-adjoint polynomial with $\{0,1\}^{r\x r}$ coefficients, in $d$ self-adjoint indeterminates and $\q$ indeterminates and their adjoints. In this section, whenever we refer to a matrix polynomial $\ourpoly{}$ we mean with these constraints unless otherwise specified.

Though our main interest is with \mpl graphs, our results in this sections apply more generally to the extensions of infinite lifts $\wt\G_\infty(\ourpoly{})$ of such polynomials (which, in general, consist of the union of possibly non-isomorphic finite or infinite graphs). We will refer to any connected component of an extension of a infinite lift as a ``graph arising from a polynomial lift''.

Recall that the vertex set of $\wt\G_\infty(\ourpoly{})$ is $V_\infty \x [r]$, where $V_\infty$ is the free product of $d$ copies of $\Z_2$ and $\q$ copies of $\Z$.
We think of the vertices of $V_\infty$ as corresponding to the \emph{reduced words} formed by $d$ self-inverse generators and $\q$ pairs of generators and their adjoints, named $g_1, \dots, g_{d+2\q}$.
Following our convention of left-multiplying permutations, e.g.~we think of the word $g_2g_1$ as $g_1$ followed by $g_2$.
For a word $w$, we write $\ov w$ to denote its reduced word.
In this section we use the notation $\lift_\infty = (o, g_1, \dots, g_{d+2\q})$ where $o$ is the identity element of $V_\infty$, $g_1 \dots g_{d+2\q}$ are generators with $g_{1} \dots g_{d}$ being self-inverse.
Given some term $a_w X^w$ in $\ourpoly$, we write $g^w$ to denote substituting the generators into $X^w$.
We label the vertices of $\wt\G_\infty(\ourpoly)$ by $(w, i) \in V_\infty \x [r]$, where $w$ is a word of generators, and $i$ indexes into the cloud of $r$ vertices corresponding to $w$.

The graphs of infinite polynomial lifts are clearly locally finite, and they additionally are constrained to look ``tree-like'', i.e.\ their structure and geometry are similar to those of trees.
We can formalize this in terms of the \emph{treewidth} and \emph{hyperbolicity}. The treewidth measures how close the graph is to a tree structurally, while the hyperbolicity of a graph measures how close the graph distance metric is to that of a tree metric.

In particular, we will find that the graphs that arise as the infinite lifts of polynomials all have finite treewidth.
Let us recall some definitions:
\begin{definition}
    Let $G=(V,E)$ be a graph (possibly with infinite vertices). A \emph{tree decomposition} of $G$ is a tree $T$ whose vertices are sets $W = \bigcup_{i \in \Sigma} W_i$ indexed by some set $\Sigma$.
    Each $W_i$ is a subset of $V$. $(T,W)$ satisfies the following properties:
    \begin{enumerate}
    \item Each $v \in V$ is in at least one $W_i$.
    \item If $(u,v) \in E$, then there exists some $i \in \Sigma$ such that both $u$ and $v$ are in some $W_i$.
    \item If $u \in V$ is in $W_i$ and also $W_j$, then it is in every $W_k$ for $W_k$ in the unique path in $T$ from $W_i$ to $W_j$.
    \end{enumerate}
\end{definition}
\begin{definition}
    The \emph{treewidth} of $G$ is the minimum of $\max_{i \in \Sigma} |W_i|-1$ over all tree decompositions $(T,W)$.
\end{definition}

\begin{proposition}\label{prop:finite-treewidth}
    Let $\ourpoly$ be a matrix polynomial, and let $m$ be the sum of the degrees of the terms of $\ourpoly$.
    Then, the treewidth of the extension of the $\infty$-lift $\wt\G_\infty(\ourpoly)$ is bounded by $(m+1)r$.
\end{proposition}
\begin{proof}
    We can construct a tree decomposition.
    Each $v \in V_\infty$ is associated to a cloud of $r$ vertices in $\wt\G_\infty(\ourpoly{})$. Our vertex sets $\{W_v\}_{v \in V_\infty}$ in the tree decomposition will be indexed by $V_\infty$, and the tree structure on $\{W_v\}$ is also inherited from $V_\infty$. Each $W_v$ contains a copy of the $r$ vertices in the cloud of $v \in V_\infty$, and for each polynomial term $a_w X^w$ in $p$, $W_v$ also contains the $r$ vertices in the cloud of $u \in V_\infty$ for every $u$ along the path from $v$ to $g^wv$.
\end{proof}
\begin{corollary}
In particular, graphs which arise from infinite lifts of matrix polynomials must have finite treewidth.
An example of a graph that does not have finite treewidth is an infinite grid. Therefore, we cannot derive grids from noncommutative polynomials.
\end{corollary}

We have seen from the previous section that, in general, the infinite lifts of a polynomial can contain many connected components. It is also easy to see that these need not be isomorphic. Nevertheless, given the labels of two vertices in $\wt\G_\infty(\ourpoly{})$, based on $\ourpoly$ it is easy to decide if they belong in the same connected component.

In the scalar case where $r = 1$, we can decide connectivity with a deterministic finite automaton by utilizing Stallings foldings~\cite{Sta83}. The following is from Section 2 of Kapovich and Myasnikov~\cite{KM02}, which we refer the reader to for full details. We reproduce a sketch of the proof here.
\begin{proposition}\label{prop:dfa}
Let $T$ be a set of words from a free group which is closed under inverse. We say a word $w$ is \emph{reachable} by $T$ if the reduced word $\ov w$ can be formed by concatenating an arbitrary combination of words from $T$, possibly with repeats, and then reducing. Then, the language $L$ of such words is regular.
\end{proposition}
\begin{proof}[Proof sketch]
We construct a deterministic finite automaton which, given an input reduced word $w$, accepts if $w$ is reachable by $T$. Since the reduced words are also a regular language, and regular languages are closed under intersection, this shows that $L$ is regular.

We construct the automaton iteratively as a directed graph labeled with the generators (the direction of the edge indicates to use the generator or its inverse). We start off with a single node $z$ which will serve as both the initial and final state. For each word $t \in T$, we add a directed loop of length $|t|$ starting and ending at $z$, such that traversing the loop recovers $t$. Then, we apply a ``folding'' process: whenever a node has two outgoing edges to nodes $x$ and $y$ with the same direction and the same label, we replace the nodes $x$ and $y$ with a single new node with incident edges equal to the union of incident edges of $x$ and $y$. The process terminates when every node has at most one incident node of every label.
\end{proof}

Given a scalar-coefficient polynomial $\ourpoly$, the vertices of $\wt\G_\infty(\ourpoly{})$ correspond to words of the free group.
For the connectivity question, the coefficients of $p$ are irrelevant, so we can assume they are all 1.
We can take $T$ to be the set of words of generators corresponding to the terms of $\ourpoly$.
Then, $u$ and $v$ in $V_\infty$ are connected if and only if $\ov{vu\inv}$ is in the language defined in \Cref{prop:dfa}.
It is easy to see that the construction can be modified to allow for self-inverse generators by making the edges labeled by those generators undirected.

With a reduction to the scalar case, we can also understand connectivity for matrix-coefficient polynomials.
\begin{proposition}\label{prop:connectivity}
    Let $\ourpoly$ be a matrix polynomial.
    Let $(u,i)$ and $(v,j) \in V_\infty \x [r]$ be vertices in $\wt\G_\infty(\ourpoly)$. It is efficiently decidable whether $(u,i)$ and $(v,j)$ belong to the same connected component.
\end{proposition}
\begin{proof}
    We first create $r$ additional generators, each corresponding to one of the vertices in a cloud. Call these generators $h_1, \dots, h_r$ and $h_1\inv \dots h_r\inv$.

    $\ourpoly$ can be written as a sum of terms of the form $\ketbra{i}{j} X^w$ (i.e.\ terms whose coefficients have exactly one nonzero entry). For each such term, add the word $h_j g^w h_i\inv$ to $T$.
    Since $\ourpoly$ is self-adjoint, $T$ is closed under inverse.
    Then, $(u,i)$ and $(v,j)$ are connected if and only if $h_j v u\inv h_i \inv$ is reachable by $T$, which we can check using \Cref{prop:dfa}.
\end{proof}

We now move on to describe the \emph{hyperbolicity} of graphs arising from infinite lifts of polynomials. Hyperbolicity is another measure of how tree-like a graph is. Gromov~\cite{Gro87} defined this notion of hyperbolicity for groups; it generalizes a notion of how much a space is like a Riemannian manifold with negative curvature~\cite{Gro83}. Hyperbolicity is also interesting for finite graphs~\cite{BRS11,BRSV13}, including random graphs~\cite{CFHM12}. It is related to other combinatorial properties of the graph such as chordality~\cite{BKM01,WZ11}, independence number and max degree~\cite{RS12}, and the circumference and girth~\cite{HPR19}.
It is also useful describing real-world graphs, and has applications in networks and routing (e.g.~\cite{MSV11,Kle07}).

In the context of geometric group theory, hyperbolicity has also been studied for infinite graphs, for instance those arising from the free products of groups~\cite{Hor16}, and tessellations of the Euclidean plane~\cite{Car17}. More generally, simplicial complexes arising from free groups such as the \emph{complex of free factors}~\cite{BH14} and the \emph{free splitting complex}~\cite{HM13} have been proven to be hyperbolic.

We use the following definition for hyperbolicity of (possibly infinite) graphs.
\begin{definition}\label{def:hyperbolicity}
    Let $G$ be an unweighted, undirected graph.
    We say that $G$ is $\delta$-hyperbolic if for every 3 vertices $u,v,w$, the shortest paths $(u,v)$, $(u,w)$ and $(v,w)$ satisfy that for every node $x \in (u,v)$ there exists a node $y$ in $(u,w)$ or $(v,w)$ such that the graph distance (length of the shortest path) $\dist(x,y) \leq \delta$.
\end{definition}
As an example, a tree is $0$-hyperbolic. We are interested in graphs where $\delta$ is finite and constant.

In contrast to treewidth, which captures information about the local structure of the graph, hyperbolicity is a global property which captures global information about distances between vertices.
Hyperbolicity and treewidth are in general not comparable; for instance, a $n$-cycle has treewidth $2$ but hyperbolicity $\ceil{\frac{n}{4}}$. Meanwhile, the Cayley graph of the fundamental group of the torus $G = \angle{a,b,c,d\,|\,aba\inv b\inv cdc\inv d\inv = 0}$ has arbitrarily large grid minors, and hence infinite treewidth, but it has finite hyperbolicity.

For graphs with multiple connected components, $\delta$ is typically taken to be infinity.
Hence, we are concerned only with the hyperbolicity of connected components within infinite polynomial lifts.
In our application of spectral approximations of polynomial lifts, we typically only consider cases where the infinite polynomial lifts is a finite number of isomorphic copies of some infinite graph.

We next show that all connected components of a infinite polynomial lift are $\delta$-hyperbolic for $\delta$ a constant depending on the polynomial $p$.
Similar to the case of connectivity, we will eventually make a reduction to the scalar case of $r = 1$.
We first prove the statement in the case of $r = 1$, hence we take $\ourpoly$ to have all coefficients 1 in the following.
Given the extension of the infinity lift $\wt\G_\infty(\ourpoly)$, we will want to consider the graph given by the polynomials and the Cayley graph of the subgroup of the free group with vertex set $V_\infty$ at the same time.
When possible we will use $\wt\G_\infty(\ourpoly)$ to refer to the polynomial graph, but $V_\infty$ to refer to the Cayley graph (switching between the two views when necessary).
To proceed, we start with some definitions:
\newcommand{\treedist}{\operatorname{treedist}}
\begin{definition}
    Let $u, v \in V_\infty$.
    We define the \emph{tree path} between $u$ and $v$ to be the sequence of generators in $\ov{vu\inv}$
    (i.e.\ the edges along the (unique) path in the Cayley graph of the free group $V_\infty$).
    We denote the \emph{tree distance} $\treedist(u,v)$ as the length of the tree path.
\end{definition}
\begin{definition}
    We define the \emph{$d$-neighborhood} of a vertex $v$ to be the set of vertices $w \in V_\infty$ such that $\treedist(v,w) \leq d$.
    We also define the $d$-neighborhood of the tree path from $u$ to $v$ to be the set of vertices $w \in V_\infty$ such that there exists $x$ on the tree path from $u$ to $v$ with $\treedist(x,w) \leq d$.
\end{definition}
\begin{definition}
    Let $p$ be a polynomial.
    Let $u, v \in \wt\G_\infty(\ourpoly)$ be in the same connected component.
    We define a \emph{monomial path} between $u$ and $v$ to be a shortest path between $u$ and $v$ in $\wt\G_\infty(\ourpoly{})$.
    Further let the \emph{monomial distance} $\dist(u,v)$ be the length of the monomial path.
    Note that there may not be a unique shortest path.
\end{definition}

The following observations are key for the hyperbolicity of the scalar polynomial lifts.

\begin{lemma}\label{lem:tree-pipe}
    Let $\ourpoly$ be a polynomial with $\{0,1\}$ coefficients and the degree of any term is at most $m$.
    Let $u, v \in \wt\G_\infty(\ourpoly)$ be in the same connected component.
    Now fix a monomial path from $u$ to $v$ and consider traversing it.
    The following hold:
    \begin{enumerate}[label=\roman{*})]
        \item Every time the monomial path leaves the $m$-neighborhood of the tree path between $u$ and $v$, it returns to the $m$-neighborhood at a point at most tree distance $2m$ away.
        \item For every vertex $w$ on the tree path from $u$ to $v$, there is a vertex $x$ in the $m$-neighborhood of $w$ such that $x$ is on the monomial path from $u$ to $v$.
    \end{enumerate}
\end{lemma}
\begin{proof}
    We first prove (i).
    Suppose that the monomial path leaves the $m$-neighborhood of the tree path at a point $x$. Note that $\treedist(x,v) > m$ since $x$ has to be outside a $m$-neighborhood of $v$. Let $y$ be the vertex such that $\treedist(x,y) = m$ and $y$ is on the tree path between $x$ and $v$, and let $y'$ be one step further than $y$ on the tree path between $x$ and $v$. Also note that $y'$ must be on the tree path between $u$ and $v$.

    When removed, the edge $(y,y')$ disconnects $V_\infty$. In particular, after removing the edge we can think of $y$ and $y'$ as the roots of two subtrees, the union of which includes all the vertices of $V_\infty$. $u$ is in the subtree of $y$ and $v$ is in the subtree of $y'$. Since the maximum degree of $p$ is $m$, any monomial-step can only get to a vertex at most $m$ in tree distance away from $x$. Therefore, any monomial step from $x$ remains in the subtree of $y$.

    In order to reach $v$, the monomial path has to enter the subtree of $y'$.
    Let $z$ be the last vertex visited by the monomial path in the subtree of $y$ before entering the subtree of $y'$.
    Since each monomial-step can traverse at most $m$ tree-steps, $z$ has to be tree-distance at most $m-1$ from $y$, and is hence in the $m$-neighborhood of $y'$. Finally, we note that $\treedist(x,z) \leq \treedist(x,y) + \treedist(y,z) \leq 2m$.

    We observe (ii) as a consequence of (i). At any vertex on the monomial path from $u$ to $v$, the next step either goes to a vertex at most $m$ away in the $m$-neighborhood of the tree path from $u$ to $v$, or it leaves the $m$ neighborhood and rejoins at a point at most $2m$ away.
\end{proof}

We can now prove the hyperbolicity of the connected components, where connectivity is defined based on the monomial paths.
\begin{proposition}\label{prop:hyperbolic}
    A graph arising from a polynomial lift has finite (depending on the polynomial $\ourpoly$) hyperbolicity in every connected component.
\end{proposition}
\begin{proof}
    We first prove the statement for scalar-coefficient $\ourpoly$, then make a reduction for general $\ourpoly$.

    Fix a scalar-coefficient $p$ with maximum degree $m$, and let $C$ be the maximum monomial distance between any two vertices which are connected and tree distance at most $2m$ apart.

    Let $u$, $v$, $w$ be vertices in $V_\infty$ from the same connected component.
    Let $x$ be on a monomial path from $u$ to $v$.
    By part (i) of \Cref{lem:tree-pipe}, $x$ is either in the $m$-neighborhood of the tree path from $u$ to $v$, or $x$ is on a path which leaves the $m$-neighborhood of the tree path but will rejoin at a point at most $2m$ away from where it left.
    In either case, $x$ is at most monomial-distance $C$ away from a vertex $x'$ which is in the $m$-neighborhood of the tree path from $u$ to $v$ and is on the monomial path from $u$ to $w$.
    Let $y$ be the point which is on the tree path from $u$ to $v$ and is closest in tree distance to $x'$.

    Since the union of the tree paths between $u$ and $v$, $v$ and $w$ and $u$ and $w$ is a tree, $y$ lies on either the tree path from $u$ to $w$ or from $v$ to $w$.
    Suppose without loss of generality $y$ is on the tree path from $u$ to $w$.
    Now fix a monomial path between $u$ to $w$. From part (ii) of \Cref{lem:tree-pipe} there is a vertex $y'$ such that $y'$ is on the monomial path between $u$ to $w$ and $\treedist(y',y) \leq m$. $x'$ and $y'$ are in the same connected component since they are on the monomial path between $u$ and $v$ and the monomial path between $u$ and $w$ respectively.
    We also know $\treedist(x', y') \leq \treedist(x',y) + \treedist(y,y') = 2m$, therefore $\dist(x',y') \leq C$, so $\dist(x,y') \leq 2C$. Therefore, the connected component that $u$, $v$, and $w$ belong to is $2C$-hyperbolic.

    Finally, for matrix-coefficient polynomials, we make a similar reduction as in \Cref{prop:connectivity}, by adding a generator and its inverse for each $i \in [r]$, which we call $h_1, \dots h_r$ and $h_1\inv, \dots, h_r\inv$.
    We use $V_\infty'$ to denote the free product of $V_\infty$ with another $r$ copies of $\Z$.
    Next, we make a transformation of $\ourpoly$ to a scalar-coefficient polynomial $\ourpoly'$.
    We write $\ourpoly$ as a sum of terms of the form $\ketbra{i}{j} X^w$, and for each such term, we add $X_j X^w X_i\inv$ to $\ourpoly'$.
    Given $(u, i)$, $(v, j)$ and $(w, k) \in V_\infty \x [r]$, we transform these vertices to $o$, $\ov{h_j u\inv h_i\inv}$ and $\ov{h_k w u\inv h_i\inv}$ in $V_\infty'$.
    This increases all tree distances by $2$, and leaves monomial distances unchanged.
    Therefore, the connected components of $\wt\G_\infty(\ourpoly{})$ are also hyperbolic.
\end{proof}

\subsection{Benjamini--Schramm convergence and unimodularity}
Next, we study some properties of the automorphism group of the graph $\wt\G_\infty(\ourpoly)$.

\begin{proposition}
Let $\ourpoly$ be a matrix polynomial with coefficients in $\{0,1\}^{r \x r}$. $\wt\G_\infty(\ourpoly)$ has at most $r$ vertex orbits.
\label{prop:finite-orbits}
\end{proposition}
\begin{proof}
Consider the graph of $\wt\G_\infty(\ourpoly)$, and recall that each vertex can be labeled $(v,c)$ for $v \in V_\infty$ and $c \in [r]$.
The natural covering map from $\wt\G_\infty(\ourpoly)$ to $\wt\G_1(\ourpoly)$ maps any two vertices labeled with the same $c$ in $\wt\G_\infty(\ourpoly{})$ onto the same vertex in $\wt\G_1(\ourpoly{})$.
For any two vertices $(u,c)$ and $(v,c)$, there must be some sequence $g$ a product of generators corresponding to the polynomial terms of $\ourpoly$ that such that $gu = v$.
Then, the map $P_g \otimes I_{r \x r}$ applied to $A_\infty(\ourpoly)$ brings $(u,c)$ to $(v,c)$, so these two vertices are in the same orbit.
Therefore, there can be at most $r$ orbits.
\end{proof}

We move on to another property that the automorphism groups of these infinite graphs have.
We say that a graph is \emph{unimodular} the Haar measure on the automorphism group is both left invariant and right invariant\footnote{unrelated to totally unimodular matrices}.
Intuitively, this condition means that the infinite graphs ``looks the same'' in any direction.
Unimodularity can be a useful property in characterizing graphs, for example, all unimodular trees are the universal covering tree of some finite graph~\cite{BK90}.

\begin{definition}
Let $G = (V,E)$ be a graph. We say that $f:V \x V \to [0,\infty]$ is \emph{diagonally invariant} if for all automorphisms $g \in \Aut(G)$, we have $f(gu,gv) = f(u,v)$.
In the case where we let the graph be a parameter of $f$, we also say $f$ is diagonally invariant if $f(G;u,v) = f(G;gu,gv)$ for all automorphisms $g \in \Aut(G)$.
\end{definition}

The following is the Mass-Transport Principle of Benjamini, Lyons, Peres and Schramm~\cite{BLPS99}, and in a vertex transitive graph, satisfying this condition is equivalent to unimodularity.

\begin{definition}
    Let $G = (V,E)$ be a vertex-transitive graph. Fix some vertex $o \in V$. We say that $G$ satisfies the \emph{Mass-Transport Principle} if for all diagonally invariant $f:V\x V \to [0,\infty]$,
    \[
        \sum_{x \in V} f(o,x) = \sum_{x \in V} f(x,o).
    \]
\end{definition}

The Mass-Transport Principle can be extended to non-vertex-transitive graphs.
For finite graphs, we can choose $o \in V$ uniformly at random and require that the equation hold in expectation. We can extend this to \emph{quasi-transitive} graphs (graphs which have a finite number of vertex orbits) as well. To do so, we have to introduce some machinery. For a graph $G$, we call the pair $(G,o)$ with $o \in G$ a \emph{rooted} graph. $(G,o)$ is isomorphic to $(G',o')$ if there is an isomorphism from $G$ to $G'$ bringing $o$ to $o'$. We write $[G,o]$ to denote the isomorphism class of $(G,o)$. Let $\calG$ be the space of isomorphism classes of locally finite graphs. We sometimes also need to consider isomorphism classes of graphs with two roots, which we will denote $[G; a, b]$. We call $\calG_\star$ the space of these isomorphism classes.

We can define a metric on $\calG$,
\[
    d([G,o],[G',o']) = \sup_{r > 0}\{2^{-r}: B([G,o],r) \text{ is isomorphic to } B([G',o'],r)\}.
\]

We write $\calG_D$ for the subspace of $\calG$ where the degrees of the vertices are uniformly bounded by $D$. Under the above metric, $\calG_D$ is compact metric space. In the following, we consider probability measures on the Borel $\sigma$-algebra generated by the metric.

To formulate the Mass-Transport Principle, one can try to extend the notion of picking a vertex uniformly at random to infinite graphs. The natural way to do this for a quasi-transitive graph is to pick a isomorphism class with probability proportional to the density of the type of the root. This is a probability measure $\mu$ on $\calG$.
This form of the Mass-Transport Principle is known as the Intrinsic Mass-Transport Principle due to Benjamini and Schramm~\cite{BS01}.
\begin{definition}
We say that a probability measure $\mu$ on $\calG$ satisfies the Mass-Transport Principle or is \emph{unimodular} if for all Borel $f:\calG_{\star} \to [0,\infty]$,
\[
    \E_{[G,o]\sim \mu} \bracks*{\sum_{x \in V(G)} f(G;o,x)} = \E_{[G,o]\sim\mu}\bracks*{ \sum_{x \in V(G)} f(G;x,o)}.
\]
\end{definition}

For quasi-transitive graphs, the Intrinsic Mass-Transport Principle is equivalent to unimodularity~\cite[Theorem 3.1]{AL07}, which explains the terminology. For a finite graph $G$, the measure supported on $[G, o]$ for $o \in G$ defined by
\[
    \mu_n[G,o] = \frac{|\Aut(G)\cdot o|}{|V(G)|}
\]
is unimodular.

One example of a non-unimodular graph is Trofimov's grandparent graph~\cite{Tro85}. The construction is as follows: Start with a 3-regular tree, and pick an end $\xi$. For each vertex, call the vertex one step in the direction of $\xi$ its parent, and the vertex two steps in the direction of $\xi$ its grandparent. Join each vertex to its grandparent with an undirected edge. This produces a 6-regular graph. Another class of examples are the Diestel--Leader graphs~\cite{DL01}.

Our next goal is to show that the graphs arising from matrix polynomials.
This gives us some slightly surprising examples of graphs that cannot be expressed by non-commutative polynomials.
For instance, the grandparent graph appears to have a repeating structure, and is vertex transitive, but nevertheless non-commutative matrix polynomials cannot express it.

A nice consequence of defining unimodularity via the Mass-Transport principle is that we can now consider limits of graphs as limits of probability measures on $\calG$.
For a single finite graph $G$, we can take the probability measure on equivalence classes of $[G,o]$ for $o \in G$ chosen uniformly at random, or in other words, $[G,o]$ is weighted according to the size of the vertex isomorphism class of $o$.
The notion of convergence we will use is that of weak convergence of measures.
Recall that we say a sequence of probability measures $\{\mu_n\}_{n = 1}^\infty$ \emph{converges weakly} or \emph{converges in distribution} to $\mu$ if for every continuous function $f$,
\[
    \lim_{n\to\infty} \E_{\mu_n}\bracks*{f}  = \E_\mu[f].
\]
In the literature, $\mu$ (or $G$ on which $\mu$ is supported) is sometimes known as the ``local limit'' or the convergence is known as ``convergence in the sense of Benjamini--Schramm''.
By having the support of $\mu$ on a random collection of $G$, we can also consider the limit of random graphs.

If each $\mu_n$ is supported on some finite graph $G_n$, this form of convergence intuitively says that with high probability, each radius-$r$-neighborhood of $G_n$ looks like a radius-$r$ neighborhood of $G$ appearing with probability according to $\mu$, where $r \to \infty$ as $n \to \infty$.

An important open question is whether every unimodular measure is such a limit of finite graphs~\cite[Question 10.1]{AL07}.

Rooted graphs of depth at most $k$ (for $k > 0$) which we call $\calG^{(k)}$ forms a base of $\calG$ (as a metric space). The finite graphs are all isolated points in this metric, thus these base elements are discrete. To show weak convergence in $\calG_D$ (i.e. when the degree is uniformly bounded), it suffices to show convergence in measure on every set in the base.

To show convergence on the base sets, it suffices to show that for every finite rooted graph $\Gamma$ and $k > 0$,
\[
    \lim_{n\to\infty} \Pr_{[G,o] \sim \mu_n}\bracks*{B(o,k) \cong \Gamma} = \Pr_{[G,o] \sim \mu} \bracks*{B(o,k)\cong\Gamma},
\]
where $B(o,k)$ denotes a $k$-neighborhood around the root of $G$.

One useful property is that unimodularity is preserved under weak convergence~\cite{AL07}, and we will use this to show the following:
\begin{proposition}\label{prop:unimodular}
Graphs arising from infinite lifts of matrix polynomials are unimodular.
\end{proposition}
\begin{proof}
We will do so by showing that the random finite lifts weakly converge to the measure on the infinite graph.
Let $\mu_n$ be the uniform distribution over randomly-rooted $n$-lifts.
Since this is a uniform distribution over unimodular measures, it's unimodular as well.

From \Cref{prop:finite-orbits}, there are at most $r$ vertex orbits in the infinite graph. Thus, picking a random root can be done by picking a representative corresponding to one of the $r$ base vertices uniformly. We call this probability measure $\mu$.

For every $k > 0$, with high probability as $n \to \infty$ none of the permutations have cycles smaller than $k$. In this case, for all finite rooted graphs $\Gamma$, $\Pr_{[G,o] \sim \mu_n}\bracks*{B(o,k) \cong \Gamma} = \Pr_{[G,o] \sim \mu} \bracks*{B(o,k)\cong\Gamma}$. Therefore, $\mu_n \to \mu$ in measure on the base sets.
\end{proof}

\subsection{Connectedness of finite lifts produced by our algorithm} \label{sec:connected}
We now turn to the question of whether the finite lifts produced from our derandomization of \Cref{thm:bc-main-intro} (with the full theorem statement in \Cref{thm:polynomials-hausdorff}) result in connected (extension) graphs when applied to a polynomial. To do so, we consider random walks on the graphs.

\begin{definition}
    Let $A$ be the adjacency matrix of a locally finite multigraph $G$ (indexed by pairs of vertices). We define the \emph{random walk matrix} of $G$ to be $P_G$ such that
    \[
        (P_G)_{uv} \coloneqq A_{uv}/\deg(u),
    \]
    where $\deg(u)$ is the degree of the vertex $u$.
\end{definition}

We need the following fact about $P_G$.
\begin{fact}\label{fact:connected-random-walk}
    For a finite multigraph $G$, the largest eigenvalue of $P_G$ is 1 and it has multiplicity 1 if and only if $G$ is connected.
\end{fact}

For spectral computations, it is convenient to analyze a symmetrized version of $P_G$.
\begin{fact}\label{fact:symmetrized-random-walk}
    Let $A$ be the adjacency matrix of a locally finite multigraph $G$. Let $D$ be the \emph{degree matrix}, i.e. the diagonal matrix where $D_{uu}$ is the degree of vertex $u$ in $A$. Then, $P_G = D\inv A$, and we define the symmetric matrix $S_G \coloneqq D^{-1/2}A D^{-1/2}$.

    One can check that if $\lambda$ is an eigenvalue of $S_G$ with eigenvector $v$, then $\lambda$ is an eigenvalue of $P_G$ with eigenvector $D^{1/2}v$, so in particular their spectra are equal.
\end{fact}

For an infinite graph $G$, it turns out that whether $\rho(P_G) = 1$ is characterized by an isoperimetric property: whether the graph is \emph{amenable}. Amenability was first defined as a property of groups, see, e.g.~\cite[Ch. 12]{Woe00}.
The definitions for amenability of Cayley graphs can also be extended to general graphs.
There are many equivalent definitions for amenability in graphs; the following is a statement about the \emph{isoperimetric constant} of the graph.

\begin{definition}\label{def:amenable}
    Let $G = (V,E)$ be a locally finite undirected graph. Let $S \subset V$ be a set of vertices, and we define the \emph{boundary} $\bdry S \subset E$ to be the set of edges $\{u,v\}$ with $u \in S$ and $v \notin S$. Let $m(S)$ be the sum of the vertex degrees in $S$. We say that $G$ is \emph{amenable} if
    \[
        \inf_{S \subset V,\,|S| < \infty}\frac{|\bdry S|}{m(S)} = 0.
    \]
    The quantity $\inf_{S \subset V,\,|S| < \infty}{|\bdry S|}/{|S|}$ is known as the \emph{isoperimetric constant} or \emph{isoperimetric number}.
\end{definition}
As some examples, every finite graph is amenable, an infinite path is amenable, while a $d$-regular infinite tree for $d \geq 3$ is not amenable.

The following was shown by Gerl~\cite{Ger88} (which also contains other equivalent characterizations of amenability), and is an extension the result on Cayley graphs by Kesten~\cite{Kes59}.
The result for $d$-regular graphs was also shown in~\cite{BMS88}.
\begin{proposition}\label{prop:amenable-rho}
    Let $G$ be a locally finite undirected graph. Then, $G$ is amenable if and only if \mbox{$\rho(P_G) = 1$}.
\end{proposition}

We can therefore conclude:
\begin{proposition}\label{prop:connected}
    Let $\ourpoly$ be a matrix polynomial with $r$-dimensional coefficients. Suppose that the extension of the $\infty$-lift $\wt\G_\infty(\ourpoly)$ is non-amenable, and that $\wt\G_1(\ourpoly)$ is connected. Then, there exists $\eps > 0$ such that applying \Cref{thm:polynomials-hausdorff} with $k$ the total degree of $\ourpoly$ and $R$ the maximum Frobenius norm of the coefficients of $\ourpoly$ gives a lift $\lift_n$ such that $A_n(\lift_n, \ourpoly{})$ satisfies the conclusion of \Cref{thm:polynomials-hausdorff} and $\wt\G_n(\lift_n,\ourpoly)$ is connected.
\end{proposition}
\begin{proof}
    Write $\ourpoly = \sum_{w \in \terms} \a_w X_w$, $A_n = A_n(\lift_n,\ourpoly)$ and let $D_n$ be the degree matrix of $A_n$. We will show that $S_{\wt\G_n} = D_n^{-1/2}A_nD_n^{-1/2}$ projected onto $\ket{+}_n$ has the eigenvalue $1$ with multiplicity $1$, and $\rho(S_{\wt\G_n,\bot}) < 1$ (where $S_{\wt\G,\bot}$ denotes the projection onto the orthogonal complement of $\ket{+}_n$).

    Recall that $\G_n(\lift_n, \ourpoly)$ is a $\terms$-regular graph. Therefore, each group of vertices in $\wt\G_n$ corresponding to the same vertex in $\G_n$ has the same set of degrees. Hence, $D_n$ is a block-diagonal matrix with identical blocks, in fact $D_n = I_n \otimes D_1$, where $D_1$ is the degree matrix of $\wt\G_1(\ourpoly)$. Therefore, we can express $S_{\wt\G_n}$ as the lift of another polynomial. If we define $\ourpoly' = \sum_{w \in \terms} D_1^{-1/2}\a_w D_1^{-1/2} X^w$, then $S_{\wt\G_n} = A_n(\lift_n,\ourpoly')$. Note that $\ourpoly'$ has the same degree as $\ourpoly$, each of the coefficients has Frobenius norm at most that of the corresponding coefficient in $\ourpoly$, and that $\wt\G_n(\lift_n,\ourpoly')$ is connected if and only if $\wt\G_n(\lift_n,\ourpoly)$ is connected.

    By \Cref{fact:symmetrized-random-walk,fact:connected-random-walk}, $\max \sigma(S_{\wt\G_n}) = 1$. Since we assume $\wt\G_1$ is connected, $S_{\wt\G_1}$ has maximum eigenvalue $1$ with multiplicity $1$. Using \Cref{prop:trivial-eigs}, the eigenvalue $1$ is in $\spec(S_{\wt\G_n}) \setminus \spec(S_{\wt\G_n,\bot})$. Hence, showing that $\norms[\big]{S_{\wt\G_n,\bot}} < 1$ ensures that $S_{\wt\G_n}$ has eigenvalue $1$ with multiplicity $1$.

    Let $\eps > 0$ be a parameter to be chosen later. We apply \Cref{thm:polynomials-hausdorff} with $r$ being the dimension of the coefficients of $\ourpoly$, $k$ the total degree and $R$ the maximum norm of the coefficients of $\ourpoly$, and $N$ of our choice, obtaining an $n$-lift $\lift_n$. From \Cref{prop:amenable-rho}, the assumption that $\wt\G_\infty$ is non-amenable implies $\rho(A_\infty(\lift_\infty, \ourpoly')) < 1$. If we choose $\eps$ small enough that $\rho(A_\infty(\calL_\infty, \ourpoly')) + \eps < 1$, then $\norms[\big]{S_{\wt\G_n, \bot}} < 1$, and we conclude that $\wt\G_n(\lift_n,\ourpoly')$ is connected, hence $\wt\G_n(\lift_n,\ourpoly)$ is connected.
\end{proof}

We will not attempt to characterize which polynomial lifts are non-amenable here, but refer the reader to some previous work done on amenable graphs.
The isoperimetric constant, and hence determining if a graph is amenable, is related to the \emph{growth rate}~\cite{Moh88} (in particular, non-amenable graphs have exponential growth).
Amenability also has connections to percolation processes~\cite{BLPS99}, and other spectral properties and properties of the random walk~\cite{Ger88,Kai92,Dod84,Sal92}.
There are also some combinatorial characterizations e.g.~on trees~\cite[Thm.~10.9]{Woe00}, \cite{FM17}, and on vertex-transitive graphs~\cite{SW90}.

\section{Additional analytic setup and definitions}
\label{sec:setup2}

\subsection{Nonbacktracking operators}
Throughout this section, let $\arcsz = \{0, 1, \dots, d+2\q\}$ be an index set, and write $\arcs = \arcsz \setminus \{0\}$ for the version with no identity-index.
Let $\ourpoly = \sum_{w \in \terms} a_w X^w$ be a self-adjoint matrix polynomial over~$\arcsz$, with coefficients in $\C^{r \times r}$,
let $\calK = a_0 + a_1 X_1 + \cdots + a_{d+2\q} X_{d+2\q}$ be a self-adjoint linear matrix polynomial over~$\arcsz$, with coefficients in $\C^{r \times r}$, and let $\lift_n = (\bbi, \sigma_1, \dots, \sigma_{d+2\q})$ be an $n$-lift, $n \in \N^+ \cup \{\infty\}$.
Let $\calG_n = \G_n(\lift_n,\calK)$ be the resulting $n$-lifted color-regular graph on vertex set~$V_n$.
Recall that the adjacency operator of~$\calG_n$ is
\[
    A_n = A_n(\lift_n, \calK) = \sum_{i \in \arcsz} P_{\sigma_i} \otimes \a_i.
\]

\begin{definition}[Nonbacktracking operator]
    The \emph{nonbacktracking operator} of $\calG_n$ is the following (in general non-Hermitian) bounded operator on $\ell_2(V_n) \otimes \ell_2(\arcs) \otimes \C^r$:
    \[
        B_n = B_n(\lift_n, \calK) = \sum_{i,j=1}^{d+2\q}  \bone[j \neq i^*] \cdot P_{\sigma_i} \otimes \ketbra{j}{i} \otimes \a_j.
    \]
    As $\ell_2(V_n) \otimes \ell_2(\arcs) \cong \ell_2(V_n \times \arcs)$, we may think of $B_n$ as arising from nonbacktracking steps along the colored arcs $V_n \times \arcs$ of~$\calG_n$.
\end{definition}

\begin{remark}
    We only define the nonbacktracking operators of $\calG_n$ arising as lifts of linear polynomials, and not general polynomials.
\end{remark}

\begin{remark}
    A warning: the summation defining~$B_n$ above does not include $i,j=0$; i.e., it excludes the identity-loop color in~$\calG_n$.
    Hence $B_n$ does not depend on~$\a_0$.
    The reason for this convention (chosen by Bordenave--Collins) is that in the needed technical theorems involving~$B_n$, we will be in a setup where $a_0 = 0$ anyway.
    In case~$\a_0 = 0$, the operator $B_n$ corresponds to the ``usual'' nonbacktracking operator of the extended graph~$\wt{\calG}_n$.
\end{remark}

\begin{notation}
    We use $B_\infty(\calK) = B_\infty(\lift_\infty,\calK)$ to denote the nonbacktracking operator of the $\infty$-lift of $\calK$ $A_\infty(\calK)$.
\end{notation}

\begin{notation}[Projection to the nontrivial subspace]
    As in \Cref{sec:projections},
    for a linear polynomial $p$ we write $B_{n,\bot}(\lift_n, \ourpoly)$ for the action of adjacency/nonbacktracking operators on $\ket{+}_n^\bot$.
\end{notation}

Finally, we record a kind of formula for the spectral radius of~$B_\infty$.
The following fact is proven within ``Proof of Lemma~14'' in Bordenave and Collins's work~\cite{BC19}:\footnote{Therein they write ``We deduce from Theorem~16(ii) that $\rho(B_\star)$ is equal to the spectral radius of~$L$'', but this is clearly a typo for ``the \emph{square-root of} the spectral radius of~$L$''; personal communication with the authors confirms this.}
\begin{proposition} \label{prop:sqrtL}
    For $i \in [d+2\q]$, define the $r^2 \times r^2$ matrix $A_i = a_i \otimes \overline{a}_i$, where $\overline{a}_i$ denotes the matrix formed by taking the complex conjugate of each entry in~$a_i$.
    Also, define $L$ to be the $(d+2\q) \times (d+2\q)$ block matrix, with blocks of size $r^2 \times r^2$, whose $(i,j)$th block is $1[j \neq i^*] \cdot A_i$.
    Then $\specrad(B_\infty) = \sqrt{\specrad(L)}$.
\end{proposition}
\begin{corollary}                                       \label{cor:rhoB-bound}
    If $\|a_i\|_\frob \leq R$ for all $i \in [d+2\q]$, then $\specrad(B_\infty) \leq \sqrt{d+2\q} R$.
\end{corollary}
\begin{proof}
    The claim follows from $\specrad(B_\infty)^4 = \specrad(L)^2 \leq \tr(L L^\conj)$ and
    \[
        \tr(L L^\conj) = \sum_{i=1}^{d+2\q} \sum_{j \neq i^*} \tr(A_i A_i^\conj) = \sum_{i=1}^{d+2\q} \sum_{j \neq i^*} \|a_i\|_\frob^2 \|\overline{a}_i\|_\frob^2 \leq (d+2\q)^2 R^{4}. \qedhere
    \]
\end{proof}

\subsection{Lifting permutations and product lifts}

We continue with the notation from the previous section, assuming henceforth that $n < \infty$.

\begin{definition}[The graph $G_{\lift_n}$ associated to a lift]
    Let $\sigma \in \symm{n}$ be a general permutation.
    Then  $P_{\sigma} + P_{\sigma^{-1}}$  is the adjacency matrix of an $n$-edge undirected graph on vertex set~$[n]$, which we will denote by~$G_\sigma$.
    Similarly, suppose $\sigma \in \symm{n}$ is a matching permutation.
    Then $P_\sigma$ alone is the adjacency matrix of an $n/2$-edge undirected graph on vertex set~$[n]$; abusing notation slightly, we will also denote it by~$G_\sigma$.
    Now supposing $\lift_n = (\Id, \sigma_1, \dots, \sigma_{d+2\q})$ is an $n$-lift, we define an associated undirected (multi)graph $G_{\lift_n} = G_{\sigma_1} + \cdots + G_{\sigma_d} + G_{\sigma_{d+1}} + \cdots + G_{\sigma_{d+2\q}}$.
    More precisely, $G_{\lift_n}$ is the graph on vertex set~$V_n$ whose adjacency matrix is $\sum_{i = 1}^{d+2\q} P_{\sigma_i}$.
    Note that we have omitted the identity-index $i = 0$ here, so $G_{\lift_n}$ has no self-loops.
    Equivalent definitions are that $G_{\lift_n} =\wt{\G}_n(\lift_n, \calK)$ where $\calK$ is the linear matrix polynomial $\sum_{j \in \arcs} X_i$ (with $r = 1$).
\end{definition}

\begin{definition}[The lift of a permutation/matching]
    Suppose again that $\sigma \in \symm{n}$ is a general permutation.
    Given a sequence $\mathfrak{T}_m$ of permutations $\tau_0 = \Id, \tau_1, \dots, \tau_{2n} \in \symm{m}$ with $\tau_{j+n} = \tau_j^{-1}$ (for $j \in [n]$), we can form the lifted extension graph $\wt{\G}_m(\mathfrak{T}_m, \calK_{G_\sigma})$.
    This graph is of the form $G_\rho$ for a permutation $\rho \in \symm{mn}$, and we refer to $(\rho, \rho^{-1})$ as the \emph{permutation pair $(\sigma, \sigma^{-1})$ lifted by~$\mathfrak{T}_m$}.
    Identifying $\C^{mn} = \C^m \otimes \C^n$ we have $P_\rho = \sum_{j=1}^n P_{\tau_j} \otimes \ketbra{\sigma(j)}{j}$.

    If $\sigma \in \symm{n}$ is a matching permutation, we may do something similar.
    In this case, $P_\sigma$ alone is the adjacency matrix of an $n/2$-edge undirected graph on vertex set~$[n]$, which we also denote by~$G_\sigma$.
    Given a sequence $\mathfrak{T}_m$ of permutations $\tau_0 = \Id, \tau_1, \dots, \tau_{n} \in \symm{m}$ with $\tau_{j+n/2} = \tau_i^{-1}$ (for $j \in [n/2]$), we can again form the lifted extension graph $\wt{\G}_m(\mathfrak{T}_m, \calK_{G_\sigma})$.
    This graph is a perfect matching, and hence may be viewed as $G_\rho$ for $\rho \in \symm{kn}$ a matching permutation.
    We again refer to $\rho$ as the \emph{matching permutation $\sigma$ lifted by~$\mathfrak{T}_m$}.
\end{definition}

\begin{definition}[Product lift]
    Let $\lift_n = (\Id, \sigma_1, \dots, \sigma_{d+2\q})$ denote an $n$-lift.
    For $m \in \N^+$, we can produce an $mn$-lift $(\rho_0 = \Id, \rho_1, \dots, \rho_{d+2\q})$ as follows.
    For for $1 \leq i \leq d$ assume that $\mathfrak{T}_m^{(i)} = (\tau^{(i)}_1, \dots, \tau^{(i)}_{n})$ is a sequence of permutations in $\symm{m}$ with $\tau^{(i)}_{j+n/2} = (\tau^{(i)}_j)^{-1}$; and, for $d+1 \leq i \leq d+\q$ assume that $\mathfrak{T}_m^{(i)} = (\tau^{(i)}_1, \dots, \tau^{(i)}_{2n})$ is a sequence of permutations in $\symm{m}$ with $\tau^{(i)}_{j+n} = (\tau^{(i)}_j)^{-1}$.
    Now for $1 \leq i \leq d$, let $\rho_i$ be the matching permutation $\sigma_i$ lifted by $\mathfrak{T}_m^{(i)}$; and, for $d+1 \leq i \leq d+\q$, let $(\rho_{i}, \rho_{i+\q})$ be the permutation pair $(\sigma_i, \sigma_{i+\q})$ lifted by $\mathfrak{T}_m^{(i)}$.
    We write $\mathfrak{T}_m \otimes \lift_n$ for the $mn$-lift $(\rho_i)_i$, calling it the \emph{product} of the $m$-lifts $\mathfrak{T}_m^{(i)}$ with the $n$-lift~$\lift_n$.
    Note that $G_{\mathfrak{T}_m \otimes \lift_n}$ is an $m$-fold graph lift of $G_{\lift_n}$ (of the usual sort).
\end{definition}

\subsection{Signed permutations}    \label{sec:signed-perms}
\begin{notation}
    We continue in the setup of the previous section, but restrict attention to the case of  $m = 2$.
    In this case, given the $2$-lifts $\mathfrak{T}_2^{(i)} = (\tau^{(i)}_j)_j$ (for $1 \leq i \leq d$, $1 \leq j \leq n$ and $d+1 \leq i \leq d+\q$, $1 \leq j \leq 2n$) we may write
    \[
        P_{\tau^{(i)}_{j}} = \frac12\begin{bmatrix} +1 & +1 \\ +1 & +1 \end{bmatrix} + \chi_i(j) \cdot \frac12\begin{bmatrix} +1 & -1 \\ -1 & +1 \end{bmatrix} = \ketbra{+}{+} + \chi_i(j) \cdot \ketbra{-}{-},
    \]
    where $\chi_{1}, \dots, \chi_{d} : [n] \to \{\pm 1\}$  and $\chi_{d+1}, \dots, \chi_{d+\q} : [2n] \to \{\pm 1\}$ and where we use the notation
    \[
        \ket{+} = \ket{+}_2 = \tfrac{1}{\sqrt{2}} (\ket{1} + \ket{2}), \qquad \ket{-} = \tfrac{1}{\sqrt{2}} (\ket{1} - \ket{2}).
    \]
\end{notation}
\begin{definition}[edge-signing/signed permutation]
    Recall that we have $\chi_i(j+n) = \chi_i(j)$ for $d+1 \leq i \leq d+\q$ and thus it is natural to treat these $\chi_i$'s as functions $[n] \to \{\pm 1\}$.
    Similarly  for $1 \leq i \leq d$ we have $\chi_i(j+n/2) = \chi_i(j)$ and thus it is natural to treat these $\chi_i$'s as functions $[n/2] \to \{\pm 1\}$.
    In this form, we will identify $\chi_{1}, \dots, \chi_{d}$ as \emph{edge-signings} of the $1$-regular matching graphs $G_{\sigma_{1}}, \dots, G_{\sigma_{d}}$, and identify $\chi_{d+1}, \dots, \chi_{d+\q}$ as edge-signings of the $2$-regular graphs $G_{\sigma_{d+1}}, \dots, G_{\sigma_{d+\q}}$.
    Relatedly, we may think of each $\chi_i\sigma_i$ as a \emph{signed permutation} (i.e., a member of the hyperoctahedral group of order~$n$).
    Collectively, $\chi = (\chi_1, \dots, \chi_d, \chi_{d+1}, \dots, \chi_{d+\q})$ is an edge-signing of the graph~$G_{\lift_n}$.
\end{definition}

\begin{notation}
    We will use $P_{\chi_i \sigma_i}$ to denote the signed permutation matrix naturally associated to the signed permutation $\chi_i \sigma_i$. In the context of a matrix polynomial $\ourpoly$, given a monomial $X^w$, we use $P_{\chi\sigma^w}$ to denote the signed permutation matrix associated to the signed permutation $(\chi \sigma)^w$ given by substituting $X_j$ with $\chi_j\sigma_j$ for each $j \in \arcs$. We also write $\chi^w$ to denote the signing obtained by substituting $X_j$ with $\chi_j$ for each $j \in \arcs$.
\end{notation}
\begin{notation}
    Let $\ourpoly = \sum_{w \in \terms} \a_w X^w$ be a matrix polynomial.
    We write $A_n(\chi \lift_n, \ourpoly)$ for the \emph{signed adjacency operator} of the \emph{sign-lifted color-regular graph} $\G_n = \G_n(\chi \lift_n, \ourpoly)$, the variant of $\G_n(\lift_n, \ourpoly)$ in which the $w$-colored edge emanating from~$u \in V_n$ is also signed by~$\chi^w(u)$.
\end{notation}
\begin{remark}
    One might also interpret this as changing the matrix-weight on the edge from $\a^w$ to $\chi^w(u) \a^w$, but one must be careful \emph{not} to use this interpretation for the signed nonbacktracking operator, as we will see.
\end{remark}

Notice that for $1 \leq i \leq d+2\q$, the $i$th permutation $\rho_i$ in the composition lift $\mathfrak{T}_2 \otimes \lift_n$ satisfies
\begin{align*}
    P_{\rho_i} = \sum_{j=1}^n P_{\tau^{(i)}_j} \otimes \ketbra{\sigma_i(j)}{j}
    &= \sum_{j=1}^n \Bigl(\ketbra{+}{+} + \chi_i(j) \cdot \ketbra{-}{-}\Bigr) \otimes \ketbra{\sigma_i(j)}{j}  \\
    &= \sum_{j=1}^n \ketbra{+}{+} \otimes \ketbra{\sigma_i(j)}{j} + \sum_{j=1}^n \chi_i(j) \cdot \ketbra{-}{-} \otimes \ketbra{\sigma_i(j)}{j}  \\
    &= \ketbra{+}{+}  \otimes  P_{\sigma_i} + \ketbra{-}{-} \otimes P_{\chi_i \sigma_i}.
\end{align*}
Thus
\begin{align}
    A_{2n}(\mathfrak{T}_2 \otimes \lift_n, \ourpoly)
    &= \ketbra{+}{+} \otimes \parens*{\sum_{w \in \terms} P_{\sigma^w} \otimes \a_w}  + \ketbra{-}{-} \otimes \parens*{\sum_{w \in \terms} P_{\chi_i \sigma^w} \otimes \a_i} \nonumber \\
    &= \ketbra{+}{+} \otimes A_n(\lift_n, \calK) + \ketbra{-}{-} \otimes A_n(\chi \lift_n, \calK), \label{eqn:A-spec}
\end{align}
From these, we deduce
\begin{proposition}\label{prop:A-spec-identities}
    The following multiset identities hold:
    \begin{enumerate}
    \item $\spec(A_{2n}(\mathfrak{T}_2 \otimes \lift_n, \ourpoly)) = \spec(A_n(\lift_n,\ourpoly)) \cup \spec(A_{n,\chi}(\lift_n,\ourpoly))$, \label{item:A-spec}
    \item $\spec(A_{2n,\bot}(\mathfrak{T}_2 \otimes \lift_n, \ourpoly)) = \spec(A_{n,\bot}(\lift_n,\ourpoly)) \cup \spec(A_{n,\chi}(\lift_n,\ourpoly))$. \label{item:A-spec-perp}
    \end{enumerate}
\end{proposition}
\begin{proof}
    \Cref{item:A-spec} follows immediately from \Cref{eqn:A-spec}.
    Using $\ket{+} \otimes \ket{+}_n = \ket{+}_{2n}$, we have
    \begin{equation}\label{eqn:A-spec2}
        \ketbra{+}{+} \otimes A_n(\lift_n, \ourpoly) = \sum_{w \in \terms} \parens*{\ketbra{+}{+} \otimes P_{\sigma^w}} \otimes \a_i = \sum_{w \in \terms}\ket{+}_{2n}\bra{+}_{2n} \otimes \a^w + \sum_{w \in \terms} (0 \oplus P_{\sigma^w,\bot})\otimes \a^w,
    \end{equation}
    and combining this with \Cref{item:A-spec} gives us \Cref{item:A-spec-perp}
\end{proof}

Similarly, in the case of a linear polynomial $\calK$, we have
\begin{proposition}\label{prop:B-spec-identities}
    The following multiset identities hold:
    \begin{enumerate}
    \item $\spec(B_{2n}(\mathfrak{T}_2 \otimes \lift_n, \calK)) = \spec(B_n(\lift_n,\calK)) \cup \spec(B_{n,\chi}(\lift_n,\calK))$, \label{item:B-spec}
    \item $\spec(B_{2n,\bot}(\mathfrak{T}_2 \otimes \lift_n, \calK)) = \spec(B_{n,\bot}(\lift_n,\calK)) \cup \spec(B_{n,\chi}(\lift_n,\calK))$. \label{item:B-spec-perp}
    \end{enumerate}
\end{proposition}
\begin{proof}
Similarly to the computations for $A_n$, we have
\begin{align}
    B_{2n} = B_{2n}(\mathfrak{T}_2 \otimes \lift_n, \calK)
    &= \ketbra{+}{+} \otimes \parens*{\sum_{i,j=1}^{d+2\q}  \bone[j \neq i^*] \cdot  P_{\sigma_i} \otimes \ketbra{j}{i} \otimes \a_j}   \nonumber\\
    &\,{} + \ketbra{-}{-} \otimes \parens*{\sum_{i,j=1}^{d+2\q}  \bone[j \neq i^*] \cdot P_{\chi_i \sigma_i} \otimes \ketbra{j}{i} \otimes \a_j}  \nonumber\\
    &{} = \ketbra{+}{+} \otimes B_n(\lift_n, \calK)  + \ketbra{-}{-}\otimes B_n(\chi \lift_n, \calK), \label{eqn:B-spec}
\end{align}
where here we've introduced the notation
\begin{equation} \label{eqn:signedB}
    B_{n,\chi} = B_n(\chi \lift_n, \calK) = \sum_{i,j=1}^{d+2\q}  \bone[j \neq i^*] \cdot P_{\chi_i \sigma_i} \otimes \ketbra{j}{i} \otimes \a_j
\end{equation}
for the \emph{signed nonbacktracking operator} of the sign-lifted color-regular graph $\G_n(\chi \lift_n, \calK)$.
Notice here that the action of $B_n(\chi \lift_n, \calK)$ on a colored arc $(u,i)$ picks up the \emph{sign} $\chi_i(u)$ on the ``stepped-from'' arc~$(u,i)$, but picks up the \emph{matrix-weights}~$\a_j$ on the ``stepped-to'' arcs $(v,j)$ (for $v = \sigma_i(u)$ and $j \neq i^*$).
(This is why one shouldn't automatically think of an edge-signing simply as changing the matrix-weight on colored edge $(u,i)$ from~$\a_i$ to $\chi_i(u)\a_i$.)
\Cref{eqn:B-spec} allows us to conclude \Cref{item:B-spec}.
We can further compute,
\begin{multline*}
    \ketbra{+}{+} \otimes B_n(\lift_n, \calK)
    = \sum_{i,j=1}^{d+2\q}  \bone[j \neq i^*] \cdot  \Bigl(\ketbra{+}{+}  \otimes P_{\sigma_i}\Bigr) \otimes \ketbra{j}{i} \otimes \a_j \\
    = \sum_{i,j=1}^{d+2\q}  \bone[j \neq i^*] \cdot  \ket{+}_{2n}\!\bra{+}_{2n}\otimes \ketbra{j}{i} \otimes \a_j
     \ \ +\ \  \ketbra{+}{+}   \otimes \sum_{i,j=1}^{d+2\q}  \bone[j \neq i^*] \cdot  (0 \oplus P_{\sigma_i,\bot}) \otimes \ketbra{j}{i} \otimes \a_j,
\end{multline*}
and hence from \Cref{item:B-spec} we may further conclude \Cref{item:B-spec-perp}.
\end{proof}

\section{Derandomization tools}

We require the following standard derandomization tools:

\begin{definition}
    Let $\delta \in [0,1]$ and $t \in \N^+$.  A sequence of random bits $\bx = (\bx_1, \dots, \bx_n) \in \{\pm 1\}^n$ is said to be \emph{$(\delta,t)$-wise uniform} if, for every $S \subseteq [n]$ with $0 < |S| \leq t$, it holds that $\abs{\E[\prod_{i \in S} \bx_i]} \leq \delta$.
\end{definition}

The key property of such sequences we use is the following:
\begin{fact}                                        \label{fact:t-wise-bits}
    Let $q \in \C[X_1, \dots, X_n]$ be a (usual) polynomial of degree at most~$t$ such that the sum of the magnitudes of the coefficients of~$q$ is~$M$.  Then if $\bx \in \{\pm 1\}^n$ is $(\delta,t)$-wise uniform and $\bu \in \{\pm 1\}^n$ is truly uniformly random, then  $\abs{\E[q(\bx)] - \E[q(\bu)]} \leq \delta M$.
\end{fact}

A classic result is that $(\delta,t)$-wise uniform bits can be  strongly explicitly constructed from a truly random seed of length $O(\log t  + \log \log n + \log(1/\delta))$:
\begin{theorem}                                  \label{thm:nn93}
    (\cite{NN93,AGHP92,Sho90}.)  There is a deterministic algorithm that, given $\delta$,  $t$, and $n$, runs in time $\poly(n/\delta)$ and outputs a multiset $X \subseteq \{\pm 1\}^n$ of cardinality $S = \poly(t \log(n)/\delta)$ (a power of~$2$) such that, for $\bx \sim X$ chosen uniformly at random, the sequence $\bx$ is $(\delta,t)$-wise uniform.  Indeed, if the algorithm is additionally given $1 \leq s \leq S$ and $1 \leq i \leq n$ (written in binary), it can output the $i$th bit of the $s$th string in~$X$ in deterministic time $\polylog(n/\delta)$.
\end{theorem}

\begin{definition}
    Let $t \in \N^+$ and let $[n]_t$ denote the set of all sequences of~$t$ distinct indices from~$[n]$.  A random permutation $\bpi \in \symm{n}$ is said to be \emph{$t$-wise uniform} if, for every sequence $(i_1, \dots, i_t) \in [n]_t$, the distribution of $(\bpi(i_1), \dots, \bpi(i_t))$ is uniform on~$[n]_t$. For $\delta \in [0,1]$, we say that $\bpi$ is \emph{$\delta$-almost $t$-wise uniform} if its distribution is $\delta$-close in total variation distance to that of a truly $t$-wise uniform distribution~$\bpi'$.
\end{definition}

Again, the key property of such permutations we use is the following:
\begin{fact}                                        \label{fact:t-wise-perm}
    Let $q$ be a (usual) polynomial of degree at most~$t$ in indicator random variables $(1[\bpi(i) = j])_{i,j=1}^n$.  Then $q$ has the same expectation under a $t$-wise uniform $\bpi \in \symm{n}$ as it has under a truly uniform $\bpi \in \symm{n}$.
\end{fact}

Combining the strongly explicit $(\delta,t)$-wise uniform permutations of \cite{Kas07,KNR09} with the main theorem in~\cite{AL13}, we obtain the following theorem (called ``Corollary~2.6'' in~\cite{MOP20b}):
\begin{theorem}                                       \label{thm:prg-perm}
    (\cite{KNR09,Kas07,AL13}) There is a deterministic algorithm that, given $t$ and~$n$, runs in time $\poly(n^t)$ and outputs a multiset $\Pi \subseteq \symm{n}$ (closed under inverses) of cardinality $S = \poly(n^t)$ (a power of~$2$) such that, when $\bpi \sim \Pi$ is chosen uniformly at random, $\bpi$ is $n^{-100t}$-almost $t$-wise uniform. Indeed, if the algorithm is additionally given $1 \leq s \leq S$ and $1 \leq i \leq n$ (written in binary), it can output  $\pi_s(i)$ and $\pi_s^{-1}(i)$ (where $\pi_s$ is the $s$th permutation in~$\Pi$) in deterministic time $\poly(t \log(n/\delta))$.
\end{theorem}

\begin{definition}
    For $n$ even, we say a general permutation $\pi \in \symm{n}$ has the \emph{associated matching} $\sigma \in \symm{n}$, where $\sigma$ matches (i.e., has as a $2$-cycle) the pairs $(\pi(1), \pi(2)), \dots, (\pi(n-1), \pi(n))$.  We note that when $\bpi \sim \symm{n}$ is uniformly random, its associated matching $\bsigma$ is uniformly distributed among matching permutations.
\end{definition}

\begin{remark}  \label{rem:sim}
    If $\pi \in \symm{n}$ has associated matching $\sigma \in \symm{n}$, each indicator $1[\sigma(k) = \ell]$ is a degree-$2$ polynomial in the indicators $1[\pi(i) = j]$.
\end{remark}

\begin{definition}  \label{def:pr-lift}
    Given an index set $\arcs = \{0, 1, \dots, d+2\q\}$, we say a random $n$-lift $\blift_n = (\bbi, \bsigma_1, \dots, \bsigma_{d+2\q})$ is \emph{$\delta$-almost $t$-wise uniform} if:
    \begin{itemize}
        \item the permutations $\bsigma_1, \dots, \bsigma_d, \bsigma_{d+1}, \dots, \bsigma_{d+\q}$ are independent;
        \item each of $\bsigma_{1}, \dots, \bsigma_{d}$ is the associated matching of a $\frac{\delta}{d+\q}$-almost $2t$-wise uniform permutation;
        \item each of $\bsigma_{d+1}, \dots, \bsigma_{d+\q}$ is a $\frac{\delta}{d+\q}$-almost $t$-wise uniform permutation.
    \end{itemize}
    When $\delta = 0$, we simply say that $\blift_n$ is \emph{$t$-wise uniform}.  Note that a $\delta$-almost $t$-wise uniform $n$-lift is $\delta$-close, in total variation distance, to a $t$-wise uniform $n$-lift.
\end{definition}

The $2t$-wise requirement in the second bullet of this definition was chosen so that, in light of \Cref{rem:sim}, we could make the following observation:
\begin{fact}                                        \label{fact:t-wise-lift}
    In the notation of \Cref{def:pr-lift}, suppose $q$ is a polynomial of degree at most~$t$ in the indicator random variables $(1[\bsigma_i(j) = k])_{i \in \arcsz}^{j,k=1 \dots n}$.  Then $q$ has the same expectation under a $t$-wise uniform $n$-lift~$\blift_n$ as it has under a truly uniform $n$-lift $\blift_n$.
\end{fact}

Finally, applying \Cref{thm:prg-perm} and adjusting constants, we conclude the following:
\begin{theorem}                                     \label{thm:prg-lift}
    Fix an index set $\arcs$.
    There is a deterministic algorithm that, given $t$ and~$n$, runs in time $\poly(n^t)$ and outputs a multiset $\Lambda$ of $n$-lifts (indexed by~$\arcs$) such that, when $\blift_n \sim \Lambda$ is chosen uniformly at random, $\blift_n$ is $n^{-100t}$-close in total variation distance to being a $t$-wise uniform random $n$-lift.
    Indeed, if the algorithm is additionally given $i \in \arcsz$, $1 \leq s \leq S$, and $1 \leq j \leq n$ (written in binary), it can output  $\sigma^{(s)}_i(j)$ in deterministic time $\poly(t \log(n/\delta))$ (where $\sigma^{(s)}_i$ is the $i$th permutation of the $s$th lift in~$\Lambda$).
\end{theorem}
(Note that we rely on \Cref{thm:prg-perm}'s ability to compute permutation-inverses strongly explicitly, not just for computing $\sigma^{(s)}_i$ with $d+\q+1 \leq i \leq d+2\q$, but also for computing the matching permutations $\sigma^{(s)}_i$ with $1 \leq i \leq d$ strongly explicitly.
E.g., for such a matching index, we may efficiently compute $\sigma_i(j)$ from its underlying associated general permutation~$\pi_i$ by first letting $k = \pi_i^{-1}(j)$, then setting $k'$ to be $k+1$ if $k$ is odd or $k-1$ if $k$ is even, and finally computing $\sigma_i(j) = \pi_i(k')$.)

\section{Random 2-lifts}
\label{sec:random-2lift}

Throughout this section we fix an index set $\arcs = \{1, \dots, d+2\q\}$ with no identity-index (as we will be working exclusively with nonbacktracking operators).
We will be considering linear matrix polynomials over this index set, $a_1 X_1 + \cdots + a_{d+2\q} X_{d+2\q}$, where $a_j \in \C^{r \times r}$.
We will abuse our own terminology very slightly by referring to these as ``bouquets''~$\calK = (a_1, \dots, a_{d+2\q})$, the abuse being the possibility of $a_i = 0$.
In fact, it will not even be an abuse, because we will focus on bouquets whose matrix-weights are bounded and have bounded inverses (and therefore are nonzero).

\subsection{A net of bouquets}

\begin{definition}
    Given $R > 0$ (and implicitly $\arcs$ and~$r$), we define $\mathfrak{K}_R$ to be the collection of all matrix bouquets $(a_1, \dots, a_{d+2\q}) \in (\C^{r \times r})^{d+2\q}$ with the property that each~$a_i$ is \emph{$R$-bounded}, meaning that both $\norm{a_i}_\frob \leq R$ and $\|a_i^{-1}\|_\frob \leq R$.
\end{definition}

The below $\eps$-net result is mostly  proven in~\cite[Sec.~4.5]{BC19}.  We remark that for our main \Cref{thm:main} we will only need the $\lambda = 1$ case, but for the weak derandomization of~\cite{BC19} in \Cref{sec:weak-derand} we need the $\lambda = \Theta(\log n)$ case.
\begin{proposition}                                     \label{prop:eps-net}
    Given constants $d, \q, r, \eps > 0$ and $R > 1$, there is a large constant $\kappa = \kappa(d, \q, r, R, \eps)$ such that the following holds:  Given $\lambda \in \N^+$, there is an (efficiently computable)  net~$\Xi_\lambda \subset \mathfrak{K}_{2R}$ with the following properties:
    \begin{itemize}
        \item $|\Xi_\lambda| \leq \kappa^{\lambda}$, and each element of $\Xi_\lambda$ has ``encoding length''\footnote{Meaning, number of bits needed to represent the matrix, with the real and complex part of each entry being stored as the ratio of two integers.}
            at most~$\kappa \lambda$.
        \item For every $R$-bounded bouquet $\calK \in \mathfrak{K}_R$, there exists $\ul{\calK} \in \Xi_\lambda$ such that both:
            \begin{itemize}
                \item $\Bigl|\norm{B_n(\chi \lift_n, \calK)^\lambda}_\opnorm - \norm{B_n(\chi \lift_n, \ul{\calK})^\lambda}_\opnorm\Bigr| \leq \eps^\lambda$ holds for every finite signed lift $\chi \lift_n$;
                \item $\Bigl|\rho(B_\infty(\calK)) - \rho(B_\infty(\ul{\calK}))\Bigr| \leq \eps$.
            \end{itemize}
    \end{itemize}
\end{proposition}
As this exact statement does not precisely appear in \cite{BC19}, we provide a proof (which closely follows that in \cite[Sec.~4.5]{BC19}).
\begin{proof}
    Let $\mathfrak{F}_R$ denote the set of all matrices in $\C^{r \times r}$ with Frobenius norm at most~$R$, and let $\mathfrak{F}_R^H$ denote the subset of Hermitian such matrices.  For any fixed rational $\delta > 0$ we can create a ``grid'' $\mathfrak{G}_\delta$ of all matrices in~$\mathfrak{F}_R$ with entries being complex integer multiples of~$\delta$, such that: (i)~$|\mathfrak{G}_\delta| \leq O(rR/\delta)^{r^2}$; (ii)~for every $b \in \mathfrak{F}_R$ there is $\ul{b} \in \mathfrak{G}_\delta$ with $\|b - \ul{b}\|_\frob \leq \delta$.  (Also, (ii)~holds for~$\mathfrak{F}_R^H$ when restricting to $\mathfrak{G}_\delta^H = \mathfrak{G}_\delta \cap \mathfrak{F}_R^H$.)  Suppose now that $b\inv \in \mathfrak{F}_R$ as well.  Then provided $\delta < \frac{1}{2R}$ (which we will certainly ensure) we have, writing $\Delta = \ul{b} - b$,
    \begin{align*}
          \norm{b\inv - \ul{b}\inv}_\frob &= \norm{b\inv(\Delta b\inv + (\Delta b\inv)^2 + (\Delta b\inv)^3 + \cdots}_\frob \\
          &\leq \|b\inv\|_\frob (\norm{\Delta}_\frob \|b\inv\|_\frob)/(1- \norm{\Delta}_\frob \|b\inv\|_\frob) \leq 2\delta R^2 \leq R;
    \end{align*}
    hence $\|\ul{b}\inv\|_\frob \leq 2R$, and therefore $\ul{b}\inv \in \mathfrak{F}_{2R}$.  For a suitable $\delta > 0$ to be chosen later we will define
    \begin{equation}    \label{eqn:Xi}
        \Xi_\lambda = \{(\ul{b}_{1}, \dots, \ul{b}_{d}, \ul{b}_{d+1}, \dots, \ul{b}_{d+\q}, \ul{b}_{d+1}^\conj, \dots, \ul{b}_{d+\q}^\conj) : \ul{b}_{1}, \dots, \ul{b}_{d} \in \mathfrak{G}_\delta^H \text{ and } \ul{b}_{d+1}, \dots, \ul{b}_{d+\q} \in \mathfrak{G}_\delta\}.
    \end{equation}
    A key observation is now that for any bouquet $\calK = (a_1, \dots, a_{d+2\q}) \in \mathfrak{K}_R$ there is a bouquet $\ul{\calK} = (\ul{a}_1, \dots, \ul{a}_{d+2\q}) \in \Xi$ with
    \[
        \norm{a_i - \ul{a}_i}_\opnorm \leq \norm{a_i - \ul{a}_i}_\frob \leq \delta \quad \forall i \in [d+2\q];
    \]
    hence from the definition \Cref{eqn:signedB} and the bounds $\|P_{\chi_i \sigma_i}\|_\opnorm \leq 1$ and $\|\ketbra{j}{i}\|_\opnorm \leq 1$, we conclude
    \[
        \Bigl|\norm{B_n(\chi \lift_n, \calK)}_\opnorm - \norm{B_n(\chi \lift_n, \ul{\calK})}_\opnorm\Bigr| \leq (d+2\q)^2 \cdot \delta
    \]
    for every finite signed lift $\chi \lift_n$.    This bound already suffices for the $\lambda = 1$ case; more generally, using $\|B^\lambda - \ul{B}^\lambda\| \leq \lambda \cdot \max\{\norm{B}, \norm{\ul{B}}\}^{\lambda-1} \cdot \norm{B - \ul{B}}$, and also the bounds
    \[
        \norm{B_n(\chi \lift_n, \calK)}_\opnorm,\ \norm{B_n(\chi \lift_n, \ul{\calK})}_\opnorm \leq (d+2\q)^2 \cdot R,
    \]
    we conclude
    \begin{equation}    \label[ineq]{eqn:light1}
        \Bigl|\norm{B_n(\chi \lift_n, \calK)^\lambda}_\opnorm - \norm{B_n(\chi \lift_n, \ul{\calK})^\lambda}_\opnorm\Bigr| \leq \lambda ((d+2\q)^{2} R)^\lambda \cdot \delta.
    \end{equation}

    Finally, it is not hard to verify that for the matrix ``$L$'' defined in \Cref{prop:sqrtL}, and the analogously defined ``$\ul{L}$'', we have
    \[
        \|L - \ul{L}\|_\opnorm \leq \|L - \ul{L}\|_\frob \leq (d+2\q) \cdot 2 \delta R,
    \]
    and also that $\|L\|_\opnorm, \|\ul{L}\|_\opnorm \leq (d+2\q) R^2$.  Hence by \cite[Thm.~VIII.1.1]{Bha97} we get
    \begin{multline}
        \abs{\rho(L) - \rho(\ul{L})} \leq (2 (d+2\q) R^2)^{1-1/r} ((d+2\q) \cdot 2 \delta R)^{1/r} \leq 2(d+2\q)R \cdot \delta^{1/r} \\
        \implies \quad \Bigl|\rho(B_\infty(\calK)) - \rho(B_\infty(\ul{\calK}))\Bigr| \leq \sqrt{2(d+2\q)R} \cdot \delta^{1/(2r)}. \label[ineq]{eqn:light2}
    \end{multline}
    In light of \Cref{eqn:Xi,eqn:light1,eqn:light2}, we may be complete the proof by taking $\delta = c(d, \q, r, R, \eps)^\lambda$ for a suitably small rational constant $c(d,\q,r,R,\eps) > 0$, and then taking $\kappa = \kappa(c, d, \q, r, R, \eps)$ a large enough constant.
\end{proof}

\subsection{Our main technical theorem}

Our goal for the remainder of the section is to prove \Cref{thm:main} below, which we state after making a few preliminary definitions.
\begin{definition}
    A \emph{color sequence $\gamma$} is a sequence $i_1, \dots, i_t \in \arcs$.  Its \emph{reverse}, denoted $\gamma^*$, is the color sequence $i_t^*, \dots, i_1^*$.  We call $\gamma$ a \emph{nonbacktracking (n.b.) color sequence} if $i_{s+1} \neq i_s^*$ for all $s \in [t-1]$.
\end{definition}
\begin{definition} \label{def:reversi}
    Given a color sequence $\gamma = (i_1, \dots, i_t)$ and matrix-weights $a_1, \dots, a_{d+2\q}$, we introduce the notation $a(\gamma) = a_{i_t} a_{i_{t-1}} \cdots a_{i_1}$.  Note that $a(\gamma^*) = a(\gamma)^*$ and hence $\norms*{a(\gamma^*)}_\frob = \norms*{a(\gamma)}_\frob$.
\end{definition}
\begin{theorem}                                     \label{thm:main}
    Let $d$, $\q$, $r$, $R$, $\ol{\eps}$, $p > 0$ be constants, and let $\Xi_1 \subset \mathfrak{K}_{2R}$ and~$\kappa$ be as in \Cref{prop:eps-net}, with $\lambda = 1$ and $\eps =\ol{\eps}/3$.   Then there are large enough constants $C_1, C_2$ such that the following holds:

    Let $\lift_n$ be any $n$-lift ($n \geq C_1$) that is $\lambda$-bicycle-free, where  $\lambda = C_2 (\log \log n)^2$. Then for a uniformly random edge-signing $\bchi$ of $G_{\lift_n}$, except with probability at most $n^{-p}$ the following holds simultaneously for all matrix bouquets $\calK \in \Xi_1$:
    \begin{equation} \label[ineq]{eqn:mainbd}
        \norm{B_n(\bchi \lift_n, \calK)}_\opnorm \leq \rho(B_\infty(\calK)) + \ol{\eps}/3.
    \end{equation}
    When this happens we may easily deduce from \Cref{prop:eps-net} that
    \begin{equation} \label[ineq]{eqn:mainbd-conclude}
        \rho(B_n(\bchi \lift_n, \calK)) \leq \norm{B_n(\bchi \lift_n, \calK)}_\opnorm \leq \rho(B_\infty(\calK)) + \ol{\eps}
    \end{equation}
    holds simultaneously for all $R$-bounded bouquets $\calK \in \mathfrak{K}_{R}$.
\end{theorem}

Fix any $\lift_n$ and any $\calK = (a_1, \dots, a_{d+2\q}) \in \calK_{2R}$.  It suffices to show that \Cref{eqn:mainbd} holds for this fixed~$\calK$ except with probability at most $n^{-2p}$ over the choice of~$\bchi$; we may then take a union bound over the at most $\kappa$ bouquets $\calK \in \Xi_1$, and enlarge~$C_1$ if necessary so that $\kappa n^{-2p} \leq n^{-p}$.

Let us write $\bB_n = B_n(\bchi \lift_n, \calK)$ for brevity and write $\ol{\rho} = \rho(B_\infty(\calK)) + \ol{\eps}/6$.  The main characterization of $\ol{\rho}$ that we will need is the  below proposition of Bordenave--Collins:\footnote{We remark that this is the one place in this section where the inverse norm-bound $\|a_i^{-1}\|_\opnorm \leq 2R$ is used; it takes care of a ``fencepost error'' in the natural proof of this proposition. In fact, a personal communication from the authors of~\cite{BC19} suggests that a variation of their argument can eliminate the need to bound the values $\|a_i^{-1}\|_\opnorm$.}
\begin{proposition}                                     \label{prop:50}
    (\cite[eqn.~(50)]{BC19}.)  There is a constant $c \geq 1$ depending only on $d$, $\q$, $r$, $R$, $\ol{\eps}$ (and not on the specific~$\a_i$'s) such that for all $t \in \N$,
    \[
        \sum_{\substack{\textnormal{n.b.\ color sequences}\\ \gamma = (i_1, \dots, i_t)}} \norms*{a(\gamma)}_\frob^2
            \leq c \ol{\rho}^{2t}.
    \]
    In particular, $\norms*{a(\gamma)}_\frob \leq c \ol{\rho}^{t}$ for all n.b.\ color sequences of length~$t$.
\end{proposition}

It now remains to prove that except with probability at most~$n^{-2p}$:
\begin{equation}\label[ineq]{eqn:fixed-K-bound}
    \norm{\bB_n}_\opnorm \leq \ov\rho + \ov{\eps}/6.
\end{equation}

The key to proving this is to use the Trace Method in expectation.  In preparation for this, we make some definitions (borrowing some terminology from~\cite{MOP20b}):
\begin{definition}
    In the color-regular graph $\calG_n = \G_n(\lift_n, \calK)$, a length-$t$ walk may be naturally specified by giving a starting vertex $u \in V_n$ and a sequence $i_1, \dots, i_t \in \arcs$ of colors.  We call the walk \emph{closed} if the final vertex is equal to~$u$.  We define an \emph{$\ell$-hike} to be a closed, length-$2\ell$ walk composed of two consecutive n.b.~color sequences $(i_1, \dots, i_\ell)$ and $(j_{\ell+1}, \dots, j_{2\ell})$; equivalently, it is defined by a length-$2\ell$ color sequence that is nonbacktracking except possibly at the midpoint of the walk.  Finally, we call an $\ell$-hike \emph{even} if each (undirected) edge is traversed an even number of times (in either direction); more generally, we call it \emph{singleton-free} if no edge is traversed exactly once.
\end{definition}

Fix
\begin{equation}    \label{eqn:def-ell}
    \ell = \lceil C_3 \log n \rceil,
\end{equation}
where $C_3$ is a large constant depending only on $d$, $\q$, $r$, $R$, $\ol{\eps}$,~$p$.  (The constant $C_2$ will also depend on $C_3$, and the constant $C_1$ will depend on $C_2$ and $C_3$.) Write $\bB_n = B_n(\bchi \lift_n, \calK)$, and define the random variable
\[
    \OurTrace = \tr(\bB_n^\ell(\bB_n^\ell)^\conj),
\]
which is nonnegative since $\bB_n^\ell(\bB_n^\ell)^\conj$ is positive semidefinite.  Temporarily abbreviating $\wt{\sigma}_i = P_{\bchi_i \sigma_i}$, one can check that
\[
    \bB_n^\ell(\bB_n^\ell)^\conj = \sum \wt{\sigma}_{j_{2\ell -1}} \wt{\sigma}_{j_{2\ell-2}} \dots \wt{\sigma}_{j_\ell} \wt{\sigma}_{i_\ell} \wt{\sigma}_{i_{\ell-1}} \dots \wt{\sigma}_{i_1}\otimes \ketbra{j_{2\ell}}{i_0^*} \otimes a_{j_{2\ell}}a_{j_{2\ell -1}} \dots a_{j_{\ell + 1}}a_{i_{\ell-1}}a_{i_{\ell-2}}\dots a_{i_0},
\]
where the sum is over n.b.~color sequences $(i_0,i_1,\dots, i_\ell)$ and $(j_\ell,\dots,j_{2\ell})$ such that $i_\ell^* = j_\ell$. Taking the trace, we obtain
\[
    \OurTrace = \tr\parens*{\bB_n^\ell(\bB_n^\ell)^\conj} = \sum_{\substack{\text{special } (\ell+1)\text{-hikes } (u,\gamma) \\ \gamma = (i_0, \dots, i_\ell, j_\ell, \dots, j_{2\ell})
    }} \bchi[(u,\gamma)] \tr\parens*{a(\wh{\gamma})},
\]
where:
\begin{itemize}
    \item the adjective ``special'' denotes that $i_\ell^* = j_\ell$ and $i_0^* = j_{2\ell}$;
    \item the notation $\bchi[(u,\gamma)]$ denotes the product of the edge-signs from~$\bchi$ for each edge traversed by the hike~$(u,\gamma)$;\footnote{Note that the first and last edge-sign do not actually arise in $\OurTrace$, but it is okay that we have artificially inserted them: in any special hike, the first and last step are always along the same edge and hence contribute equal signs; thus their product is always~$1$.}
    \item $\wh{\gamma}$ denotes that the middle two elements $i_{\ell}$ and $j_{\ell}$ are omitted from~$\gamma$.
\end{itemize}
Taking the expectation over~$\bchi$ kills all non-even hikes, and we conclude
\begin{equation} \label{eqn:expected-trace}
    \E[\OurTrace] = \sum_{\substack{\text{even, special} \\ (\ell+1)\text{-hikes } (u,\gamma)}} \tr\parens*{a(\wh{\gamma})}
    \leq \sqrt{r} \sum_{\substack{\text{even, special} \\ (\ell+1)\text{-hikes } (u,\gamma)}} \norms*{a(\wh{\gamma})}_\frob.
\end{equation}
Note that if we delete the first/last steps, and also the middle two steps, of a special $(\ell+1)$-hike $(u,\gamma)$ (these two step-pairs both being a step and its reverse), we obtain an $(\ell-1)$-hike $(v,\gamma')$, where $v = \sigma_{i_0}(u)$.  We have $a(\wh{\gamma}) = a_{j_{2\ell}} a(\gamma') a_{i_0}$, and hence $\norms{a(\wh{\gamma})}_\frob \leq (2R)^{2} \norms{a(\gamma')}_\frob$.  For each $(\ell-1)$-hike $(v, \gamma')$, there are at most $(d+2\q)^2$ ways to add a step and its reverse to the beginning and the middle of the hike to get a special $(\ell+1)$-hike. Replacing also the condition ``even'' by the broader condition ``singleton-free'', we finally conclude:
\begin{equation}     \label[ineq]{eqn:main-bound}
    \E[\OurTrace] \leq 4\sqrt{r} R^{2} (d+2\q)^2 \sum_{\substack{\text{singleton-free} \\ (\ell-1)\text{-hikes } (v,\gamma)}} \norms*{a(\gamma)}_\frob.
\end{equation}
Our goal is now to bound the right-hand side of \Cref{eqn:main-bound}
using only the $\lambda$-bicycle-free property of~$\calG_n$.

\subsection{Elementary graph theory}
We begin with some elementary graph theory.  In the below definitions, by ``graph'' we mean an undirected multigraph, possibly with self-loops.
\begin{definition}
    The \emph{excess} of a graph $H$ is $\exc(H) = |E(H)| - |V(H)|$.
\end{definition}
\begin{definition}
    Given a degree-$2$ vertex $v$ in a graph~$H$, \emph{smoothing}~$v$ refers to the operation of deleting~$v$ from~$H$ and adding an edge between its two (former) neighbors.  The \emph{smoothing} of a graph~$H$, which we will denote $\ul{H}$, is the graph obtained by iteratively smoothing degree-$2$ vertices (it is easy to see that the order does not matter).  Each vertex in $\ul{H}$ has the same degree in~$\ul{H}$ as it had in~$H$; hence the vertex set of~$\ul{H}$ consists of all vertices that originally had degree other than~$2$ in~$H$.  The edges in $\ul{H}$ correspond to paths in~$H$ whose internal nodes all had degree~$2$.  We call these paths \emph{stretches}.  Finally, it is easy to see that $\exc(\ul{H}) = \exc(H)$.
\end{definition}
We will also need the following lemma of elementary graph theorem from~\cite{MOP20b}:
\begin{lemma}                                       \label{lem:exc}
    (\cite[Thm.~2.13]{MOP20b}.) Let $H$ be a $k$-vertex graph that is $\lambda$-bicycle free, where $\lambda \geq 10 \ln k$.  Then $\exc(H) \leq \frac{\ln(ek)}{\lambda} \cdot k$.
\end{lemma}

Given an $(\ell-1)$-hike $(v,\gamma)$, let $\calH = \calH_{(v,\gamma)}$ be the edge-colored subgraph of~$\calG_n$ induced by the edges on which the hike walks.  This $\calH$ is $\lambda$-bicycle free (since~$\calG_n$ is), and it has $k \leq 2\ell-2$ vertices.  Provided $C_2$ and then $C_1$ are taken large enough given~$C_3$, we will have
\[
    10 \ln (2\ell-2)  \leq \lambda.
\]
(The left-hand side is $\Theta(\log \log n)$, the right-hand side is $\Theta((\log \log n)^2)$.) Thus \Cref{lem:exc} implies $\exc(\calH) \leq \frac{\ln(ek)}{\lambda} \cdot k$.  Due to the nonbacktracking nature of the hike, every vertex in~$\calH$ has degree at least~$2$, except possibly its initial vertex~$v$ and its midpoint vertex, which we call~$w$.  Abusing notation slightly, let us write $\ul{\calH} = (\ul{\calV}, \ul{\calE})$ for the smoothing of~$\calH$ where we do \emph{not} smooth~$v$ or~$w$ if they have degree~$2$.  We still have $\exc(\ul{\calH}) = \exc(\calH)$.  Writing $c_i$ for the number of vertices of degree~$i$ in $\ul{\calH}$, we have
\begin{align*}
    \frac{\ln(ek)}{\lambda} \cdot k \geq \exc(\calH) = \exc(\ul{\calH}) = |\ul{\calE}| - |\ul{\calV}| &= \tfrac12(c_1 + 2c_2 + 3c_3 + \cdots) - (c_1 + c_2 + c_3 + \cdots) \\
    &\geq -\tfrac23 c_1 -\tfrac13 c_2 + \tfrac13 \cdot \tfrac12 (c_1 + 2c_2 + 3c_3 + 4c_4 + \cdots) \\
    &= -\tfrac23 c_1 -\tfrac13 c_2 + \tfrac13 |\ul{\calE}|.
\end{align*}
We have  $c_1 + c_2 \leq 2$ (since $v,w$ are the only vertices in $\ul{\calH}$ that may have degree smaller than~$3$).
Hence:
\begin{corollary}                                       \label{cor:exc}
    For any $(\ell-1)$-hike $(v,\gamma)$, the smoothing $\ul{\calH}$ of $\calH_{(v,\gamma)}$ has
    \[
         |\ul{\calE}| \leq \tfrac{3\ln(ek)}{\lambda}\cdot k + 4 \leq 10 \tfrac{\ln \ell}{\lambda} \cdot \ell.
    \]
    (The latter inequality uses $\lambda \leq \frac14 \ell$, which holds provided $C_2$ and then~$C_1$ are taken large enough given $C_3$; recall $\lambda = \Theta((\log \log n)^2)$ and $\ell = \Theta(\log n)$.) That is, $\calH_{(v,\gamma)}$ consists of the vertices $\ul{\calV}$, together with at most $10\tfrac{\ln \ell}{\lambda} \cdot \ell$ stretches (disjoint paths) linking these vertices.
\end{corollary}

\subsection{Encoding and decoding}
\newcommand{\encode}{\text{{\scshape Encode}}}
\newcommand{\decode}{\text{{\scshape Decode}}}
\newcommand{\outline}{\text{{\scshape Outline}}}
\newcommand{\colordata}{\text{{\scshape Fresh}}}
\newcommand{\staledata}{\text{{\scshape Stale}}}
For the purpose of analyzing the right-hand side of \Cref{eqn:main-bound}, we will now introduce two algorithms called $\encode$ and $\decode$.  The $\encode$ algorithm is inspired by the analysis in~\cite[Lem.~26]{BC19}, and is a significant elaboration of a similar kind of algorithm in~\cite{MOP20b}.  It will take as input a singleton-free $(\ell-1)$-hike $(v,\gamma)$, and will output~$v$ together with a few ``data structures'', which we now define.
\begin{definition}
    An \emph{outline data structure} $\outline$ consists of:
    \begin{itemize}
        \item The \emph{trek count} $T$: an integer between $1$ and $50\tfrac{\ln \ell}{\lambda} \cdot \ell$.
        \item The \emph{length data} $\ell_1, \dots, \ell_T$: a sequence of positive numbers adding to $2(\ell-1)$.
        \item The \emph{type data}: A subset of $[T]$ called $\calF_1$ (the letter $\calF$ stands for ``fresh''), a disjoint subset of~$[T]$ called $\calF_2$ with $|\calF_2| = |\calF_1|$, and the remaining subset of~$[T]$ called~$\calS$ (``stale''); thus $[T] = \calF_1 \sqcup \calF_2 \sqcup \calS$.
        \item The \emph{matching data}: this consists of a bijection $\tau : \calF_2 \to \calF_1$ such that $\tau(t) < t$ for each $t \in \calF_2$ (that is, $\tau$~matches each $t \in \calF_2$ with an ``earlier'' $t' \in \calF_1$); and also, a mapping $\rho : \calF_2 \to \{\text{forward}, \text{backward}\}$.
    \end{itemize}
\end{definition}
\begin{definition}
    A \emph{fresh data structure} $\colordata$ consists of a list $C_1, \dots, C_f$, where each $C_i$ is an n.b.~color sequence.  We say that $\colordata$ is \emph{compatible} with an outline data structure $\outline$ (as above) if $f = |\calF_1|$ and $|C_j| = \ell_{i_j}$ when  the elements of $\calF_1$ in increasing order are denoted $i_1, \dots, i_f$.
\end{definition}
\begin{definition}
    A \emph{stale data structure} $\staledata$ consists of a list $(r_1,h_1), \dots, (r_s,h_s)$, where each $0 \leq r_i < 2(\ell-1)$ is an integer and each $h_i$ is in~$\arcs \cup \{0\}$.  We say that $\staledata$ is \emph{compatible} with an outline data structure $\outline$ (as above) if $s = |\calS|$.
\end{definition}

\subsubsection{Encoding}
As mentioned, the $\encode$ algorithm takes as input a singleton-free $(\ell-1)$-hike $(v,\gamma)$, and it outputs $(v, \outline, \colordata, \staledata)$, where $\outline$ is an outline data structure and $\colordata$/$\staledata$ are compatible fresh/stale data structures (respectively).  Recall the notation $\ul{\calH} = (\ul{\calV}, \ul{\calE})$ for the smoothing of $\calH = \calH_{(v,\gamma)}$.  The algorithm $\encode$ analyzes how the hike traverses stretches from~$\calH$.  Crucially, by the singleton-free property, every stretch is traversed (in either direction) \emph{at least twice}.

To explain in detail how the $\encode$ algorithm works, imagine a hiker walking along the edges of the hike $(v,\gamma)$ from beginning to end.  At various points the hiker will pause at some vertex in~$\ul{\calV}$, look at the upcoming edges of the hike, designate some initial portion of them one or more \emph{treks}, record information about the trek(s) into the data structures (in particular, increasing the trek count~$T$), then walk the trek(s) and pause again.  The first pause occurs at the beginning of the hike before any edges have been traversed.  In general, when the hiker has paused after completing $t-1$ treks, the hiker looks ahead to the next stretch $\ul{e} \in \ul{\calE}$ that will be traversed and performs the following three rules:

\paragraph{1. First traversals.} If this will be the \emph{first} time the hiker traverses~$\ul{e}$ (in either direction), then it will be designated as the single next trek.  Its length (number of edges)~$\ell_t$ is appended to the length data of~$\outline$, and the trek index~$t$ is included into $\calF_1$ in the type data.  Finally, the n.b.~color sequence~$C$ is appended to the fresh data structure~$\colordata$.

\paragraph{2. Second traversals.} If this will be the \emph{second} time the hiker traverses~$\ul{e}$, then again it will be designated the single next trek.  Its length~$\ell_t$ is appended to the length data of~$\outline$, and the trek index~$t$ is included into $\calF_2$ in the type data.  Furthermore, matching data is recorded: $\tau(t)$ is set to the earlier trek index on which~$\ul{e}$ was traversed; and, $\rho(t)$ is set to ``forward'' if $\ul{e}$ will be traversed in the same direction as last time, and ``backward'' otherwise.

\paragraph{3. Third or higher traversals.} If this will be the \emph{third or higher} time the hiker traverses~$\ul{e}$, the hiker begins preparing a ``stale supertrek''.   The hiker looks further ahead at upcoming stretches to be traversed.  All upcoming ``third-time-or-higher'' stretches $\ul{e}', \ul{e}'', \dots$ are added to the stale supertrek, up to (but not including) the next stretch that will be a first-time or second-time traversal.

Next, the stale supertrek is broken down into multiple ``stale subtreks'' of length equal to the bicycle-free radius~$\lambda$, with the necessary exception that the last of these subtreks allowed to have length less than~$\lambda$.  Note that while the first vertex of the first subtrek and the last vertex of the last subtrek will be in~$\ul{\calV}$, the intermediate starting/ending vertices may be any vertices in~$\calH$ (i.e., they may be ``inside'' a stretch from~$\ul{\calE}$).  There is one additional proviso: if some stale subtrek encompasses the exact midpoint of the hike (i.e., both the $(\ell-1)$th and the $\ell$th steps) then this stale subtrek is further broken into two subtreks at this the midpoint.  This proviso is included to ensure that the hiker traverses each stale subtrek in a nonbacktracking fashion.

Now the hiker designates these stale subtreks as the next few treks, and data about them is encoded into the data structures.  First, their lengths $\ell_t, \ell_{t+1}, \dots$ are added to $\outline$'s length data.  Next, their trek indices are included into the type data's~$\calS$.  Finally, for each stale subtrek~$Q$, an $(r,h)$ pair is appended into the stale data structure~$\staledata$.  The role of the $(r,h)$ pair is to allow the $\decode$ algorithm to reconstruct the stale subtrek~$Q$.  Let $y$ be the vertex of~$\calH$ at which~$Q$ begins and let $z$ be the vertex of~$\calH$ at which~$Q$ ends.  We know the hike must have reached~$z$ at an earlier time (in fact, at least two earlier times).  The parameter~$r$ is set to the first ``time'' (between $0$ and $2(\ell-1)$) at which the hike reached~$z$.  As for defining~$h$, let~$B$ be the subgraph of~$\calH$ induced by the radius-$\lambda$ neighborhood of~$x$. By the bicycle-free property, $B$~contains at most one cycle.  If there is such a cycle, and the trek walked along it for at least one step, then~$h$ is set to the color from~$\arcs$ of the first such cycle-step.  Otherwise, if the trek does not walk on~$B$'s cycle, or if $B$~does not even have a cycle, then $h$~is conveniently set to~$0$.\\

This completes the description of the $\encode$ algorithm. Finally, we need to briefly verify two things about the $\outline$ data structure output by~$\encode$.  First, as we noted, the singleton-free property implies that each stretch that is traversed is traversed at least twice; this means that the type data will indeed have $|\calF_2| = |\calF_1|$ and a well-defined bijection~$\tau$.  Second, and more importantly, we need to verify that the trek count~$T$ will indeed be at most $50\tfrac{\ln \ell}{\lambda} \cdot \ell$ as promised.  From \Cref{cor:exc} we know that there are at most $10\tfrac{\ln \ell}{\lambda} \cdot \ell$ stretches.  These can contribute at most $20\tfrac{\ln \ell}{\lambda} \cdot \ell$ ``fresh treks'' in total.  Let us now upper-bound the possible number of ``stale treks'' (ignoring the $+1$ we may get due to potentially  traversing the hike's midpoint).  Note that when the hiker produces a stale supertrek of length~$k$, it gets decomposed into $\lceil k/\lambda \rceil \leq  k/\lambda + 1$ stale subtreks.  If we sum this quantity over all stale supertreks, we get
\[
    (\text{\# of steps involved in stale supertreks})/\lambda + (\text{\# of stale supertreks}) \leq 2(\ell-1)/\lambda + 20 \tfrac{\ln \ell}{\lambda} \cdot \ell,
\]
the last figure because each stale supertrek is preceded by a fresh trek.  Now including the at most $20 \tfrac{\ln \ell}{\lambda} \cdot \ell$ fresh treks, as well as the potential $+1$ for the midpoint, we get a total trek count of at most
\[
    20 \tfrac{\ln \ell}{\lambda} \cdot \ell + 2(\ell-1)/\lambda + 20 \tfrac{\ln \ell}{\lambda} \cdot \ell + 1 \leq 50\tfrac{\ln \ell}{\lambda} \cdot \ell,
\]
as claimed. (Here we again used $\lambda \leq \frac14 \ell$.)

\subsubsection{Decoding}
We now describe the $\decode$ algorithm, which has the following properties:
\begin{enumerate}
    \item It takes as input a vertex~$v$, an outline data structure $\outline$, and compatible fresh and stale data structures $\colordata$ and $\staledata$.
    \item It either outputs a singleton set $\{(v,\gamma)\}$, where $(v,\gamma)$ is an $(\ell-1)$-hike in $\calG_n$; or, it outputs the empty set~$\emptyset$ (``invalid'').
    \item \label{item:dec-ppty} If $(v,\outline,\colordata,\staledata)$ is the output of the $\encode$ algorithm applied to some singleton-free $(\ell-1)$-hike $(v,\gamma)$ in~$\calG_n$, then $\decode$ will output $\{(v,\gamma)\}$.
\end{enumerate}
In other words, a singleton-free hike may be recovered from the data output by~$\encode$.   We will not explicitly prove \Cref{item:dec-ppty} above; rather, it will become clear as we describe the $\decode$ algorithm.  Also, implicit in our description is that whenever the $\decode$ algorithm cannot continue, or is required to do something nonsensical, or produces an invalid singleton-free $(\ell-1)$-hike, it simply outputs~$\emptyset$.\\

The $\decode$ algorithm will (attempt to) recover a singleton-free hike $\{(v,\gamma)\}$ in a trek-by-trek fashion from its input data.  The algorithm will always maintain a ``current vertex'' from~$\ul{\calV}$.  The initial value of the ``current vertex'' is of course the vertex~$v$ which is part of $\decode$'s input.  For $t = 1 \dots T$, the $\decode$ algorithm recovers the $t$th trek by first considering whether $t$ is in $\calF_1$, $\calF_2$, or $\calS$.
If $t \in \calF_1$, then $\decode$ will take the first unused n.b.~color sequence~$C$ from $\colordata$ (which by compatibility will have length~$\ell_t$) and will follow~$C$ from the current vertex. The final vertex reached becomes the new ``current vertex''.  At this point, the $\decode$ algorithm will have ``learned'' a new stretch~$\ul{e} \in \ul{\calE}$, which it may associate to this~$t$th trek.

If $t \in \calF_2$, then $\decode$ will first let $t' = \tau(t) \in \calF_1$.  Since $t' < t$, the $\decode$ algorithm will have already learned the stretch~$\ul{e} \in \ul{\calE}$ associated to~$t'$.  The $\decode$ algorithm can therefore typically recover the~$t$th trek: it just follows~$\ul{e}$ again.
If $\ul{e}$ is a self-loop in~$\ul{\calH}$, there is an ambiguity about the direction in which the trek traverses~$\ul{e}$, but this is resolved using~$\rho(t)$ from the matching data.

Finally, if $t \in \calS$, then $\decode$ will recover the stale trek with the aid of the associated stale data~$(r,h)$.  Let $y$ denote the current vertex (i.e., the initial vertex for the stale trek) and~$\ell_t$ the trek length. As this trek is stale, it will take place entirely within a subgraph of~$\calH$ that the $\decode$ algorithm has already ``learned''.  In fact, it will take place within a further subgraph~$B$, the one of radius~$\ell_t$ centered at~$y$.  The $\decode$ algorithm can infer the ending vertex~$z \in B$ of the trek using the stale datum~$r$; the $r$th vertex visited so far on the hike is~$z$.  Next, since $\ell_t \leq \lambda$ (in the properly encoded case), $B$ will have at most one cycle.  It is not hard to see that in this unicyclic~$B$, a nonbacktracking walk from~$y$ to~$z$ of length~$\ell_t$ is \emph{almost} uniquely defined.  The only possible ambiguity that may arise comes from the direction in which $B$'s cycle (should it exist) is traversed.  This ambiguity is resolved using the stale datum~$h$; it supplies the first color used by the trek when the cycle is entered.

This completes the description of the $\decode$ algorithm.

\subsection{Counting}
We now bound the right-hand side of \Cref{eqn:main-bound}.  By virtue of the \Cref{item:dec-ppty} property of the $\decode$ algorithm, we have
\begin{equation} \label[ineq]{eqn:mybound}
    \sum_{\substack{\text{singleton-free} \\ (\ell-1)\text{-hikes } (v,\gamma)}} \norms*{a(\gamma)}_\frob  \leq \sum_{v,\ \outline} \quad \sum_{\text{compatible } \colordata,\ \staledata} \quad \sum_{(v,\gamma) \in \decode(v, \outline, \colordata, \staledata)} \norms*{a(\gamma)}_\frob
\end{equation}
where on the right-hand side we sum over all possible $v, \outline$, then all possible $\colordata, \staledata$ compatible with $\outline$, then all hikes in the set $\decode(v, \outline, \colordata, \staledata)$ (which set, recall, either has cardinality~$0$ or~$1$).  A key property of this new summation is that there is a natural way to decompose and bound $\norms*{a(\gamma)}_\frob$ as a function of~$\outline$.  Given the outline data structure~$\outline$ with trek count $T$ and length data $\ell_1, \dots, \ell_T$, we can naturally break up any length-$2(\ell-2)$ color sequence~$\gamma$ into parts $\gamma_1, \dots, \gamma_T$ of length $\ell_1, \dots, \ell_T$.  Then by submultiplicativity of norms we have
\[
    \norms*{a(\gamma)}_\frob  \leq \norms*{a(\gamma_1)}_\frob \cdots \norms*{a(\gamma_T)}_\frob.
\]
We then bound each of these factors depending on the type data in~$\outline$.  For a factor $\norms*{a(\gamma_i)}_\frob$ with $i$~in the stale subset $\calS$, we use \Cref{prop:50} to bound it by $c \ol{\rho}^{\ell_i}$.

The remaining factors may be naturally grouped into pairs according to $\outline$'s matching data.  Suppose that $t \in \calF_1$, $t' \in \calF_2$, and $\tau(t') = t$.  It is a property of the decoding algorithm that whenever $(v,\gamma) \in \decode(v, \outline, \colordata, \staledata)$, we have either $\gamma_t = \gamma_{t'}$ or $\gamma_t = \gamma_{t'}^*$ (depending on $\rho(t')$).  Recalling \Cref{def:reversi} we see that either way, $\norms*{a(\gamma_t)}_\frob = \norms*{a(\gamma_{t'})}_\frob$.  Furthermore, this norm \emph{only} depends on the fresh data structure~$\colordata$; specifically, if $t$ is the $j$th element of $\calF_1$ (in increasing order), then $\norms*{a(\gamma_t)}_\frob = \norms*{a(\gamma_{t'})}_\frob = \norms*{a(C_j)}_\frob$.  As a consequence, for a given $(v, \outline, \colordata, \staledata)$ the following bound holds:
\[
    \sum_{(v,\gamma) \in \decode(v, \outline, \colordata, \staledata)} \norms*{a(\gamma)}_\frob
    \leq \parens*{\prod_{i \in \calS} c \ol{\rho}^{\ell_i}} \parens*{\prod_{j = 1}^f \norms*{a(C_j)}_\frob^2}
    = c^{|\calS|} \ol{\rho}^{\ell_{\calS}} \cdot
                                              \prod_{j = 1}^f \norms*{a(C_j)}_\frob^2,
\]
where we have written $\ell_{\calS} = \sum_{i \in \calS} \ell_i$ for the number of stale steps.  Notice that the above bound does not depend on $v$, $\decode$ or, on $\staledata$, just on $\outline$ and $\colordata$.  Substituting it into \Cref{eqn:mybound} yields
\begin{equation} \label[ineq]{eqn:mybound2}
    \sum_{\substack{\text{singleton-free} \\ (\ell-1)\text{-hikes } (v,\gamma)}} \norms*{a(\gamma)}_\frob
    \leq n \cdot \sum_{\outline}  c^{|\calS|} \ol{\rho}^{\ell_{\calS}} \cdot \#\{\text{compatible } \staledata\} \cdot \sum_{\text{compatible } \colordata} \quad  \prod_{j = 1}^f \norms*{a(C_j)}_\frob^2.
\end{equation}
Now for a fixed $\outline$ we can interchange the final sum and product in the above:
\begin{align*}
    \sum_{\text{compatible } \colordata} \quad \prod_{j = 1}^f \norms*{a(C_j)}_\frob^2
    &= \prod_{t \in \calF_1}   \quad  \sum_{\text{n.b.\ color sequences $C$ of length $\ell_t$}} \norms*{a(C)}_\frob^2 \\
    &\leq \prod_{t \in \calF_1} (c \ol{\rho}^{2\ell_t}) = c^{|\calF_1|} \ol{\rho}^{\ell_{\calF}},
\end{align*}
where the inequality is \Cref{prop:50} and we have written $\ell_{\calF} = 2\sum_{i \in \calF_1} \ell_i = \sum_{i \in \calF_1 \sqcup \calF_2} \ell_i$.  Putting this into \Cref{eqn:mybound2} yields
\begin{align*}
     \sum_{\substack{\text{singleton-free} \\ (\ell-1)\text{-hikes } (v,\gamma)}} \norms*{a(\gamma)}_\frob
     &\leq n \cdot \sum_{\outline}  c^{|\calF_1| + |\calS|} \ol{\rho}^{\ell_{\calF} + \ell_{\calS}} \cdot \#\{\text{compatible } \staledata\}  \\
     & \leq n \cdot c^{T} \cdot \ol{\rho}^{2(\ell-1)} \cdot \#\{\text{compatible } (\outline, \staledata)\}.
\end{align*}

It remains to bound the number of compatible $(\outline, \staledata)$ pairs, and it suffices to do this very crudely.  For a given choice of~$T$, the number of possibilities for the length data is at most $(2\ell)^T$.  The number of possibilities for the type data is at most~$3^T$.  The number of choices for the matching data is at most $(2T)^T$.  And the number of choices for compatible stale data is at most $(2D\ell)^T$ for $D = d+2\q$.  Thus we have at most~$(24 D \ell T)^T$ possibilities, given~$T$.  Recalling $T \leq 50\frac{\ln \ell}{\lambda} \ell$ and also \Cref{eqn:main-bound}, we finally conclude
\begin{equation} \label[ineq]{eqn:final-Etrace}
    \E[\OurTrace] \leq n \cdot \exp\parens*{O\parens*{\tfrac{\log^2 \ell}{\lambda} \ell}} \cdot \ol{\rho}^{2\ell},
\end{equation}
where we have now begun using $O(\cdot)$ notation to hide constants depending on $d$, $\q$, $r$, $R$, $\ol{\eps}$,~$p$.  Markov's inequality now implies that except with probability at most $n^{-2p}$,
\begin{equation} \label[ineq]{eqn:Markov1}
    \norm{\bB_n}_\opnorm^{2\ell} \leq \tr(\bB_n^\ell(\bB_n^\ell)^\conj) = \OurTrace \leq n^{2p+1} \cdot \exp\parens*{O\parens*{\tfrac{\log^2 \ell}{\lambda} \ell}} \cdot \ol{\rho}^{2\ell},
\end{equation}
and in this case
\begin{equation} \label[ineq]{eqn:final-error}
    \norm{\bB_n}_\opnorm \leq \exp\parens*{O\parens*{\tfrac{\log n}{\ell} + \tfrac{\log^2 \ell}{\lambda}}} \cdot \ol{\rho} \leq \ol{\rho} + O\parens*{\tfrac{\log n}{\ell}} + O\parens*{\tfrac{\log^2 \ell}{\lambda}}
\end{equation}
(where the second inequality used that $\ol{\rho} \leq O(1)$ thanks to \Cref{cor:rhoB-bound}).  By taking $C_3$, then $C_2$, then $C_1$ large enough, this shows \Cref{eqn:fixed-K-bound} holds for our fixed choice of~$\calK$, except with probability at most $n^{-2p}$.   This completes the proof of \Cref{thm:main}.

\begin{corollary}   \label{cor:main}
    In the setting of \Cref{thm:main}, suppose the edge signs $\bchi$ are not independent and uniformly random, but are merely drawn from a $(\delta, 2\ell)$-wise uniform distribution, where $\delta = 1/n^{C_0}$ for a sufficiently large constant $C_0 = C_0(d, \q, r, R, \ov\eps, p)$, and where $\ell = \lceil C_3 \log n \rceil$ as in \Cref{eqn:def-ell}.  Then the conclusion of the theorem continues to hold.
\end{corollary}
\begin{proof}
    The only place where the independence of the edge signs was used was in \Cref{eqn:expected-trace}, to say that non-even hikes dropped out of the expectation. When $\bchi$ is merely $(\delta, 2\ell)$-wise uniform, the inequality therein is not exact; however, since $\OurTrace$ is a degree-$2\ell$ polynomial in~$\bchi$, the error is (\Cref{fact:t-wise-bits}) at most an additive
    \[
         \delta \cdot \sum_{\text{special } (\ell+1)\text{-hikes } (u,\gamma)}  \abs*{\tr\parens*{a(\wh{\gamma})}} \leq \delta \cdot n (d+2\q)^{2\ell} \cdot r (2R)^{2\ell},
    \]
    where we used naive upper bounds on the number of special $(\ell+1)$-hikes and on the absolute trace of the hike weight (the latter using $\calK \in \calA_{R}$).  In turn, by choosing the constant~$C_0$ large enough (in particular, large enough vis-a-vis~$C_3$), we can ensure this error is much smaller than the right-hand side of \Cref{eqn:final-Etrace} --- for concreteness, say at most half of it. (We use here the naive lower bound~$\ov\rho \geq \ov\eps/6$.)  Thus the subsequent bound \Cref{eqn:Markov1} is affected only up to a constant factor, and in the final bound \Cref{eqn:final-error} only the hidden constant in the $O(\frac{\log n}{\ell})$ is affected.
\end{proof}

\section{Weakly derandomizing \emph{n}-lifts}  \label{sec:weak-derand}
In this section we fix the same notation and terminology as in the beginning of  \Cref{sec:random-2lift}.
We remark that except for the analysis of the net, this section is not substantially different from the analogous~\cite[Sec.~4]{MOP20b}.

\subsection{Derandomizing simple arguments about cycles}
In the course of proving their main theorem on random $n$-lifts, Bordenave and Collins need to show that the graph underlying a random $n$-lift is bicycle-free for radius $\lambda = \Omega(\log n)$ and also that it has at least one vertex with an acyclic $\lambda$-neighborhood.  These are straightforward proofs, and it is also straightforward to see that they are weakly derandomizable using $\Theta(\log n)$-wise uniform permutations; indeed, this was essentially already done in \cite[Prop.~4.3]{MOP20b}.   As the exact statement needed is not literally in~\cite[Prop.~4.3]{MOP20b}, we provide a statement and briefly sketch a proof here.
\begin{proposition}                                     \label{prop:easy-graph-derand}
     Given constants $d, \q$, there are constants $C_0, C_1$ such that the following holds:  Assume $n, \lambda \in \N^+$ satisfy $n \geq C_0$ and $\lambda \leq (\log n)/C_1$.  Then for a $4\lambda$-wise uniform $n$-lift $\blift_n$, except with probability at most $n^{-.999}$ we have both of the following:
     \begin{enumerate}
        \item \label{enum:basic-G1} $G_{\blift_n}$ has at least one vertex whose radius-$\lambda$ neighborhood is acyclic;
        \item \label{enum:basic-G2} $G_{\blift_n}$ is $\lambda$-bicycle free.
     \end{enumerate}
\end{proposition}
\begin{proof}
    We will mainly rely on~\cite[Lem.~23]{BC19}.  For a fixed graph~$H$ on~$m$ edges, let $\bX_H$ denote the number of copies of~$H$ in $G_{\blift_n}$.  In~\cite[Lem.~23]{BC19}, the quantity $\E[\bX_H]$ is computed (and then bounded) assuming the lift $\blift_n$ is truly uniformly random.  Now it is easy to see that $\bX_H$ is a polynomial of degree~$m$ in the indicator random variables $(1[\bsigma_i(j) = k])_{i \in \arcs}^{j,k=1 \dots n}$ for $\blift_n$.  Thus \Cref{fact:t-wise-lift} tells us, whenever $m \leq 4\lambda$, that $\E[\bX_H]$ has the same value under our $4\lambda$-wise uniform~$\blift_n$ as it has in~\cite[Lem.~23]{BC19}.

    In \cite[Lem.~23]{BC19}, the following is first shown: provided $\ell \leq \sqrt{n}$, the expected number of cycles of length $\ell$ in $G_{\blift_n}$ is at most~$O((d+2\q)^\ell)$.  Using this fact for all $\ell \leq 2\lambda$, it is directly derived in~\cite[Cor.~8]{BC19} that the expected number of vertices $v \in G_{\blift_n}$ that have a cycle in their $\lambda$-neighborhood is at most $O((d+2\q)^{3\lambda})$.  Now for sufficiently large constants~$C_0,C_1$ we will have $2\lambda \leq \sqrt{n}$ and $O((d+2\q)^{3\lambda}) \leq \frac12 n^{.001}$.  Thus we conclude this expectation is at most $\frac12 n^{.001}$ under our $4\lambda$-wise uniform~$\blift_n$, and hence by Markov, \Cref{enum:basic-G1} holds except with probability at most $\frac12 n^{-.999}$.

    Regarding \Cref{enum:basic-G2}, it is noted in \cite[Lem.~23]{BC19} that if $G_{\blift_n}$ is \emph{not} $\lambda$-bicycle free then it must contain a ``witness subgraph'' to this fact (either a ``handcuffs graph'' or a ``theta graph'') of at most $4\lambda$~edges.  Provided $4\lambda \leq \sqrt{n}$, they bound the expected number of such witness graphs (for truly uniform~$\blift_n$) by $O(\lambda^3 (d+2\q)^{4\lambda}/n)$.  Similar to before, for sufficiently large constants~$C_0,C_1$ we will have $4\lambda \leq \sqrt{n}$ and $O(\lambda^3 (d+2\q)^{4\lambda}/n) \leq \frac12 n^{-.999}$.  Thus we conclude the expected number of witness graphs is at most $\frac12 n^{-.999}$ even under our $4\lambda$-wise uniform~$\blift_n$, and thus the probability of $G_{\blift_n}$ not being $\lambda$-bicycle free is at most $\frac12 n^{-.999}$.
\end{proof}

\subsection{Weakly derandomizing the main argument}

In this setting, let us restate (a slight variant of) Bordenave and Collins's main theorem on the nonbacktracking operator of random $n$-lifts~\cite[Thm.~17]{BC19}:
\begin{theorem}                                     \label{thm:BC-main}
    Let $d$, $\q$, $r$, $R$, $\ol{\eps} > 0$ be constants.  Then there are constants $C_0$, $C_1 > 0$ such that the following holds:

    Let $n \geq C_0$, let $\lambda = \lfloor (\log n)/C_1\rfloor$, and let $\Xi_\lambda \subset \mathfrak{K}_{2R}$ be as in \Cref{prop:eps-net}, with $\eps =\ol{\eps}/3$.  Then except with probability at most~$n^{-.99}$ over the choice of a uniformly random $n$-lift~$\blift_n$, the following all hold:  \Cref{enum:basic-G1,enum:basic-G2} of \Cref{prop:easy-graph-derand}, and also for all matrix bouquets $\calK \in \Xi_\lambda$,
    \begin{equation}    \label[ineq]{eqn:BC-main1}
        \norm{B_{n,\bot}(\blift_n, \calK)^\lambda}_\opnorm \leq (\rho(B_\infty(\calK)) + \ol{\eps}/3)^\lambda.
    \end{equation}
    When this last event happens we may easily deduce from \Cref{prop:eps-net} that
    \begin{equation} \label[ineq]{eqn:BC-main-cor}
        \rho(B_{n,\bot}(\blift_n, \calK))^\lambda \leq \norm{B_{n,\bot}(\blift_n, \calK)^\lambda}_\opnorm \leq (\rho(B_\infty(\calK)) + \ol{\eps})^\lambda
    \end{equation}
    holds simultaneously for all $R$-bounded bouquets $\calK \in \mathfrak{K}_{R}$.
\end{theorem}

Our goal in this section is to show the following:
\begin{theorem}                                     \label{thm:BC-main-derand}
    There is a constant $C = C(d, \q, r, R, \ol{\eps})$ such that \Cref{thm:BC-main} continues to hold only assuming that $\blift_n$ is $(C\log n)$-wise uniform.
\end{theorem}
The key insight is that the only probabilistic component of \Cref{thm:BC-main} (besides
\Cref{prop:easy-graph-derand}) is a bound (\cite[Prop.~22]{BC19}) on the expectation of certain polynomials of degree~$O(\log n)$ in the lift's indicator random variables~$1[\bsigma_i(j) = k]$.  Unfortunately, to carefully verify this requires recapping details from~\cite{BC19} in a somewhat black-box fashion.

We start by remarking that the way~\cite[Thm.~17]{BC19} is originally stated only involves the ($\lambda$th root of) \Cref{eqn:BC-main-cor}; i.e., it only involves bounding $\rho(B_{n,\bot}(\blift_n, \calK))$.  However, their proof immediately begins (see \cite[(27)]{BC19}) by passing to bounding $\norm{B_{n,\bot}(\blift_n, \calK)^\lambda}_\opnorm$ for the given choice of $\lambda$ (which they call~``$\ell$'').  The bulk of the work is of course to bound $\norm{B_{n,\bot}(\blift_n, \calK)^\lambda}_\opnorm$ for a \emph{fixed} $\calK = (a_1, \dots, a_{d+2\q}) \in \mathfrak{K}_{2R}$ (where our~$2R$ essentially corresponds to their ``$\eps^{-1}$''). After this, the claim for all $\calK \in \Xi_\lambda$ follows from a union bound.  Indeed, for fixed $\calK \in \mathfrak{K}_{2R}$ we will be able to show that \Cref{eqn:BC-main1} holds except with probability $O(n^{-.999})$, provided $C_0, C_1$ are large enough --- even for $(C\log n)$-wise uniform~$\blift_n$.  Since $|\Xi_\lambda| \leq \kappa(d,\q,r,R,\ol{\eps})^{\lambda}$, by increasing~$C_1$ (and~$C_0$) if necessary, we can ensure that $|\Xi_\lambda| \cdot O(n^{-.999}) \leq n^{-.99}$, completing the proof.

Summarizing, our goal is to show:
\begin{multline}    \label{eqn:proveee}
    \text{For fixed $\calK = (a_1, \dots, a_{d+2\q}) \in \mathfrak{K}_{2R}$,} \\
    \text{\Cref{eqn:BC-main1} holds  except with probability $O(n^{-.999})$ when $\blift_n$ is $(C\log n)$-wise uniform.}
\end{multline}

The operator $B_{n,\bot}(\blift_n, \calK)^\lambda$ occurring in \Cref{eqn:BC-main1} is hard to control directly. But using a key idea from~\cite{Bor19}, it is shown in~\cite[Sec.~4.2]{BC19} that there are certain related operators $\ul{\bB}^{(\lambda)}$~and~$\bR^{(\lambda)}_k$ (defined in~\cite[(32),~(34)]{BC19}) such that the following holds: \emph{if} $G_{\blift_n}$ is $\lambda$-bicycle free, then
\[
    \norm{B_{n,\bot}(\blift_n, \calK)^\lambda}_\opnorm \leq \|\ul{\bB}^{(\lambda)}\|_\opnorm + \frac1n \sum_{k=1}^\lambda \|\bR_k^{(\lambda)}\|_\opnorm.
\]
By  \Cref{prop:easy-graph-derand} (and using $C\log n \gg 4 \lambda$), we may indeed assume $G_{\blift_n}$ is $\lambda$-bicycle free for the purposes of proving \Cref{eqn:proveee} (and \Cref{prop:easy-graph-derand} also establishes that \Cref{enum:basic-G1,enum:basic-G2} also hold with the required probability, as needed for \Cref{thm:BC-main-derand}).  Thus, writing $\rho^* = \rho(B_\infty(\calK))$,  it remains to show that
\begin{equation}    \label[ineq]{eqn:last}
    \|\ul{\bB}^{(\lambda)}\|_\opnorm + \frac1n \sum_{k=1}^\lambda \|\bR_k^{(\lambda)}\|_\opnorm \leq  (\rho^* + \ol{\eps}/3)^\lambda
\end{equation}
except with probability $O(n^{-.999})$ for $(C \log n)$-wise uniform~$\blift_n$.\footnote{In fact, we can show it holds except with probability~$n^{-100}$.  The bottleneck for showing low failure probability is the probability of $G_{\blift_n}$ failing to be $\lambda$-bicycle free.}

As mentioned earlier, the only probabilistic component of Bordenave--Collins's proof of this (for truly uniform~$\blift_n$) is the below proposition (a slightly restated version of~\cite[Prop.~22]{BC19}); we leave some of the terminology in it undefined.
\begin{proposition}                                     \label{prop:22}
    (\cite[Prop.~22]{BC19}, following directly from~\cite[Props.~11,~28]{Bor19}.)  For a fixed sequence $f = (f_1, \dots, f_\tau)$ of potential ``colored edges'', $f_t = (x_t, y_t) \in [n] \times [n]$ of color $i_t \in \arcs$, and assuming $\tau_0 \leq \tau \leq \sqrt{n}$, if $\blift_n = (\bbi, \bsigma_1, \dots, \bsigma_{d+2\q})$ is a random lift, then
    \[
        \abs{\E\bracks*{\prod_{t=1}^{\tau_0} (1[\bsigma_{i_t}(x_t) = y_t] - 1/n) \prod_{t=\tau_0+1}^{\tau} 1[\bsigma_{i_t}(x_t) = y_t]}} \leq c 2^b (1/n)^m (3\tau/\sqrt{n})^{m_1},
    \]
    where $c$ is a universal constant, $b$ is the number of ``inconsistent'' edges in~$f$, $m$~is the number of distinct edges in~$f$, and $m_1$ is the number of ``consistent edges'' of multiplicity~$1$ in~$f$.
\end{proposition}
The only thing we need to take away from this is the following:
\begin{fact}                                        \label{fact:prop22}
    \Cref{prop:22} holds equally well if $\blift_n$ is merely $\tau$-wise uniform, by \Cref{fact:t-wise-lift}.
\end{fact}

Bordenave and Collins only ever apply \Cref{prop:22} (twice) with $\tau = 2 \lambda m$, where ``$m$'' is an integer parameter they eventually set to $\Theta((\log n)/(\log \log n))$. To save on the amount of $t$-wise uniformity needed, though, we will eventually set~$m$ to be a large constant depending on $d, \q, r, R, \ol{\eps}$, and also~$C_1$.  With this setting, we will indeed have $\tau = 2\lambda m \leq \sqrt{n}$ (provided $C_0$ is large enough) and also that $(C \log n)$-wise uniformity of~$\blift_n$ suffices.

Inspecting the proof of~\cite[Props.~24,~28]{BC19}, particularly~\cite[(52),~(56)]{BC19}\footnote{Though~$(52)$ has a typo; its exponent on ``$\rho$'' should be ``$2\ell m$'' not ``$\ell m$''.}, we see that it can use \Cref{prop:22} to deduce (for a constant $c = c(d, \q, r, R, \ol{\eps})$) that
\[
    \E[\|\ul{\bB}^{(\lambda)}\|_\opnorm^{2m}] \leq n(c \lambda m)^{10m} \cdot (\rho^* + \ol{\eps}/4)^{2\lambda m}, \qquad
    \E[\|\ul{\bR}_k^{(\lambda)}\|_\opnorm^{2m}] \leq (c \lambda m)^{32m} \cdot c^{2\lambda m},
\]
provided~$n \geq n_0(d, \q, r, R, \ol{\eps})$ and $\lambda \leq \log n$ (both of which we are assuming) and also provided $(c \lambda m)^{12m}/n \leq 1/2$ (which also holds given our choice of~$m$, assuming $C_0$ is large enough).  Thus by Markov's inequality, except with probability at most~$(\lambda + 1)n^{-1} = O(n^{-.999})$ over the choice of $(C \log n)$-wise uniform~$\blift_n$, we have all of
\begin{align*}
    \|\ul{\bB}^{(\lambda)}\|_\opnorm^{2m} & \leq n^2 (c \lambda m)^{10m} \cdot (\rho^* + \ol{\eps}/4)^{2\lambda m},
    &  \|\ul{\bR}_k^{(\lambda)}\|_\opnorm^{2m} &\leq n(c \lambda m)^{32m} \cdot c^{2\lambda m} \quad \forall k \leq \lambda \\
    \implies \|\ul{\bB}^{(\lambda)}\|_\opnorm &\leq  n^{1/m} (c \lambda m)^{5} \cdot (\rho^* + \ol{\eps}/4)^{\lambda}
    & \implies \frac{1}{n} \sum_{k=1}^\lambda \|\ul{\bR}_k^{(\lambda)}\|_\opnorm  &\leq n^{1/(2m)} (c \lambda m)^{16} \cdot \frac{\lambda c^{\lambda}}{n} \\
    &= \parens*{n^{1/(\lambda m)} (c \lambda m)^{5/\lambda}}^\lambda \cdot (\rho^* + \ol{\eps}/4)^{\lambda}.
    & &
\end{align*}
Now for any constant $\delta = \delta(d, \q, r, R, \ol{\eps}) > 0$, and given the constant~$C_1$, we can ensure $\|\ul{\bB}^{(\lambda)}\|_\opnorm \leq (1 + \delta)^\lambda \cdot (\rho^* + \ol{\eps}/4)^{\lambda}$ by taking~$m$ a sufficiently large constant (and $C_0$ large enough).  Further, we can ensure  $\frac{1}{n} \sum_{k=1}^\lambda \|\ul{\bR}_k^{(\lambda)}\|_\opnorm \leq \delta^\lambda$ by taking~$C_1$ large enough; then irrespective of~$m$'s constant value, it suffices to take~$C_0$ large enough.  We conclude that by setting constants appropriately, we can ensure
\[
    \|\ul{\bB}^{(\lambda)}\|_\opnorm + \frac{1}{n} \sum_{k=1}^\lambda \|\ul{\bR}_k^{(\lambda)}\|_\opnorm \leq (1+\delta)^\lambda (\rho^* + \ol{\eps}/4)^{\lambda} + \delta^\lambda,
\]
which is enough for~\Cref{eqn:last} if $\delta$ is taken small enough (recalling that \Cref{cor:rhoB-bound} upper-bounds~$\rho^*$ by a constant).

\section{Constructing explicit good lifts}  \label{sec:good-lifts}

We remark that except for the analysis of the net, the below theorem is not substantially different from its analogue in~\cite[Sec.~5]{MOP20b}.
\begin{theorem}                                     \label{thm:weakly-explicit-main}
    Fix an index set $\arcs = \{1, \dots, d+2\q\}$.
    Fix also constants $r \in \N^+$ and $R, \eps > 0$.
    Then for any constant $C' > 0$, there is a $\poly(N)$-time deterministic algorithm that, on input~$N$, outputs an $N'$-lift $\lift_{N'}$ (with $N \leq N' \leq N + o(N)$) such that:
    \begin{itemize}
        \item $G_{\lift_{N'}}$ has at least one vertex whose radius-$C'\sqrt{\log N}$ neighborhood is acyclic;
        \item for every $R$-bounded bouquet $\calK \in \mathfrak{K}_{R}$ with $r'$-dimensional matrix-weights ($r' \leq r$), we have
        \[
            \rho(B_{N',\bot}(\lift_{N'},\calK)) \leq \rho(B_\infty( \calK)) + \eps.
        \]
    \end{itemize}
\end{theorem}
\begin{proof}
    We may assume $R$ and $\eps$ are rational, and it suffices to assume (as we will tacitly do a couple of times) that $N$ is at least a sufficiently large constant.  We will prove the result just for $r$-dimensional bouquets; the fact that it works simultaneously for all $r' \leq r$ is a small twist we will comment on at the end of the proof.

    Given $N$, the algorithm will choose an even integer~$n_0$ on the order of $2^{\Theta(\sqrt{\log N})}$ and then define $n_i = 2^i n_0$ for each $1 \leq i \leq t = \lceil \log_2(N/n_0) \rceil$.  The value~$N'$ will be equal to~$n_t$; it is clear that $N' \geq N$, and it is easy for the algorithm to arrange that $N' \leq N(1+2^{-\Theta(\sqrt{\log N})})$ by adjusting $n_0$'s value by at most a factor of~$2$.  The idea is that the algorithm will sequentially produce $n_i$-lifts $\lift_{n_i}$ for $i = 0 \dots t$.  The $0$th one $\lift_{n_0}$ will be produced using \Cref{thm:BC-main-derand} and each subsequent $\lift_{n_i} = \lift_{2n_{i-1}}$ will be produced via \Cref{cor:main} as an edge-signing of~$\lift_{n_{i-1}}$ (recalling from \Cref{sec:signed-perms} that edge-signing a lift corresponds to forming a product with $2$-lifts).  As these lifts are formed, the algorithm will maintain the following invariants:
    \begin{enumerate}
        \item \label{enum:inv1} $G_{\lift_{n_i}}$ has at least one vertex whose radius-$C'\sqrt{\log N}$ neighborhood is acyclic;
        \item \label{enum:inv2} $G_{\lift_{n_i}}$ is $C'\sqrt{\log N}$-bicycle free;
        \item \label{enum:inv3} for every $R$-bounded bouquet $\calK \in \mathfrak{K}_{R}$ with $r$-dimensional matrix-weights we have
        \[
            \rho(B_{n_i,\bot}(\lift_{n_i},\calK)) \leq \rho(B_\infty(\calK)) + \eps.
        \]
    \end{enumerate}
    Note that as soon as invariant \Cref{enum:inv1,enum:inv2} hold for~$n_0$, they automatically hold for all subsequent~$n_i$.  The reasons are that the $2$-lift of any an acyclic radius-$\lambda$ neighborhood in $G_{\lift_{n_{i-1}}}$ becomes two acyclic radius-$\lambda$ neighborhoods in $G_{\lift_{n_{i}}}$; and, every $2$-lift of a $\lambda$-bicycle free graph is $\lambda$-bicycle free (\cite[Prop.~2.12]{MOP20b}).  Thus it suffices for the algorithm to ensure that all three invariants hold for the $0$th lift $\lift_{n_0}$, and then that the $3$rd invariant holds for all subsequent lifts.

    \paragraph{The $0$th lift.} As mentioned, the $0$th lift $\lift_{n_0}$ is constructed via \Cref{thm:BC-main-derand}, with its ``$\ol{\eps}$'' parameter set to $\eps/2$ and its ``$n$'' parameter set to $n_0 = 2^{\Theta(\sqrt{\log N})}$ large enough so that the resulting value of $\lambda$ is at least $C'\sqrt{\log N}$ (and so that $n_0 \geq C_0$).  The theorem tells us that if $\blift_{n_0}$ is a $t$-wise uniform $n_0$-lift, $t = C \log n_0 = O(\sqrt{\log N})$, then except with probability at most $n_0^{-.99} = 2^{-\Theta(\sqrt{\log N})}$, invariant \Cref{enum:inv1,enum:inv2} hold, as does the following:
    \begin{equation}     \label[ineq]{eqn:check0}
        \forall \calK \in \Xi_\lambda, \qquad \norm{B_{n_0,\bot}(\blift_{n_0}, \calK)^\lambda}_\opnorm \leq (\rho(B_\infty(\calK)) + \eps/6)^\lambda.
    \end{equation}
    As noted in \Cref{thm:BC-main}, \Cref{eqn:check0} has invariant \Cref{enum:inv3} as a consequence (even if it only has $\eps/3$ in place of $\eps/6$).  A crucial point is that a deterministic $\poly(N)$-time algorithm can \emph{check}, given a realization of $\blift_{n_0}$, that invariant \Cref{enum:inv1,enum:inv2} hold (just use breadth-first search), and it can also \emph{check} that \Cref{eqn:check0} holds (with $\eps/3$ in place of~$\eps/6$).  To see this last claim, note first from \Cref{prop:eps-net} that the bouquets $\calK$ in the net $\Xi_\lambda$ can be enumerated in $\poly(N)$ time (in fact, in subpolynomial time, $O(1)^{\lambda} = 2^{O(\sqrt{\log N})}$).  Next note that each matrix $B = B_{n_0,\bot}(\blift_{n_0}, \calK)^\lambda$ can straightforwardly be computed exactly, as can be $S = B^\conj B$.  Observe that $\|B\|_\opnorm \leq \beta$ iff $\|S\|_\opnorm \leq \beta^2$ iff $\beta^2 \Id - S \succeq 0$, and the last condition (testing if a Hermitian matrix is positive semidefinite) can be checked efficiently~\cite[pg.~295]{GLS93} for any given~$\beta$.  Finally, note that each $\rho(B_\infty(\calK))$ can be efficiently computed to additive accuracy~$\eps/6$, by \Cref{prop:sqrtL} (in constant time, in fact).  Having established that \Cref{eqn:check0} can be efficiently and deterministically checked (up to replacing $\eps/6$ with $\eps/3$, which is sufficient for invariant \Cref{enum:inv3}), it remains to note that using \Cref{thm:prg-lift} (with its ``$t$'' parameter set to $C \log n_0$), the algorithm can enumerate all lifts $\lift_{n_0} \in \Lambda$ in $\poly(n_0^{t}) = \poly(N)$ time, check all invariants for each, and we are assured that at least one (in fact, all but an $n_0^{-.99} + n_0^{-100t}$ fraction) will satisfy the invariants.

    \paragraph{The subsequent lifts.} As noted, we only need to show that for each $1 \leq i \leq t$, the algorithm can deterministically and efficiently find a $2$-lift $\lift_{n_i}$ of~$\lift_{n_{i-1}}$ such that invariant \Cref{enum:inv3} holds, presuming that it holds for all smaller~$i$.  By \Cref{prop:B-spec-identities}, \Cref{item:B-spec-perp}, this is equivalent to finding an edge-signing $\chi$ of $\lift_{n_{i-i}}$ such that
        \begin{equation} \label[ineq]{eqn:verifyme}
            \rho(B_{n_{i-1}}(\chi \lift_{n_{i-1}},\calK)) \leq \rho(B_\infty(\calK)) + \eps \qquad \forall \calK \in \mathfrak{K}_R.
        \end{equation}
    To do this, the algorithm employs \Cref{cor:main} with, say, $\ol{\eps} = \eps/6$ and $p = 1$.  A key point is that $G_{\lift_{n_{i-1}}}$ is $C'\sqrt{\log N}$-bicycle free by invariant \Cref{enum:inv3}, and this is at least $C_2(\log \log n_i)^2$ even for the largest $n_t = N'$, as needed for \Cref{cor:main} (presuming $C'$ is taken large enough).  Similarly, even for the largest~$n_t$, \Cref{cor:main} only requires $(1/\poly(N), O(\log N))$-wise uniform signings of $O(N)$ bits, and by \Cref{thm:nn93} the algorithm can enumerate over a $\poly(N)$-size set of such signings.   \Cref{cor:main} tells us that for all but an $n_i^{-p} \leq n_0^{-1} = 2^{-\Theta(\sqrt{\log N})}$ fraction of such signings, the resulting $\chi \lift_{n_{i-1}}$ has
    \[
        \|B_{n_{i-1}}(\chi \lift_{n_{i-1}},\calK))\|_\opnorm \leq \rho(B_\infty(\calK)) + \eps/6
    \]
    for all $\calK \in \Xi_1$.  As with the $0$th signing, the algorithm can deterministically and efficiently verify this (with $\eps/3$ in place of $\eps/6$) by enumerating over all $\calK \in \Xi_1$ and calculating; and again, this in turn implies \Cref{eqn:verifyme}.\\

    This concludes the proof, except for verifying the algorithm can find a lift that works simultaneously for all $r' \leq r$.  But this only requires the algorithm to check all the nets in dimensions $1 \leq r' \leq r$ (only a constant-factor running time increase) and to note that in the failure probability analysis, union-bounding over all~$r'$ has a negligible effect.
\end{proof}

In the above proof, note that if we use one set of $(\delta,t)$-wise uniform bits for all the edge-signings ($1 \leq i \leq t$), then by a union-bound the probability of any failure is $2^{-\Theta(\sqrt{\log N})}$.  Furthermore, the total ``seed-length'' used is $O(\log N)$.  Thus using the ``strongly explicit'' aspect of \Cref{thm:prg-perm,thm:prg-lift}, and omitting all the deterministic ``checks'', we obtain the following ``probabilistically strongly explicit'' construction, just as in~\cite[App.~B]{MOP20b}:
\begin{theorem}                                     \label{thm:probabilistic-explicit-main}
    In the setting of \Cref{thm:weakly-explicit-main}, there is also an algorithm that takes as input a number~$N$ and a seed $s \in \{0,1\}^{O(\log N)}$, and in deterministic $\polylog(N)$ time outputs a binary circuit~$\mathfrak{C}$ that implements the adjacency list of a ``color-regular lift graph'' $G_{\lift_{N'}}$ (with $N \leq N' \leq N + o(N)$).  Furthermore, with high probability over the choice of a uniformly random~$s$ (namely, except with probability $2^{-\Theta(\sqrt{\log N})}$), the resulting $G_{\lift_{N'}}$ satisfies the conclusions of \Cref{thm:weakly-explicit-main}.
\end{theorem}

\section{Relating the spectra of the adjacency operator and nonbacktracking operator}\label{sec:ihara-bass}
In this section we will relate the spectra of $A_n$ and $B_n$.
We require the following fact about spectra of possibly infinite-dimensional operators on Hilbert spaces (see e.g.~\cite[VII.5]{DS88a}).
\begin{fact}\label{fact:spectrum}
Let $T:\scrH \to \scrH$ be a bounded linear operator on a Hilbert space $\scrH$. The spectrum of $T$ can be decomposed into three disjoint sets. Suppose $\lambda \in \sigma(T)$. Then one of the following hold:
\begin{enumerate}
\item $\lambda$ is in the \emph{point spectrum} where $\lambda I - T$ is not injective. In this case, $\lambda$ is an eigenvalue and there exists an eigenvector $v \in \scrH$ such that $\lambda v = T v$.
\item $\lambda$ is in the \emph{continuous spectrum} where $\lambda I - T$ is not surjective, but the range of $\lambda I - T$ is dense in $\scrH$. In this case, for every $\eps > 0$ there exists an approximate eigenvector $v_\eps \in \scrH$ with $\norm{v_\eps} = 1$ such that $\norm{\lambda v_\eps - Tv_\eps} \leq \eps$, where $\norm{\cdot}$ denotes the Hilbert space norm.
\item $\lambda$ is in the \emph{residual spectrum} where $\lambda I - T$ is not surjective and its range is not dense in $\scrH$. In this case, $\ov \lambda$ is an eigenvalue of the adjoint $T^*$, i.e.\ there exists $v \in \scrH$ such that $\ov\lambda v = T^* v$.
\end{enumerate}
\end{fact}

The following is essentially~\cite[Prop.~9]{BC19}.
\begin{proposition}\label{prop:9}
Let $\calK = (a_0,\dots,a_{d+2\q})$ be a matrix bouquet, and let $\lift_n$ be an $n$-lift (possibly with $n = \infty$). Let $\lambda \in \C$ such that $\lambda^2 \notin \sigma(a_{i^*}a_i)$ for all $1 \leq i \leq d+2\q$. We define $\calK_\lambda = (a_0(\lambda),\dots,a_{d+2\q}(\lambda))$ where the coefficients are defined
\[
     a_0(\lambda) = -1 - \sum_{i=1}^{d+2\q} a_{i} (\lambda^2-a_{i^*}a_{i})\inv a_{i^*} \quad \text{and} \quad a_i(\lambda) = \lambda a_{i} (\lambda^2 - a_{i^*}a_{i})\inv.
\]
Note that $\calK_\lambda$ does not satisfy the symmetry conditions of matrix bouquets. We can nevertheless define the $n$-lift $A_n(\lift_n,\calK_\lambda)$.
Then, $\lambda \in \sigma(B_n(\lift_n,\calK))$ if and only if $0 \in \sigma(A_n(\lift_n,\calK_\lambda))$. Moreover, for $n < \infty$, $\lambda \in \sigma(B_{n,\bot}(\lift_n, \calK))$ if and only if $0 \in \sigma(A_{n,\bot}(\lift_n, \calK_\lambda))$.
\end{proposition}
\begin{proof}
In this proof, when $A_n$ (an operator on $\ell_2(V_n) \otimes \C^r$) acts on a vector $v$, we think of $v$ as being indexed by $V_n$, and each entry is an element in $\C^r$. We write $v_x$ for $x \in V_n$ to indicate the block indexed by $x$. Similarly, when $B_n$ acts on a vector $v$, we write $v_{x,i} \in \C^r$ for $x \in V_n$ and $i \in \arcs$ to index into $v$.

Let $\lambda \in \sigma(B_n(\lift_n,\calK))$.
We prove the statement assuming there exists some $i \in \arcs$ such that $a_i$ is invertible, and use a density argument. Suppose no $a_i$ is invertible. Then, let $\eps > 0$ and let $i \in \arcs$ be an index where $a_i$ is nonzero.
Suppose $\lambda \in \sigma(B_n(\lift_n, \calK_\lambda))$, and let $A_n = A_n(\lift_n, \calK_\lambda)$ be the $n$-lift of $\calK_\lambda$.
Let $\calK_\eps$ be defined the same as $\calK_\lambda$ but with $a_i + \eps I_r$ in place of $a_i$. Let $A_\eps = A_n(\lift_n, \calK_\eps) - A_n(\lift_n, \calK_\lambda)$.
Let $\delta > 0$. In our proof below, in all cases we will construct a vector $u_\delta \in V_n \otimes \C^r$ such that $\norm{u_\delta} \leq C$ for a universal constant $C$ and $\norm{A_\eps u_\delta} \leq \delta$. Since $\norm{A_\eps u_\delta} \leq \norm{A_n u_\delta} + O(\eps)$, taking $\eps \to 0$ shows $0$ is in the spectrum of $A_n$.

We split the proof into cases based on the classification in \Cref{fact:spectrum}.

\paragraph*{Case 1:} First suppose $\lambda$ is in the point spectrum of $B_n$. Then there exists $v$ such that $\lambda v = B_nv$. Our strategy is to construct a $u \in V_n \otimes \C^r$ such that $u$ is in the kernel of $A_n(\lift_n,\calK_\lambda)$, which shows $0$ is an eigenvalue of $A_n(\lift_n,\calK_\lambda)$.

Fix $x \in V_n$ and $i \in \arcs$ such that $a_i$ is invertible.
By definition of $B_n$,
\begin{equation}\label{eqn:v}
    \lambda v_{x,i} = \bra{i}\bra{x} B_n v = \bra{i}\bra{x} \sum_{y \in V_n} \sum_{j,k=1}^{d+2\q} \bone[k \ne j^*] \ket{\sigma_j(y)}\ket{k}\otimes a_k v_{y,j} = \sum_{j \ne i^*} a_i v_{\sigma_{i^*}(x),j}.
\end{equation}
Our candidate for $u$, defined coordinate-wise, is
\begin{equation}\label{eqn:u}
    u_x = \sum_{j = 1}^{d+2\q} v_{x,j}.
\end{equation}
We can then write
\[
    \lambda v_{x,i} = a_iu_{\sigma_{i^*}(x)} - a_i v_{\sigma_{i^*}(x), i^*}.
\]
Plugging in $(\sigma_{i^*}(x),i^*)$ in place of $(x,i)$ in the above, we get
\[
    \lambda v_{\sigma_{i^*}(x),i^*} = a_{i^*} u_x - a_{i^*}v_{x,i^*}.
\]
Multiplying \Cref{eqn:v} by $\lambda$, we get
\[
    \lambda^2 v_{x,i} = \lambda a_i u_{\sigma_{i^*}(x)} - \lambda a_i v_{\sigma_{i^*}(x),i^*} = \lambda a_i u_{\sigma_{i^*}(x)} - a_ia_{i^*} u_x + a_i a_{i^*} v_{x,i}.
\]
By assumption, $(\lambda^2 - a_i a_{i^*})$ is invertible. Therefore we can rearrange the above to
\begin{equation}\label{eqn:u-v-relation}
    v_{x,i} = \lambda(\lambda^2 - a_{i}a_{i^*})\inv a_i u_{\sigma_{i^*}(x)} - (\lambda^2 -a_i a_{i^*})\inv a_ia_{i^*} u_x.
\end{equation}
Now let $y = \sigma_{i^*}(x)$. Substituting \Cref{eqn:u-v-relation} into \Cref{eqn:v}, we get
\begin{align}
    \lambda v_{x,i} &= \lambda^2(\lambda^2-a_{i}a_{i^*})\inv a_i u_{\sigma_{i^*}(x)} - \lambda (\lambda^2-a_ia_{i^*})\inv a_i a_{i^*} u_x \nonumber \\
    &= \sum_{j \ne i^*} a_i v_{y,j} \nonumber \\
    &= \sum_{j \ne i^*} a_i \lambda (\lambda^2-a_ja_{j^*})\inv a_j u_{\sigma_{j^*}(y)} - \sum_{j \ne i^*} a_i (\lambda^2 - a_j a_{j
    ^*})\inv a_j a_{j^*} u_y. \label{eqn:u-v-relation2}
\end{align}
We now note the following identities:
\begin{align}
    (\lambda^2 - a_i a_{i^*})\inv a_i &= (\lambda^2 - a_i a_{i^*})\inv a_i (\lambda^2 - a_{i^*} a_{i}) (\lambda^2 - a_{i^*} a_i)\inv \nonumber\\
    &= (\lambda^2 - a_i a_{i^*})\inv (\lambda^2 - a_{i} a_{i^*})a_i (\lambda^2 - a_{i^*} a_i)\inv \nonumber \\
    &= a_i (\lambda^2 - a_{i^*} a_i)\inv, \label{eqn:resolvent-commute}
\end{align}
and using the above,
\begin{align}
    \lambda^2(\lambda^2 - a_ia_{i^*})\inv - a_i (\lambda^2 - a_{i^*}a_{i})\inv a_{i^*}
    &= \lambda^2(\lambda^2 - a_ia_{i^*})\inv - a_ia_{i^*} (\lambda^2 - a_{i}a_{i^*})\inv \nonumber \\
    &= (\lambda^2-a_ia_{i^*})(\lambda^2-a_ia_{i^*})\inv = 1. \label{eqn:resolvent-id}
\end{align}
In the two terms on the right hand side of \Cref{eqn:u-v-relation2}, if we set $j = i^*$ and use \Cref{eqn:resolvent-commute} we get
\[
    a_i \lambda (\lambda^2 a_{i^*}a_i)\inv a_i^* u_{\sigma_i(y)} = \lambda (\lambda^2-a_{i}a_{i^*})\inv a_i a_{i^*} u_x
\]
and by using \Cref{eqn:resolvent-id} we get
\[
    a_i(\lambda^2 - a_{i^*}a_i)\inv a_{i^*} a_i u_y = a_ia_{i^*}(\lambda^2 - a_i a_{i^*})\inv a_i u_y = \lambda^2 (\lambda^2 - a_{i}a_{i^*})\inv a_i u_y - a_i u_y.
\]
Therefore, using \Cref{eqn:resolvent-commute,eqn:resolvent-id}, we can rearrange \Cref{eqn:u-v-relation2} to get
\[
    \lambda^2(\lambda^2 - a_i a_{i^*})\inv a_i u_{\sigma_{i^*}(x)} = \sum_{j} a_i \lambda(\lambda^2 - a_j a_{j^*})\inv a_j u_{\sigma_{j^*}(y)} - \sum_{j\ne i^*} a_i(\lambda^2 - a_j a_{j^*})\inv a_j a_{j^*} u_y,
\]
and then
\[
    a_i \parens*{1+\sum_j a_j(\lambda^2 - a_{j^*}a_j)\inv a_{j^*}}u_y = a_i \sum_j \lambda a_j(\lambda^2 - a_{j^*} a_{j})\inv u_{\sigma_{j^*}(y)}.
\]
Since we assumed $a_i$ is invertible, this implies
\[
    \parens*{1+\sum_j a_j(\lambda^2 - a_{j^*}a_j)\inv a_{j^*}}u_y = \sum_j \lambda a_j(\lambda^2 - a_{j^*} a_{j})\inv u_{\sigma_{j^*}(y)}.
\]
This holds for all $x \in V_n$. Observing that
\[
    \bra{x} A_n(\lift_n, \calK_\lambda) u = \sum_{i=0}^{d+2\q} a_i(\lambda) u_{\sigma_{i^*}(x)},
\]
we conclude that $0$ is in the kernel of $A_n(\lift_n, \calK_\lambda)$.

\paragraph*{Case 2:} If $\lambda$ is in the continuous spectrum of $B_n$, then for every $\eps > 0$ there exists a $v_\eps \in V_n \otimes \arcs \otimes \C^r$ such that $\norm{v_\eps} = 1$ and $\norm{\lambda v_\eps - B_n v_\eps} \leq \eps$. The above construction applied to $v_\eps$ gives us a vector $u_\eps$ such that $\norm{A_n u_\eps} \leq \eps$, and by scaling $u$ appropriately we can conclude that $0$ is in the continuous spectrum of $A_n(\lift_n, \calK_\lambda)$.

\paragraph*{Case 3:} Suppose $\lambda$ is in the residual spectrum of $B_n$. Then there exists $v$ such that $\ov\lambda v = B_n^* v$. A similar proof as Case 1 holds here.
Fix $x \in V_n$ and $i \in \arcs$.
We find that
\begin{equation}\label{eqn:B-star-v}
    \ov \lambda v_{x,i} = \bra{i}\bra{x} B_n v = \bra{i}\bra{x} \sum_{y \in V_n} \sum_{j,k=1}^{d+2\q} \bone[k \ne j^*] \ket{\sigma_{j^*}(y)}\ket{j}\otimes a_{k^*} v_{y,k} = \sum_{j \ne i^*} a_{j^*} v_{\sigma_{i}(x),j}.
\end{equation}
Our candidate for $u$ here is
\begin{equation}\label{eqn:star-u}
    u_x = \sum_{j = 1}^{d+2\q} a_{j^*} v_{x,j}.
\end{equation}
Similar to Case 1, we get
\begin{equation}\label{eqn:u-v-relation-star}
    v_{x,i} = \ov\lambda(\lambda^2 - a_{i}a_{i^*})\inv u_{\sigma_{i}(x)} - (\lambda^2 -a_i a_{i^*})\inv a_i u_x.
\end{equation}
Now let $y = \sigma_{i}(x)$. Substituting \Cref{eqn:u-v-relation-star} into \Cref{eqn:B-star-v}, we get
\begin{align}
    \ov\lambda v_{x,i} &= \lambda^2(\lambda^2-a_{i}a_{i^*})\inv u_{\sigma_{i}(x)} - \ov\lambda (\lambda^2-a_ia_{i^*})\inv a_i u_x \nonumber = \sum_{j \ne i^*} a_{j^*} v_{y,j} \nonumber \\
    &= \sum_{j \ne i^*} a_{j^*} \ov \lambda (\lambda^2-a_ja_{j^*})\inv u_{\sigma_{j}(y)} - \sum_{j \ne i^*} a_{j^*} (\lambda^2 - a_j a_{j
    ^*})\inv a_j u_y. \label{eqn:u-v-relation-star2}
\end{align}
Therefore, using \Cref{eqn:resolvent-commute,eqn:resolvent-id}, we can rearrange \Cref{eqn:u-v-relation-star2} to get
\[
    1+\sum_j a_{j^*}(\lambda^2 - a_{j}a_{j^*})\inv a_{j}u_y = \sum_j \ov\lambda a_{j^*}(\lambda^2 - a_{j} a_{j^*})\inv u_{\sigma_{j}(y)}.
\]
This holds for all $x \in V_n$. Observing that
\[
    \bra{x} (A_n(\lift_n, \calK_\lambda))^* u = \sum_{i=0}^{d+2\q} a_i(\lambda)^* u_{\sigma_{i}(x)},
\]
and using that $a_i^* = a_{i^*}$ we conclude that $0$ is in the kernel of $(A_n(\lift_n, \calK_\lambda))^*$ (and hence in the kernel of $A_n(\lift_n,\calK_\lambda)$ since $A_n(\lift_n,\calK_\lambda)$ is a normal operator).

\paragraph*{Projection onto $\ket{+}_n^\bot$} Finally, we can check that the computations in Case 1 still hold with $v$ and $u$ projected onto $\ket{+}_n^\bot$ (this is the only case to check for finite dimensional operators). Moreover, \Cref{eqn:u-v-relation} requires that $\ketbra{+}{+} u$ is nonzero when $\ketbra{+}{+} v$ is nonzero, hence $\ketbra{+}{+}u$ is in the kernel of $A_{n,\bot}(\lift_n,\calK_\lambda)$.
\end{proof}

Next, we relate the spectral radius of the nonbacktracking operator to that of the adjacency operator. To begin, we recall some standard definitions.

Let $\calK$ be a matrix bouquet and let $A_\infty = A_\infty(\lift_\infty,\calK)$ be the $\infty$-lift of $\calK$.
The \emph{resolvent} of $A_\infty$ is defined as
\[
    G(\mu) \coloneqq (\mu I - A_\infty)\inv.
\]
$G(\mu)$ is an operator on $\ell_2(V_\infty)\otimes \C^r$ defined for all $\mu \in \C \setminus \sigma(A_\infty)$. Since $A_\infty$ is self-adjoint, we will henceforth restrict ourselves to $\mu$ being real. $G$ is self-adjoint since $A$ is self-adjoint. We can index into $G(\mu)$ in the following sense: Given $x, y \in V_\infty$, define
\[
    G_{xy}(\mu) \coloneqq \bra{x} G(\mu) \ket{y} \in \C^{r \x r}.
\]

For $\mu \notin [-\rho(A_\infty), \rho(A_\infty)]$, we define an auxiliary matrix bouquet $\wh\calK(\mu)$ based on the resolvent. Then, studying the nonbacktracking operator of $\wh\calK(\mu)$ will tell us whether $\mu$ is in $\sigma(A_\infty)$. Let
\begin{equation}\label{eqn:hat-K-def}
    \hat a_i(\mu) = G_{oo}(\mu)^{-1/2} G_{o g_i}(\mu) G_{oo}(\mu)^{-1/2} \qquad \text{for} \quad 1 \leq i \leq d+2\q,
\end{equation}
where $o \in V_\infty$ represents the origin, and $g_i$ is a generator of $\Z$ or $\Z_2$. $\wh\calK(\mu)$ is the matrix bouquet with the same index set and involution as $\calK$ but with $(\hat a_1(\mu), \hat a_2(\mu), \dots, \hat a_{d+2\q}(\mu))$ as the coefficients.
Note that if $\calK$ is self-adjoint, then so is $\wh{\calK}(\mu)$.

Before proving that $\hat{a}_i$'s are well-defined, we need a recursive characterization of $G_{oo}$. We define the operator $A_\infty^{(o)}$ to be $A_\infty$ but with the edges around the root removed, i.e.
\[
    A_\infty^{(o)} = A_\infty - \sum_{i \in \arcsz} a_i \ketbra{g_i}{o} + a_{i*} \ketbra{o}{g_i}.
\]
Let $V_\infty^{(i)}$ be the subset of $V_\infty$ where the first word in the free product of $\Z$'s and $\Z_2$'s is $g_i$, the $i$'th generator.
$\C^r\otimes \ell^2(V_\infty)$ then decomposes into the direct sum $\C^r\otimes\ket{o} \oplus_i \C^r\otimes \ell^2(V_\infty^{(i)})$, and $A_\infty^{(o)}$ acts on each part of the direct sum independently. Let $A_\infty^{(i)}$ be $A_\infty^{(o)}$ restricted to $\C^r\otimes \ell^2(V_\infty^{(i)})$. We define $G^{(o)}$ to be the resolvent of $A_\infty^{(o)}$ and $G^{(i)}$ to be the resolvent of $A_\infty^{(i)}$. Note that $A_\infty^{(i)}$ is self-adjoint, and further $\rho(A_\infty^{(i)}) \leq \rho(A_\infty)$. This is because, for every $x \in \C^r \otimes \ell^2(V_\infty^{(i)})$ with $\norm{x} = 1$, we can extend it to an $x'$ in $\C^r \otimes \ell^2(V_\infty)$ by filling the other coordinates with $0$, and $\norm{x'} =1$, so
\[
    x^T A_\infty^{(i)} x = x'^T A_\infty x' \leq \norm{A_\infty}_\opnorm = \rho(A_\infty).
\]
Note that $G^{(o)}_{g_i, g_i} = G^{(i)}_{oo}$. We will write $\gamma_i(\mu)$ to refer to either of these.

Where the expressions are well-defined, the resolvents satisfy the following recursion relations. This is almost identical to Lemma 11 from \cite{BC19}, but since our notation is slightly different, we reproduce the proof below.
\begin{lemma}\label{lem:recursion-relations}
    Let $\calK = (a_0, \dots, a_{d+2\q})$ be a matrix bouquet. Let $\mu \notin \sigma(A_\infty) \cup \sigma(A_\infty^{(o)})$, and let $G_{oo}(\mu)$ and $\gamma_i(\mu)$ be invertible for all $i \in \arcsz{}$. Then, the following hold (with the dependence of on $\mu$ suppressed in the notation):
    \begin{enumerate}
        \item $G_{og_i} = G_{oo}a_{i^*} \gamma_i = \gamma_{i^*}a_{i^*}G_{oo},$ hence $\hat{a}_i = G_{oo}^{1/2} a_{i^*} \gamma_i G_{oo}^{-1/2}$, \label{item:G1}
        \item $\displaystyle G_{oo}\inv = \mu I_r - a_0 - \sum_{i \in \arcs}G_{oo}\inv G_{og_i} a_{i}$, \label{item:G2}
        \item $\displaystyle G_{oo}^{1/2}a_{i^*} G_{oo}^{1/2} = G_{oo}^{-1/2} G_{og_i} \parens*{I_r - G_{oo}\inv G_{og_{i^*}}G_{oo}\inv G_{og_i}}\inv G_{oo}^{-1/2}$. \label{item:G3}
    \end{enumerate}
\end{lemma}
\begin{proof}
In this proof we also suppress the dependence on $\mu$ for the resolvents. Applying the resolvent identity on $G$ and $G^{(o)}$, we have
\begin{equation}\label{eqn:resolvent-identity}
    G = G^{(o)} + G(A_\infty - A_\infty^{(o)})G^{(o)} = G^{(o)} + G^{(o)}(A_\infty - A_\infty^{(o)})G.
\end{equation}
We note that for $x, y \in V_\infty$ in different connected components of $A_\infty^{(o)}$, we have $G_{xy}^{(o)} = 0$. Then, composing the first equality in \Cref{eqn:resolvent-identity} with $\bra{o}\cdot\ket{g_i}$ gives us $G_{og_i} = G_{oo}a_{i^*}\gamma_i$, and composing the second expression with $\bra{g_i} \cdot \ket{o}$ gives $G_{g_i o} = \gamma_i a_{i}G_{oo}$ which implies \Cref{item:G1} by symmetry.

Next, composing \Cref{eqn:resolvent-identity} with $\bra{o}\cdot\ket{o}$ gives
\[
    G_{oo} = G_{oo}^{(o)} + \sum_{j \in \arcsz{}} G_{og_i} a_j G_{oo}^{(o)}.
\]
We then use that $G_{oo}^(o) = (\mu I_r - a_0)\inv$ (since $o$ is an isolated vertex in $A^{(o)}$), which is well defined since $\norm{a_0} \leq \norm{A_\infty}$. Rearranging the previous equation, we get \Cref{item:G2}
\[
    G_{oo}\parens*{\mu I_r - a_0 - \sum_{j \in \arcsz{}} a_{j^*} \gamma_j a_{j}} = I_r.
\]

Finally, we fix $i \in \arcsz$ and remove the edges around both $o$ and $g_i \in V_\infty$ by defining
\[
    A_\infty^{(og_i)} = A_\infty^{(o)} - \sum_{j \ne i^*}a_j\ketbra{g_ig_j}{g_i} - a_{j^*} \ketbra{g_i}{g_ig_j}.
\]
Using the resolvent identity similarly as before, and using that $G_{g_ig_j, g_ig_j}^{(og_i)} = \gamma_j$ by translation invariance, we obtain
\[
    \gamma_i\inv = \mu I_r - a_o - \sum_{j \ne i^*}a_{j^*} \gamma_j a_{j^*}.
\]
Using this, we find that
\[
    \gamma_i \parens*{I_r - a_i \gamma_{i^*}a_{i^*}\gamma_i}\inv = (\gamma_i\inv - a_i \gamma_{i^*}a_{i^*})\inv = \parens*{\mu I_r - a_0 - \sum_j a_{j^*}\gamma_j a_j}\inv = G_{oo}.
\]
Multiplying on the left by $G^{1/2}a_{i^*}$, substituting in $a_{i^*}\gamma_i = G_{oo}^{-1}G_{og_i}$ and then multiplying on the right by $G^{-1/2}$ gives \Cref{item:G3}.
\end{proof}

Next, we show that outside $[-\rho(A_\infty),\rho(A_\infty)]$, $G_{oo}(\mu)$ and $G_{oo}^{(i)}(\mu)$ are well-defined and bounded, hence $\hat{a}_i$ is well-defined.
\begin{lemma}\label{lem:hat-a-well-defined}
    Let $\eps > 0$, and let $\calK = (a_0, \dots, a_{d+2\q})$ be a matrix bouquet with coefficients satisfying $\norm{a_i}_\opnorm \leq 1$ for $i \in \arcsz$, and let $\mu \in \R$ with $|\mu| > \rho(A_\infty) + \eps$.
    Then, $\hat{a}_i(\mu)$ is well-defined and satisfies $\norm{\hat{a}_i(\mu)}_\opnorm \leq \eps^{-3/2}(|\mu| + (d+2\q+1))$.
    Moreover, if there is $\delta > 0$ such that $\norms{a_i\inv} \leq 1/\delta$ for $i \in \arcs$, then we can also conclude that $\hat{a}_i(\mu)$ is non-singular and satisfies $\norms{\hat{a}_i(\mu)\inv}_\opnorm \leq \eps \delta (|\mu| + (d+2\q+1))^{-3/2}$.
\end{lemma}
\begin{proof}
In this proof we will need several facts about spectral measures, which can be found for instance in \cite[Ch. X]{DS88b}.

Given $A_\infty$, there exists a unique positive operator-valued measure $E$ (the \emph{resolution to the identity}) supported on $\sigma(A_\infty)$ such that
\[
    \int_{\sigma(A_\infty)} \,dE = I, \quad
    \int_{\sigma(A_\infty)} \lambda \, dE(\lambda) = A_\infty
    \quad \text{and} \quad
    \int_{\sigma(A_\infty)} \frac{1}{\mu-\lambda} \, dE(\lambda) = G(\mu).
\]
We define $\nu_o$ to be the positive operator-valued measure $\angle{\delta_o, E(\cdot)\delta_o}$. In particular,
\[
    \int_{\sigma(A_\infty)} \frac{1}{\mu - \lambda} \,d\nu_o(\lambda) = G_{oo}(\mu).
\]
We will use the fact that since $\nu_o$ is positive-valued, for any function $f \geq 0$ pointwise,
\[
    \int_{\sigma(A_\infty)} f \,d\nu_o \geq 0 \quad \text{in Loewner order.}
\]
One way to see this is, for any vector $x \in \C^r$, $\angle{x,\nu_o(\cdot)x}$ defines a nonnegative real-valued measure which is equal to the usual spectral measure with respect to the vector $x \otimes \ket{o} \in \C^r \otimes \ell^2(V_\infty)$, therefore
\[
    \angle*{x,\parens*{\int_{\sigma(A_\infty)} f \,d\nu_o}x} = \int_{\sigma(A_\infty)} f\,d\angle{x,\nu_o(\cdot) x} \geq 0.
\]
We can bound $\rho(A_\infty) = \norm{A_\infty}_\opnorm \leq \sum_{i \in \arcsz} \norm{a_i}_\opnorm \leq d+2\q+1$.
For a real $\mu > \rho(A_\infty)$, for $\lambda \in \sigma(A_\infty)$ we can lower bound
\[
    \frac{1}{\mu + (d+2\q+1)} \leq \frac{1}{\mu - \lambda}
\]
and therefore
\[
    G_{oo}(\mu) = \int_{\sigma(A_\infty)} \frac{1}{\mu - \lambda} \,d\nu_o \geq \int_{\sigma(A_\infty)} \frac{1}{\mu+(d+2\q+1)} \,d\nu_o = (\mu + (d+2\q+1))\inv I.
\]
Similarly, for $\mu < -\rho(A_\infty)$,
\[
    |G_{oo}(\mu)| = \int_{\sigma(A_\infty)} \frac{1}{\lambda - \mu} \,d\nu_o \geq \int_{\sigma(A_\infty)} \frac{1}{|\mu|+(d+2\q+1)} \,d\nu_o = (|\mu| + (d+2\q+1))\inv I.
\]
We can upper bound
\[
    \norm{G_{oo}(\mu)}_\opnorm \leq \norm{G(\mu)}_\opnorm \leq \frac{1}{\dist(\mu, \sigma(A_\infty))} \leq \eps\inv.
\]
Therefore, we have shown that
\[
    (|\mu| + (d+2\q+1))\inv I_r \leq \abs{G_{oo}(\mu)} \leq \eps\inv I_r
\]
which in particular shows that $G_{oo}(\mu)^{-1/2}$ is well-defined. Since $\rho(A^{(i)}_\infty) \leq \rho(A_\infty)$, we get an identical bound for each $\gamma_i(\mu)$. Now, combining with $\norm{a_i}_\opnorm \leq 1$ and $\norms{a_i\inv}\leq 1/\delta$ and using \Cref{lem:recursion-relations}, \Cref{item:G1} we get
\[
    \norms{\hat{a}_i(\mu)}_\opnorm \leq \eps^{-3/2}(|\mu| + (d+2\q+1))
    \quad \text{and} \quad
    \norms{\hat{a}_i(\mu)\inv}_\opnorm \leq \eps \delta (|\mu| + (d+2\q+1))^{-3/2}.
    \qedhere
\]
\end{proof}

\begin{proposition}\label{prop:bc-prop-10}
Let $\calK$ be a matrix bouquet with coefficients satisfying $\norm{a_i}_\opnorm < 1$ for all $i \in \arcsz{}$, and let $A = A_\infty(\lift_\infty,\calK)$.
Let $\mu \in \R$ with $|\mu| > \rho(A)$.
Let $\wh\calK(\mu)$ be defined as in \Cref{eqn:hat-K-def}.
Let $\wh B_\mu = B_\infty(\lift_\infty, \wh\calK(\mu))$ be the corresponding nonbacktracking operator of the $\infty$-lift of $\wh\calK(\mu)$.
Then, $1 \notin \sigma(\wh B_\mu)$.

Moreover, the above also holds for any $n$-lift $\lift_n$ with $A = A_{n,\bot}(\lift_n,\calK)$ and $\wh B_\mu = B_{n,\bot}(\lift_n, \wh\calK(\mu))$.
\end{proposition}
\begin{proof}
We will apply \Cref{prop:9} with $\lambda = 1$. Let $\calK' = (a_0',\dots, a_{d+2\q}')$ be the matrix bouquet obtained from applying \Cref{prop:9} to $\wh \calK(\mu)$, and let $A' = A_\infty(\lift_\infty, \calK')$ be the corresponding lift. We know that $0 \in A'$ if and only if $1 \in \sigma(\wh B_\mu)$, so showing that $A' = (G_{oo}(\mu)^{1/2} \otimes I)(A - \mu I) (G_{oo}(\mu)^{1/2})$ shows that if $|\mu| > \rho(\wh A_\mu)$ then $1 \notin \sigma(\wh B_\mu)$. Note that $G_{oo}(\mu)^{1/2}$ is well-defined by \Cref{lem:hat-a-well-defined}.

Let $i \in \arcs{}$. We check that (omitting the dependence on $\mu$),
\begin{align*}
    a_i' &= \hat{a}_i\parens*{1-\hat{a}_{i^*}\hat{a}_i}\inv
    = G_{oo}^{-1/2}G_{og_i} G_{oo}^{-1/2}\parens*{1 - G_{oo}^{-1/2} G_{og_{i^*}}G_{oo}^{-1} G_{og_i}G_{oo}^{-1/2}}\inv\\
    &= G_{oo}^{-1/2}G_{og_i} \parens*{G_{oo}^{1/2} - G_{oo}^{-1/2} G_{og_{i^*}}G_{oo}^{-1} G_{og_i}}\inv
    = G_{oo}^{-1/2}G_{og_i} \parens*{1 - G_{oo}^{-1} G_{og_{i^*}}G_{oo}^{-1} G_{og_i}}\inv G_{oo}^{-1/2},
\end{align*}
which is equal to $G_{oo}(\mu)^{1/2} a_{i^*} G_{oo}(\mu)^{1/2}$ by \Cref{lem:recursion-relations}, \Cref{item:G3}.

For $a_0'$, we check that
\begin{align*}
    a_0' &= -1 - \sum_{i \in \arcs} \hat{a}_i\parens*{1 - \hat{a}_{i^*} \hat{a}_i}\inv\hat{a}_{i^*}\\
    &= -1 - \sum_{i \in \arcs} G_{oo}^{-1/2}G_{og_i} G_{oo}^{-1/2}\parens*{1 - G_{oo}^{-1/2} G_{og_{i^*}}G_{oo}^{-1} G_{og_i}G_{oo}^{-1/2}}\inv G_{oo}^{-1/2}G_{og_{i^*}} G_{oo}^{-1/2}\\
    &= -1 - \sum_{i \in \arcs} G_{oo}^{-1/2}G_{og_i} \parens*{G_{oo}^{1/2} - G_{oo}^{-1/2} G_{og_{i^*}}G_{oo}^{-1} G_{og_i}}\inv G_{oo}^{-1/2}G_{og_{i^*}} G_{oo}^{-1/2}\\
    &= -1 - \sum_{i \in \arcs} G_{oo}^{-1/2}G_{og_i} \parens*{1 - G_{oo}^{-1} G_{og_{i^*}}G_{oo}^{-1} G_{og_i}}\inv G_{oo}^{-1}G_{og_{i^*}} G_{oo}^{-1/2}.
\end{align*}
Now applying \Cref{lem:recursion-relations}, \Cref{item:G3}, we get
\[
    a_0' = -1 - \sum_{i \in \arcs} G_{oo}^{1/2}a_{i^*} G_{og_{i^*}} G_{oo}^{-1/2}.
\]
Now using $G_{og_{i^*}} = \gamma_i a_i G_{oo}$ from \Cref{lem:recursion-relations}, \Cref{item:G1}, we have
\begin{align*}
    a_0' &= -1 - \sum_{i \in \arcs} G_{oo}^{1/2}a_{i^*} \gamma_i a_i G_{oo}^{1/2}
    = -1 - \sum_{i \in \arcs} G_{oo}^{-1/2}G_{og_i} a_i G_{oo}^{1/2}\\
    &= G_{oo}^{1/2}\parens*{-G_{oo}^{-1} - \sum_{i \in \arcs} G_{oo}^{-1}G_{og_i} a_i}G_{oo}^{1/2},
\end{align*}
and by applying \Cref{lem:recursion-relations}, \Cref{item:G2}, we get
\[
    a_0' = G_{oo}^{1/2}(a_0 - \mu I_r) G_{oo}^{1/2}.
\]
Now, notice that if we swap each $a_i$ with $a_{i^*}$, all lifts of the resulting matrix bouquet remain the same. Therefore we conclude that $A' = G_{oo}^{1/2}(A - \mu I) G_{oo}^{1/2}$.

Repeating the argument with $A = A_{n,\bot}(\lift_n,\calK)$, $\wh A_\mu = A_{n,\bot}(\lift_n, \wh\calK(\mu))$ and $\wh B_\mu = B_{n,\bot}(\lift_n, \wh\calK(\mu))$ gives the last part of the proposition.
\end{proof}

Finally, we can prove the following, which is essentially Theorem 12 of~\cite{BC19}.
\begin{theorem}\label{thm:bc-thm-12}
Let $\calK = (a_0,\dots, a_{d+2\q})$ be a matrix bouquet with $\max_{i\in\arcsz}\norms{a_i}_\opnorm < 1$ and let $\lift_n$ be an $n$-lift.
For any $\eps > 0$, there exists $\delta > 0$ depending only on $\q, d, \eps$ and not on the particular $a_i$'s such that if for all $\mu \in \R$ with $|\mu| > \norms{A_\infty(\lift_\infty, \calK)}_\opnorm + \eps$ it holds that
\begin{equation}\label[ineq]{eqn:B-mu-bound}
    \rho(B_{n,\bot}(\lift_n, \wh\calK(\mu))) \leq \rho(B_{\infty}(\lift_\infty, \wh\calK(\mu))) + \delta,
\end{equation}
where $\wh\calK(\mu)$ is defined via \Cref{eqn:hat-K-def},
then $\norms{A_{n,\bot}(\lift_n,\calK)}_\opnorm \leq \norms{A_\infty(\lift_\infty, \calK)}_\opnorm + \eps$.
\end{theorem}
\begin{proof}
We write $A_\infty = A_\infty(\lift_\infty,\calK)$, $A_{n,\bot} = A_{n,\bot}(\lift_n, \calK)$, $\wh B_\mu = B_\infty(\lift_\infty, \wh\calK(\mu))$, and $\wh B_{n,\mu} = B_{n,\bot}(\lift_n, \wh\calK(\mu))$.

Fix $\eps > 0$.
First we show that we have $\rho(\wh B_\mu) < 1$ for all $\mu \in \R$ with $|\mu| > \norms{A_\infty}_\opnorm + \eps$.
To do so, suppose for the sake of contradiction that $\rho(\wh B_\mu) \geq 1$ for some $|\mu| > \norms{A_\infty}_\opnorm + \eps$.
As we take $\mu \to \infty$ if $\mu > 0$ or $\mu \to -\infty$ if $\mu < 0$, \Cref{lem:hat-a-well-defined} implies that $\norms{\hat{a}_i(\mu)}_\opnorm \lesssim_{\q,d,\eps} |\mu|^{-1/2}$, therefore $\rho(\wh B_\mu) \leq \norms{\wh B_\mu}_\opnorm \lesssim_{\q,d,\eps} |\mu|^{-1/2}$ as well, in particular $\rho(\wh B_\mu) \to 0$ as $|\mu| \to \infty$. Using the continuity of $\rho(\wh B_\mu)$ as a function of $\mu$ (due to the $\hat{a}_i(\mu)$'s being continuous as a function of $\mu$ and \Cref{prop:sqrtL}), there must exist some $\mu'$ such that $\rho(\wh B_{\mu'})_\opnorm = 1$. Lemma 13 in~\cite{BC19} says that in this case, $1 \in \sigma(\wh B_{\mu'})$. Since we assumed $|\mu| > \norms{A_\infty}_\opnorm + \eps$, we know $|\mu'| > \norms{A_\infty}_\opnorm + \eps$ as well, but this contradicts \Cref{prop:bc-prop-10}, so we conclude that $\rho(\wh B_\mu) < 1$.

Again using that $\rho(\wh B_\mu) \lesssim |\mu|^{-1/2}$ as $|\mu| \to \infty$, there exists some $D > \norms{A_\infty}_\opnorm + \eps$ depending on $\q, d,\eps$ such that $\rho(\wh B_\mu) < 1/2$ for $|\mu| > D$. On $[-D,-\norms{A_\infty}_\opnorm - \eps] \cup [\norms{A_\infty}_\opnorm + \eps, D]$, $\rho(\wh B_\mu)$ is uniformly continuous as a function of $\mu$, hence there exists $\delta > 0$ such that $\rho(\wh B_\mu) < 1-\delta$ for all $|\mu| > \norms{A_\infty}_\opnorm + \eps$. Hence, if we have \Cref{eqn:B-mu-bound} for all $|\mu| > \norms{A_\infty}_\opnorm + \eps$ then $\rho(\wh{B}_{n,\mu}) < 1$ for those $\mu$, which by \Cref{prop:bc-prop-10} implies $|\mu| > \rho(A_{n,\bot})$.
\end{proof}

\section{Explicit lifts of matrix polynomials}\label{sec:final}

In this section we combine \Cref{sec:ihara-bass} with \Cref{thm:probabilistic-explicit-main} to get closeness of the spectra of $A_n(\lift_n, \ourpoly)$ and $A_\infty(\ourpoly)$ for a given $\ourpoly$ and $\lift_n$ from \Cref{thm:probabilistic-explicit-main}. The proof follows two main steps: first we will use \Cref{sec:ihara-bass} to bound the norm of $A_n(\lift_n, \calK)$ for all linear polynomials $\calK$, and then using some operator theoretic tools in the literature, we can upgrade this bound for linear polynomials into Hausdorff-closeness in spectra for all polynomials.

\subsection{Constructing lifts for linear polynomials}
In this section our goal is to prove, roughly, that given $\eps > 0$, for a linear polynomial $\calK$ with bounded coefficients we can construct an $N' \sim N$ lift in $\poly(N)$ time such that $\norm{A_\infty(\calK)}_\opnorm - \eps \leq \norm{A_{N',\bot}(\lift_{N'},\calK)}_\opnorm \leq \norm{A_\infty(\calK)}_\opnorm + \eps$.

To get the lower bound, we will require the following deterministic condition by Bordenave--Collins that ensures $\sigma(A_\infty)$ is in an $\eps$-neighborhood of $\sigma(A_{n,\bot})$.
\begin{proposition}[{\cite[Prop.~7]{BC19}}]\label{prop:bc-prop-7}
    Let $\eps, R > 0$. There exists $h \in \N^+$ such that for every $n \in \N^+$, if an $n$-lift $\lift_n$ has at least one vertex whose $h$-neighborhood is acyclic, then every matrix bouquet $\calK$ with $\max_i \norm{a_i}_\frob \leq R$ satisfies
    \[
        \sigma(A_\infty(\calK)) \subseteq \sigma(A_{n,\bot}(\lift_n, \calK)) + [-\eps, \eps].
    \]
\end{proposition}
\begin{remark}
    The proposition within Bordenave--Collins is stated for a particular matrix bouquet $\calK$,
    but it suffices to prove the proposition for an $\eps$-net of matrix bouquets, and use the stability of the spectra (\Cref{thm:perturbation}).
\end{remark}

The following theorem about the stability of the spectra of self-adjoint bounded linear operators under symmetric perturbations can be found in Kato, Ch.~5, Thm.~4.10~\cite{Kat95}:
\begin{theorem}\label{thm:perturbation}
    Let $T$ be a self-adjoint bounded linear operator on a Hilbert space $\scrH$ and let $A:\scrH \to \scrH$ be a symmetric bounded operator.
    Let $S = T + A$. Then $S$ is self-adjoint and $\dist_H(\sigma(S),\sigma(T)) \leq \norm{A}_\opnorm$.
\end{theorem}

We also need a version of \Cref{prop:bc-prop-7} for random signed lifts.
\begin{proposition}\label{prop:bc-prop-7-signed}
    Let $\eps, R > 0$. There exists $h \in \N^+$ such that for every $n \in \N^+$, if an $n$-lift $\lift_n$ has at least one vertex whose $h$-neighborhood is acyclic, and $\calK$ is a matrix bouquet with $\max_i \norm{a_i}_\frob \leq R$, then any signing $\chi$ satisfies
    \[
        \sigma(A_\infty(\calK)) \subseteq \sigma(A_{n}(\bchi\lift_n, \calK)) + [-\eps, \eps].
    \]
\end{proposition}
\begin{remark}
    The proof of this is essentially the same as \Cref{prop:bc-prop-7}, by noting that for $x \in V_n$ containing no cycle in its $h$-neighborhood, for $0 \leq k \leq 2h$,
    \[
         \int_{\R} \lambda^k \,d\mu_{A_n}^x = \frac1r \sum_\gamma \tr\prod_{t=1}^k \chi_{i_t} a_{i_t} = \frac1r \sum_\gamma \tr\prod_{t=1}^k a_{i_t} = \int_\R \lambda^k \,d\mu_{A_\infty}^o,
    \]
    where $\gamma$ is a closed length-$k$ walk starting and ending at $x$.
    Or in words, the sum of the product of the weights on closed length-$k$ walks at $x$ is equal to that of the infinite graph $A_\infty$.
\end{remark}

\begin{theorem}\label{thm:linear-polys-norm}
    Fix an index set $\arcsz = \{0, 1, \dots, d+2\q\}$ and involution~$*$ for a matrix bouquet. Fix also constants $r \in \N^+$ and $R, \eps > 0$.  Then, there is a $\poly(N)$-time deterministic algorithm that, on input~$N$, outputs an $N'$-lift $\lift_{N'}$ (with $N \leq N' \leq N + o(N)$) such that
    for every matrix bouquet $\calK = (a_0,\dots,a_{d+2\q})$ satisfying $\max_{i \in \arcsz{}} \norm{a_i}_\frob \leq R$ with $r'$-dimensional matrix weights ($r' \leq r$), we have that $\norm{A_\infty(\calK)}_\opnorm - \eps \leq \norm{A_{N',\bot}(\lift_{N'},\calK)}_\opnorm \leq \norm{A_\infty(\calK)}_\opnorm + \eps$.

    Moreover, for $h \in \N^+$ given by applying \Cref{prop:bc-prop-7}, $\calL_{N'}$ also has at least one vertex whose $h$-neighborhood is acyclic.
\end{theorem}
\begin{proof}
By homogeneity, we first reduce to the case where $\max_{i\in\arcsz} \norm{a_i}_\frob \leq \max_i\norm{a_i}_\opnorm < 1$ by scaling and replacing $\eps$ with $\eps/2R$.

The lift $\lift_{N'}$ we construct will satisfy the hypothesis of \Cref{prop:bc-prop-7}, hence guaranteeing that $\norm{A_\infty(\calK)}_\opnorm \leq \norm{A_{N',\bot}(\lift_{N'},\calK)}_\opnorm + \eps$.

Next, we show that
\[
    \norm{A_{N',\bot}(\lift_{N'},\calK)}_\opnorm \leq \norm{A_\infty(\calK)}_\opnorm + \eps.
\]
Given a matrix bouquet $\calK = (a_0, \dots, a_{d+2\q})$ with $\max_i \norm{a_i}_\opnorm < 1$, we construct another matrix bouquet $\calK' = (a_0', \dots, a_{d+2\q}')$ where the coefficients have bounded inverse. For $n \in \N_+$, denote $A_{n,\bot} \coloneqq A_{n,\bot}(\lift_n,\calK)$ and $A_{n,\bot}'\coloneqq A_{n,\bot}'(\lift_n,\calK')$. We will construct $\calK'$ such that for any $n$-lift $\lift_n$, $\norm{A_n - A_n'}_\opnorm \leq \eps/3$.
Then, if we show that $\norm{A_{n,\bot}'}_\opnorm \leq \norm{A_\infty'}_\opnorm + \eps/3$, it implies
\[
    \norm{A_{n,\bot}}_\opnorm \leq \norm{A_{n,\bot}'}_\opnorm + \norm{A_n-A_n'}_\opnorm \leq \norm{A_\infty'}_\opnorm + 2\eps/3 \leq \norm{A_\infty}_\opnorm + \eps.
\]

Note it suffices to consider $a_0,\dots, a_{d+2\q}$ such that each $a_i$ is self-adjoint (by replacing $a_i$ with $(a_i + a_{i^*})/2$, this results in the same lift $A_n(\lift_n,\calK)$). For each $a_i$, $i \in \arcs{}$, we can diagonalize $a_i = U_i D_i U_i^*$. Let $R' = 3\eps\inv(d+2\q)$. Let $D'$ be $D$ but replace each diagonal entry in $[0,1/R')$ with $1/R'$ and replace each diagonal entry in $(-1/R', 0)$ with $-1/R'$, and define $a_i' \coloneqq U_i D_i' U_i^*$. Then, $a_i'$ satisfies that $\norm{a_i'}_\opnorm < 1$, $\norms{{a_i'}\inv}_\opnorm \leq R'$ and $\norm{a_i' - a_i}_\opnorm \leq 1/R'$. Furthermore, $\norm{A_n - A_n'}_\opnorm \leq \sum_{i\in\arcsz{}} \norm{a_i - a_i'}\leq \eps/3$.

Let $\mu \in \R$ with $|\mu| > \norm{A'_\infty}_\opnorm + \eps/3$. If $|\mu| > (d+2\q+1)$, then certainly $|\mu| > \norm{A'_n}_\opnorm$. Otherwise, we have $|\mu| \leq (d+2\q+1)$ and we construct $\wh \calK(\mu)$ from $\calK'$ as in \Cref{eqn:hat-K-def}. By \Cref{lem:hat-a-well-defined}, we know for each $\hat a_i \in \arcs$,
\[
    \norms{\hat{a}_i}_\opnorm \leq 2(\eps/3)^{-3/2}(d + 2\q + 1),
\]
and
\[
    \norms{\hat{a}_i\inv}_\opnorm \leq 2(\eps/3) {R'}\inv (d+2\q+1)^{-3/2}.
\]

Let $\delta$ be from \Cref{thm:bc-thm-12} with $\eps/3$ in place of $\eps$. Let $h \in \N^+$ be from applying \Cref{prop:bc-prop-7}, with $R = \max\{2(\eps/3)^{-3/2}(d + 2\q + 1), 2(\eps/3) {R'}\inv (d+2\q+1)^{-3/2}\}$, and $\eps$. We then apply \Cref{thm:weakly-explicit-main} with the same $R$ and $\eps = \delta$ and $C' = 1$, taking $N'$ to be large enough that $\sqrt{\log N'} \geq h$,
to obtain an $N'$-lift $\lift_{N'}$ that simultaneously for all $\mu \in \R$ with $|\mu| > \norm{A'_\infty}_\opnorm + \eps/3$, $\rho(B_{N',\bot}(\lift_{N'},\wh \calK(\mu)) \leq \rho(B_\infty(\wh\calK(\mu)) + \delta$. By applying \Cref{thm:bc-thm-12} we conclude that $\norms{A_{N',\bot}'}_\opnorm \leq \norms{A_\infty'}_\opnorm + \eps/3$.

Finally, we note that our application of \Cref{thm:weakly-explicit-main} depended only on $R,\eps,\q,d,r$ and not the coefficients themselves, hence the lift $\lift_{N'}$ works for every matrix bouquet satisfying the given bounds.
\end{proof}

\subsection{A linearization trick, and lifts of matrix polynomials}
In order to derive norm convergence for polynomials from norm convergence for linear polynomials, we use the following \emph{linearization} result, which appears in~\cite{Pis18} as Corollary 11 and the following remarks (and first appeared implicitly in~\cite{Pis96}).
We will first state these results in terms of their limiting behavior, and derive quantitative bounds later. For complex-valued random variables $(X_n)_{n \in \N}$ and $x$, we say $X_n$ converges in probability to $x$ if for every $\eps > 0$,
\[
    \lim_{n \to \infty}\Pr[|X_n - x| \geq \eps] = 0.
\]
\begin{proposition}\label{prop:norm-linearization}
    Let $I$ be a finite index set, and let $(x_j)_{j \in I}$ be free Haar unitaries in a unital $C^*$-algebra $A$. Let $(X_j^{(N)})_{j \in I}$ be a system of random unitary matrices with a common dimension, for each $N$. If for every $r \in \N_+$ and any set of matrix coefficients $(a_j)_{j \in I}$ and $a_j \in \C^{r \x r}$ we have
    \[
        \norm{a_0 \otimes \bone + \sum_{j \in I} a_j \otimes X_j^{(N)} + a_j^* \otimes (X_j^{(N)})^*} \xrightarrow{N \to \infty} \norm{a_0 \otimes \bone + \sum_{j \in I} a_j \otimes x_j + a_j^* \otimes (x_j)^*} \quad \text{in probability,}
    \]
    Then, for all matrix-coefficient polynomials $P$ in $|I|$ variables and their adjoints, we have
    \[
        \norm{P((X_j^{(N)})_{j \in I}, (X_j^{(N)*})_{j \in I})}
        \xrightarrow{N \to \infty}
        \norm{P((x_j)_{j \in I}, (x_j^*)_{j \in I})}
        \quad \text{in probability.}
    \]
\end{proposition}

Finally, the convergence in norm for all polynomials then implies convergence in Hausdorff distance.
The following appears in Proposition 2.1 in~\cite{CM14}\footnote{The proposition there is stated for convergence almost surely, but the same proof works for convergence in probability. Only the first inclusion in \Cref{item:strong-convergence-equivalence-2} is shown, but the other direction is also easy to show.}.
\begin{proposition}\label{prop:strong-convergence-hausdorff}
    Let $I$ be a finite index set, and let $(x_j)_{j \in I}$ and $(X_j^{(N)})_{j \in I}$ for $N \in \N_+$ be variables in $C^*$-probability spaces with faithful states. Then, the following are equivalent:
    \begin{enumerate}
        \item For every matrix-coefficient polynomial $P$ in $|I|$ variables and their adjoints,
        \[
            \norm{P((X_j^{(N)})_{j \in I}, (X_j^{(N)*})_{j \in I})}
            \xrightarrow{N \to \infty}
            \norm{P((x_j)_{j \in I}, (x_j^*)_{j \in I})}
            \quad \text{in probability.}
        \]
        \item \label{item:strong-convergence-equivalence-2}
        For every matrix-coefficient polynomial $P$ in $|I|$ variables and their adjoints,
        let $Y_N$ be the random variable $P((X_j^{(N)})_{j \in I}, (X_j^{(N)*})_{j \in I})$, and let $Y$ be $P((x_j)_{j \in I}, (x_j^*)_{j \in I})$.
        Then, the spectrum of $Y_N$ converges in probability to the spectrum of $Y$ in Hausdorff distance, that is to say, for every $\eps > 0$, with probability going to 1 as $N$ goes to infinity,
        \[
            \sigma(Y_N) \subseteq \sigma(Y) + [-\eps,\eps]
        \]
        and
        \[
            \sigma(Y) \subseteq \sigma(Y_N) + [-\eps, \eps].
        \]
    \end{enumerate}
\end{proposition}

Using \Cref{thm:BC-main} followed by \Cref{thm:main} instead of \Cref{thm:weakly-explicit-main} in the proof of \Cref{thm:linear-polys-norm},
and also using \Cref{prop:bc-prop-7-signed} instead of \Cref{prop:bc-prop-7},
and finally applying \Cref{prop:norm-linearization,prop:strong-convergence-hausdorff}, we can deduce the following probabilistic statement.
\begin{theorem}\label{thm:lift-seq-probabilistic}
    Let $\ourpoly$ be a matrix polynomial. Let $\blift_n$ be a uniformly random $n$-lift, and let $\btlift_2^{(1)},\dots,\btlift_2^{(k)}$ be $k$ independent uniformly random $2$-lifts.
    Then, $\sigma(A_{2^kn,\bot}(\btlift_2^{(1)} \otimes \dots \otimes \btlift_2^{(k)}\otimes\blift_n, \ourpoly{}))$ converges in probability in Hausdorff distance to $\sigma(A_\infty(\ourpoly{}))$ an $n \to \infty$.
\end{theorem}

As a consequence, we get the following theorem, which says that the spectrum of a uniformly random signed lift converges to that of the $\infty$-lift, without the need to project onto the $\ket{+}_n^\bot$ space.
\begin{theorem}\label{thm:signed-lift-probabilistic}
    Let $\ourpoly$ be a matrix polynomial, and let $\bchi\blift_n$ be a uniformly random signed $n$-lift (i.e.\ by taking both $\bchi$ and $\blift_n$ uniformly random).
    Then, $\sigma(A_{n}(\bchi\blift_n, \ourpoly{}))$ converges in probability in Hausdorff distance to $\sigma(A_\infty(\ourpoly{}))$.
\end{theorem}
\begin{proof}
We take $k = 1$ in \Cref{thm:lift-seq-probabilistic}. \Cref{prop:A-spec-identities}, \Cref{item:A-spec-perp} tells us that
\[
    \sigma(A_{2n,\bot}(\btlift_2 \otimes \blift_n, \ourpoly)) = \sigma(A_{n,\bot}(\blift_n,\ourpoly)) \cup \sigma(A_{n,\bchi}(\blift_n,\ourpoly)).
\]
We know from Bordenave--Collins' theorem \Cref{thm:bc-main-intro} that $\sigma(A_{n,\bot}(\blift_n, \ourpoly))$ converges in probability in Hausdorff distance to $\sigma(A_\infty(\ourpoly))$. The above equation tells us that $\sigma(A_{n,\bchi}(\blift_n,\ourpoly)) \subseteq \sigma(A_{2n,\bot}(\btlift_2 \otimes \blift_n,\ourpoly))$, which converges to $\sigma(A_\infty(\ourpoly))$ by \Cref{thm:lift-seq-probabilistic}.
\end{proof}

Next, we find the derandomized analogue of \Cref{thm:lift-seq-probabilistic}. To do so, we will need to examine the proof of \Cref{prop:norm-linearization}, which uses the following factorization. We state this only specialized to our situation of permutations/matchings converging to generators of $\Z$ and $\Z_2$, instead of the more general form appearing in~\cite[Cor. 4, Cor. 7]{Pis18}.

\begin{proposition}\label{prop:factorization}
    Let $\lift_n$ be an $n$-lift (with $n$ possibly being $\infty$). Let $\ourpoly$ be a matrix polynomial satisfying $\norm{A_n(\lift_n, \ourpoly)}_\opnorm \leq 1$. Then, there exists $m \in \N_+$ such that $\ourpoly$ can be factorized into
    \[
        \ourpoly = \alpha_0 D_1 \alpha_1 \dots D_m \alpha_m,
    \]
    where each $\alpha_i$ is a rectangular complex-valued matrix satisfying $\norm{\alpha_i}_\opnorm \leq 1$, and each $D_i$ is a self-adjoint linear matrix-coefficient polynomial (possibly of different dimensions) satisfying $\norm{A_n(\lift_n,D_i)}_\opnorm \leq 1$.
\end{proposition}

Using this factorization, we can upgrade the convergence in norm for linear polynomials in \Cref{thm:linear-polys-norm} to convergence in the entire spectrum in Hausdorff distance, not just for linear polynomials but for all self-adjoint polynomials.

We will need to use the \emph{spectral mapping theorem} (see e.g.~\cite[Thm VII.3.11]{DS88a}).
\begin{theorem}\label{thm:spectral-mapping}
    Let $T:\scrH \to \scrH$ be a bounded linear operator on a Hilbert space $\scrH$. Let $f$ be a holomorphic function defined on a neighborhood of $\sigma(T)$. Then, $f(\sigma(T)) = \sigma(f(T))$.
\end{theorem}
We will only be using this in the case where $f$ is a polynomial. In this case $f(A_n(\lift_n, \ourpoly)) = A_n(\lift_n, f(\ourpoly))$, where we interpret $f(\ourpoly)$ by way of addition and multiplication in the polynomial ring.

\begin{theorem}\label{thm:polynomials-hausdorff}
    Fix an index set $\arcsz = \{0, 1, \dots, d+2\q\}$ and involution~$*$ for a matrix bouquet. Fix also constants $r, k \in \N^+$ and $R, \eps > 0$.
    Then, there is a $\poly(N)$-time deterministic algorithm that, on input~$N$, outputs an $N'$-lift $\lift_{N'}$ (with $N \leq N' \leq N + o(N)$) such that
    simultaneously for every self-adjoint matrix polynomial $\ourpoly$ with total degree at most $k$, coefficients in $\C^{r' \x r'}$ with $r' \leq r$ and the Frobenius norm of the coefficients of $\ourpoly$ are at most $R$, we have that $\sigma(A_{n,\bot}(\lift_{N'},\ourpoly))$ and $\sigma(A_\infty(\ourpoly))$ are $\eps$-close in Hausdorff distance.
\end{theorem}
\begin{proof}
    First, we scale $\ourpoly$ by $1/2R$ so that $\max_i \norm{a_i}_F < 1$, and replace $\eps$ with $\eps/2R$.

    The lift $\lift_{N'}$ we construct will satisfy the hypothesis of \Cref{prop:bc-prop-7}, hence $\sigma(A_\infty(\calK)) \subseteq\sigma(A_{n,\bot}(\lift_n, \calK)) + [-\eps, \eps]$.

    To show the other inclusion, we proceed by an $\eps$-net argument over all polynomials with total degree at most $k$.
    Let $\delta = \eps/(3(k+1))$.
    Let $\mathfrak{F}_{r'}$ be the set of matrices in $\C^{r' \x r'}$ with Frobenius norm less than 1, and let $\mathfrak{G}_{r',\delta}$ be a net of matrices in $\mathfrak F_{r'}$ (e.g. by taking the entries to be complex integer multiples of $\delta$) with $|\mathfrak{G}_{r',\delta}| \leq O(r'/\delta)^{r'^2}$ such that for every $b \in \mathfrak F_{r'}$ there is $\underline{b} \in \mathfrak G_{r',\delta}$ such that $\norm{b-\ul b}_\frob \leq \delta$.

    Let $\mathfrak P_{r'}$ be the set of polynomials with total degree at most $k$ with coefficients from $\mathfrak{G}_{r',\delta}$. Note that $|\mathfrak P_{r'}| \leq (|\mathfrak G_{r',\delta}|+1)^k(d+2\q)^k \leq O(r/\delta)^{kr^2}(d+2\q)^k$.
    Let $\ourpoly =\sum_{w \in \terms} a_w X^w$ now be some matrix polynomial with coefficients in $\mathfrak F_{r'}$. There exists a polynomial $\ul \ourpoly$ in $\mathfrak P_{r'}$ where $\ul\ourpoly = \sum_{w \in \terms} \ul{a_w} X^w$ where $\terms$ is the same as in $\ourpoly$, and for each $w \in \terms$, $\norm{a_w - \ul{a_w}}_\frob \leq \delta$. Since $|\terms| \leq k+1$, for any $n$-lift $\lift_n$, $\norms[\big]{A_{n,\bot}(\lift_n,\ourpoly) - A_{n,\bot}(\lift_n,\ul\ourpoly)}_\opnorm \leq \delta(k+1) \leq \eps/3$, and also $\norms[\big]{A_{\infty}(\ourpoly) - A_{\infty}(\ul\ourpoly)}_\opnorm \leq \eps/3$. Using \Cref{thm:perturbation}, we have $\dist_H(A_n(\lift_n,\ourpoly), A_n(\lift_n,\ul\ourpoly)) \leq \eps/3$. Therefore, it suffices to find an $N'$-lift $\lift_{N'}$ such that $\dist_H(A_{N'}(\lift_{N'},\ul\ourpoly), A_\infty(\ul\ourpoly)) \leq \eps/3$ for every $\ul\ourpoly \in \mathfrak{F}_{r'}$ for every $r' \leq r$.

    Now fix a polynomial $\ourpoly \in \mathfrak P_r$. Let $T$ be the interval $[-(k+1+\eps/3), (k+1+\eps/3)]$, and let $S \subset T$ be the spectrum of $A_\infty(\ourpoly)$. Let $f:\R \to \R$ be a smooth function such that $f(x) = -1$ for $x \in S$ and $f(x) = 1$ for $x \in T\setminus (S + [-\eps/3, \eps/3])$. Let $q'$ be a polynomial such that $|f-q'| < 0.1$ on $T$, and let $q$ be $q'$ scaled so that $\min q(x) = -2$ for $x \in T$.

    Using the spectral mapping theorem \Cref{thm:spectral-mapping}, we constructed $q$ such that $q(S) \subseteq [-2,0]$ so we know $\norm{q(A_\infty(\ourpoly)) - 1}_\opnorm = 1$.
    We apply \Cref{prop:factorization} on $q(A_\infty(\ourpoly)) - 1$ to get a factorization $\alpha_0 D_1 \alpha_1\dots D_m\alpha_m$ where each $\alpha_i$ is a rectangular scalar matrix, and each $D_i$ is the $\infty$-lift of some self-adjoint linear polynomial (which we view as a matrix bouquet) $\beta_i$. The factorization further satisfies $\norm{\alpha_i}_\opnorm \leq 1$ and $\norm{D_i}_\opnorm \leq 1$.

    We repeat this construction and factorization for every $\ul\ourpoly \in \mathfrak P_{r'}$ and for every $r' \leq r$, and let $\ov m$ be the maximum number of factors in the factorization, let $\ov r$ be the maximum dimension of the matrix bouquets, and let $\ov R$ be the maximum Frobenius norm of the coefficients of the matrix bouquets obtained in any of the factorizations.
    Let $\ov \eps > 0$ be small enough such that $(1+\ov \eps)^{\ov m} \leq 1.5$.
    We then apply \Cref{thm:linear-polys-norm} with the parameters $\ov r, \ov R$ and $\ov \eps$ to construct an $N'$-lift $\lift_{N'}$ with $N \leq N' \leq N + o(N)$ such that simultaneously for every matrix bouquet with $\max_{i \in \arcsz} \leq \ov R$ and at most $\ov r$-dimensional weights, $\norm{A_{N',\bot}(\lift_{N'},\calK)}_\opnorm \leq \norm{A_\infty(\calK)}_\opnorm + \ov\eps$.
    Using this $\lift_{N'}$ on the factorization of $q(A_\infty(p))+1$ (by substituting in $\lift_{N'}$ for $\lift_\infty$, we have $\norm{q(A_{N'}(\lift_{N'},p))-1}_\opnorm \leq 1.5$. Applying the spectral mapping theorem to $q(A_{N'}(\lift_{N'},p))$, since $q > 1$ on $T\setminus(S+[-\eps/3,\eps/3])$ we conclude that $\sigma(A_{N'}(\lift_{N'},p)) \subseteq \sigma(A_\infty(p)) + [-\eps/3,\eps/3]$.
    Furthermore, this holds simultaneously for every $p \in \mathfrak{P}_{r'}$ for any $r \leq r'$, hence we conclude as desired.
\end{proof}

Finally, by applying \Cref{thm:probabilistic-explicit-main} instead of \Cref{thm:weakly-explicit-main} in the proofs of \Cref{thm:linear-polys-norm,thm:polynomials-hausdorff}, we can obtain
\begin{theorem}\label{thm:probabilistic-explicit-hausdorff}
    In the setting of \Cref{thm:linear-polys-norm,thm:polynomials-hausdorff}, there is also an algorithm that takes as input a number~$N$ and a seed $s \in \{0,1\}^{O(\log N)}$, and in deterministic $\polylog(N)$ time outputs a binary circuit~$\mathfrak{C}$ that implements the adjacency list of a ``color-regular lift graph'' $G_{\lift_{N'}}$ (with $N \leq N' \leq N + o(N)$).  Furthermore, with high probability over the choice of a uniformly random~$s$ (namely, except with probability $2^{-\Theta(\sqrt{\log N})}$), the resulting $G_{\lift_{N'}}$ satisfies the conclusions of \Cref{thm:linear-polys-norm,thm:polynomials-hausdorff} respectively.
\end{theorem}

\section*{Acknowledgments}
The authors would like to thank Charles Bordenave and Beno{\^{i}}t Collins for several helpful clarifications regarding~\cite{BC19}.
The authors would also like to thank Sidhanth Mohanty for sharing the second construction in~\Cref{eg:additive}, and for pointing us to~\cite{VK19}.
Additionally, the authors would like to thank Sidhanth Mohanty and Pedro Paredes for many helpful discussions.
Finally, the authors thank Peter Sarnak for telling us of the results mentioned in \Cref{eg:sarnak}.

\bibliographystyle{alpha}
\bibliography{xinyu}

\newpage
\appendix
\section{Gallery}\label{sec:gallery}
The appendix contains pictures of a few interesting \mpl graphs.

\begin{figure}[ht]
    \centering
    \includegraphics[width=0.45\textwidth]{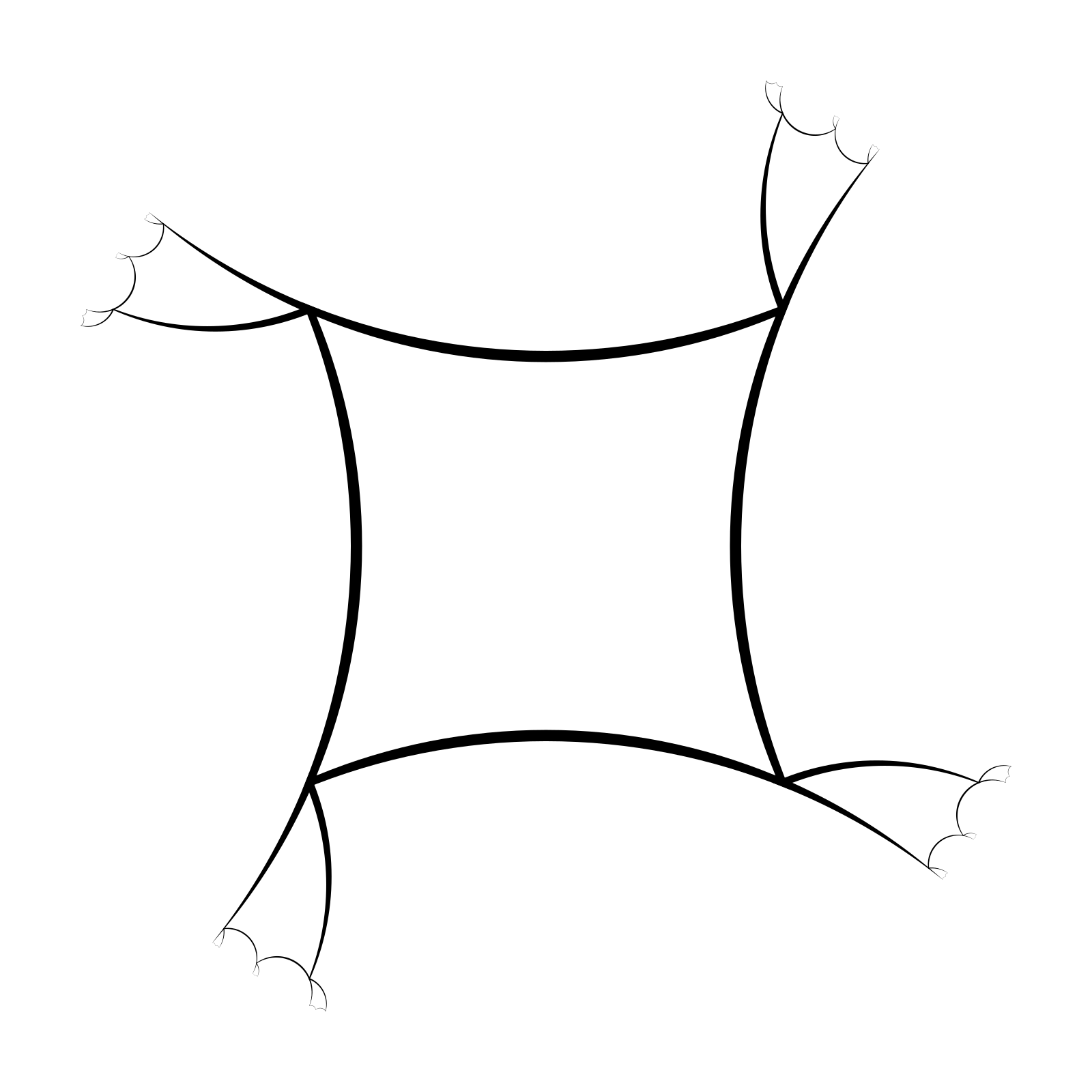}
      \qquad
    \includegraphics[width=0.45\textwidth]{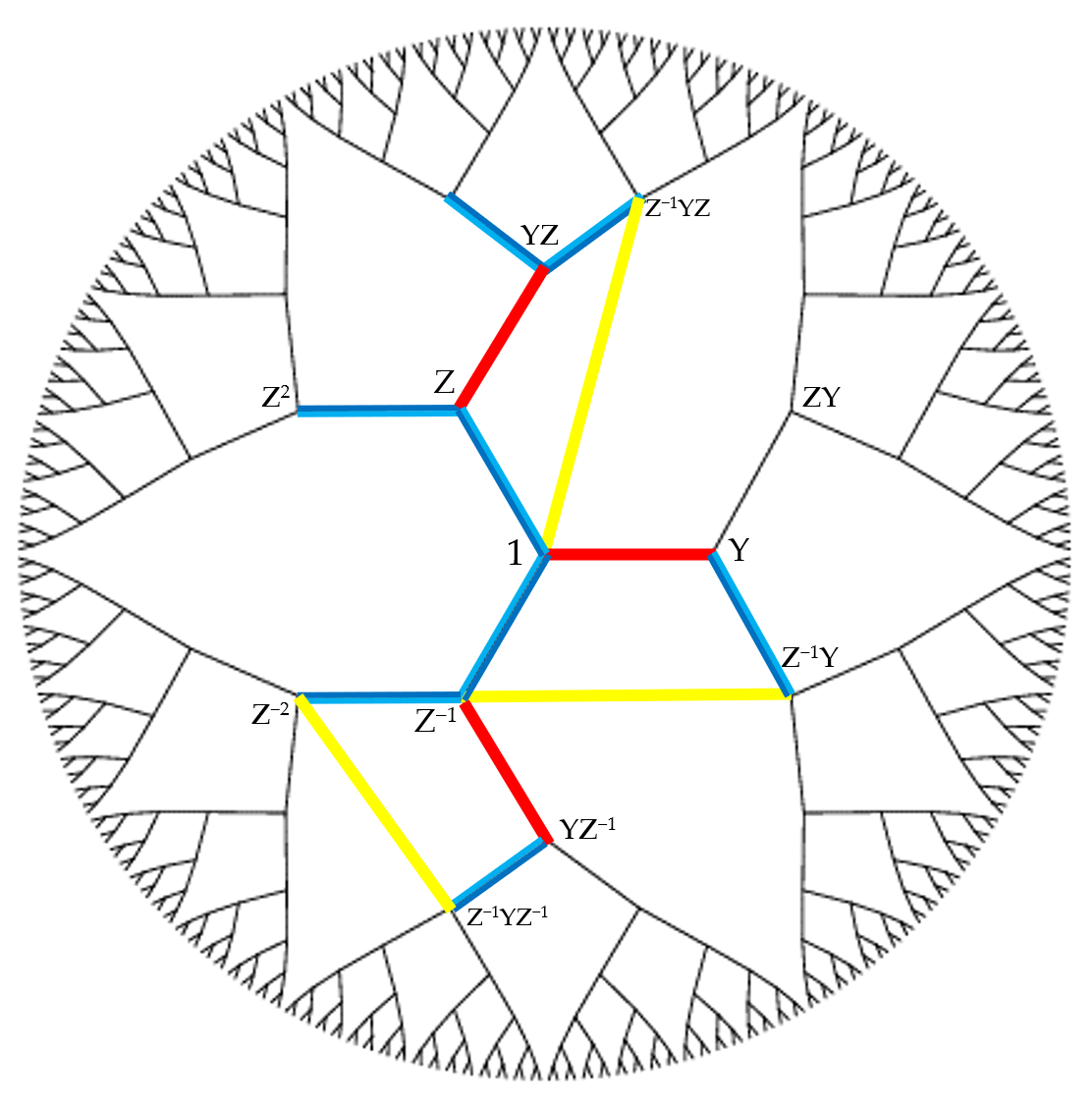}
    \caption[Gallery example 0]{\centering On the left, an illustration of $C_4 \star C_4$, the free product of two $4$-cycles.  On the right, a hint to how it arises as $\wt{\G}_\infty(Y + Z + Z^{-1} + Z^{-1} Y Z)$.
    }
    \label{fig:gallery-0}
\end{figure}

\begin{figure}[ht]
    \centering
    \includegraphics[width=0.9\textwidth]{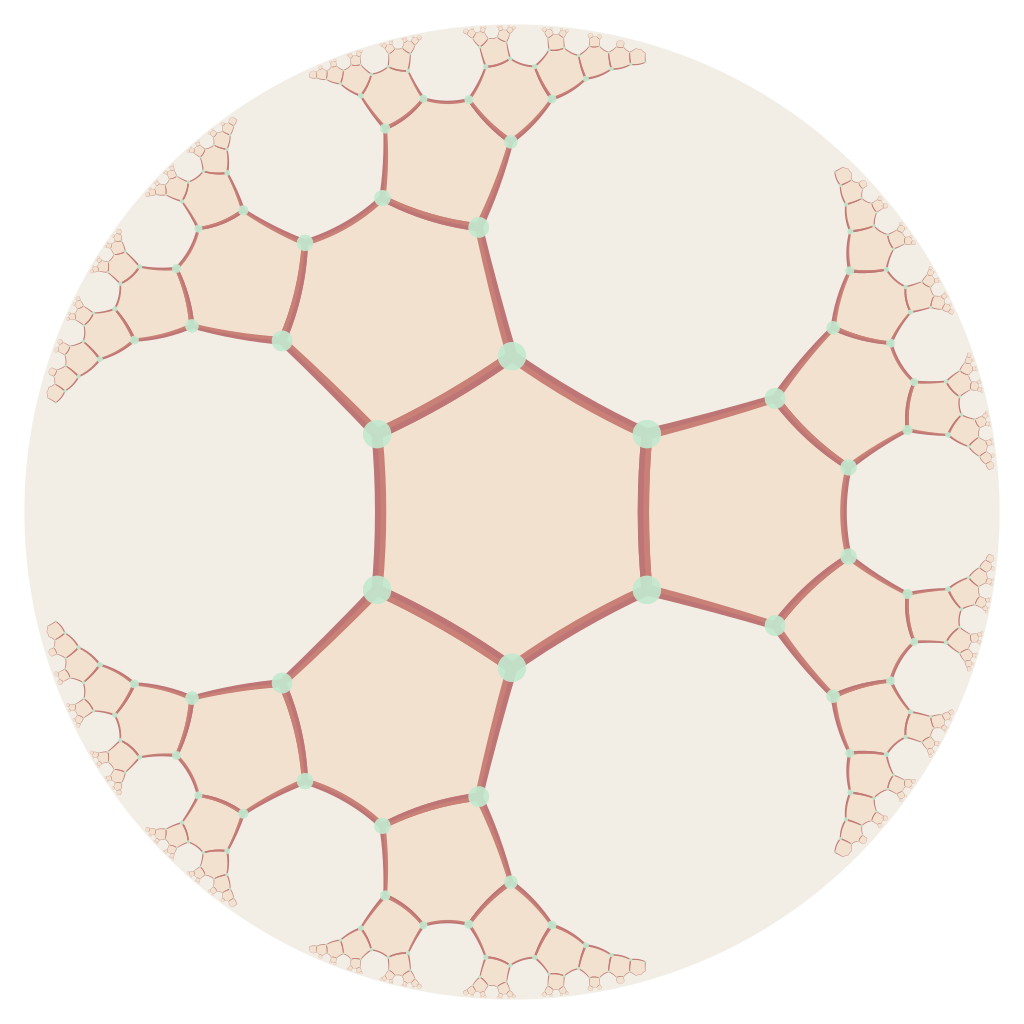}
    \caption[Gallery example 1]{\centering An illustration of
    $\wt\G_\infty(\ourpoly{})$ with
    \begin{equation*}
        \ourpoly =
        \begin{pmatrix}
        0 & 1 & 0 & 0 & 0 & 1\\
        1 & 0 & 1 & 0 & 0 & 0\\
        0 & 1 & 0 & 1 & 0 & 0\\
        0 & 0 & 1 & 0 & 1 & 0\\
        0 & 0 & 0 & 1 & 0 & 1\\
        1 & 0 & 0 & 0 & 1 & 0\\
        \end{pmatrix}
        +
        \ketbra{1}{4} Z_1 + \ketbra{3}{6}Z_2 + \ketbra{2}{5}Z_2Z_1 + \ketbra{4}{1} Z_1^* + \ketbra{6}{3}Z_2^* + \ketbra{5}{2}Z_1^*Z_2^*.
    \end{equation*}}
    \label{fig:gallery-1}
\end{figure}

\begin{figure}[ht]
    \centering
    \includegraphics[width=0.9\textwidth]{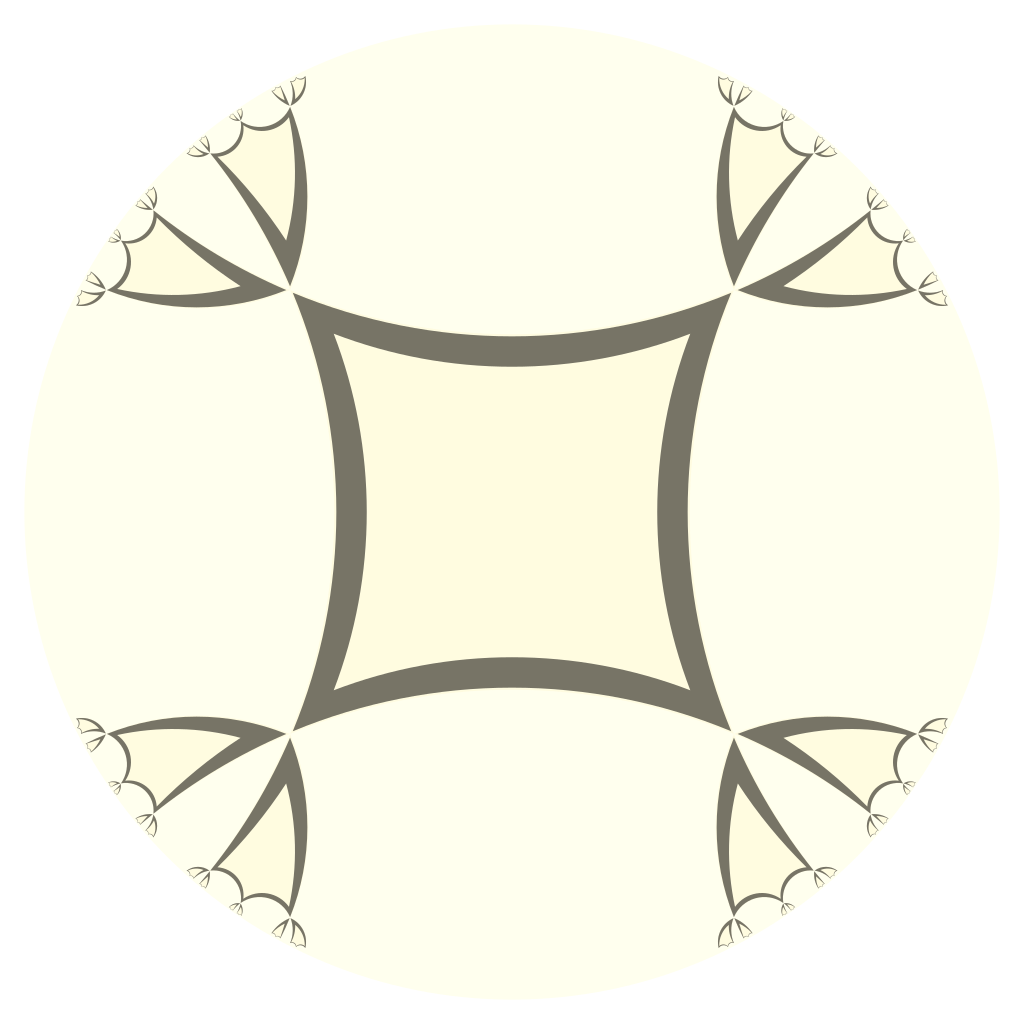}
    \caption[Gallery example 2]{\centering
    An illustration of $C_4 \star C_4 \star C_4$ as $\wt\G(\ourpoly)$ with $r = 4$ and
    \begin{align*}
        p =
        &\sum_{i=5}^{7} (\ketbra{i-3}{i-4} Z_i + \ketbra{i-4}{i-3}Z_i^*) +
        \ketbra{4}{1}Z_7Z_6Z_5 + \ketbra{1}{4}(Z_7Z_6Z_5)^* + \\
        &\sum_{i=1}^{4} (\ketbra{i}{i} Z_i + \ketbra{i}{i}Z_i^* + \ketbra{i}{i}Y_i + Z_{i}^*Y_iZ_{i}).
    \end{align*}}
    \label{fig:gallery-2}
\end{figure}
\end{document}